\documentclass[10pt,a4paper]{amsart}
\usepackage[T1]{fontenc}
\usepackage[utf8]{inputenc}
\usepackage[a4paper, left=2.5cm, right=2.5cm, top=3cm, bottom=3cm]{geometry}
\usepackage[foot]{amsaddr}
\usepackage{amssymb}
\usepackage{amsmath}
\usepackage{esint}
\usepackage{mathrsfs}
\usepackage{comment}
\usepackage[pdfdisplaydoctitle,colorlinks,breaklinks,urlcolor=blue,linkcolor=blue,citecolor=blue]{hyperref}
\usepackage{bookmark}

\usepackage{xcolor}
\usepackage{tcolorbox}
\usepackage{mathtools}
\usepackage{autonum}
\usepackage[english]{babel}
\usepackage{thmtools} % solves problem of shared counters in statements and autoref

\numberwithin{equation}{section}

\newtheorem{thm}{Theorem}[section]

\newtheorem{corollario}{Corollary}

\numberwithin{corollario}{section}

\newtheorem{lemma}{Lemma}

\numberwithin{lemma}{section}

\newtheorem{prop}{Proposition}

\numberwithin{prop}{section}

\theoremstyle{definition}
\newtheorem{defn}[thm]{Definition}

\theoremstyle{remark}
\newtheorem{oss}[thm]{Remark}
\newcommand{\vertiii}[1]{{\left\vert\kern-0.25ex\left\vert\kern-0.25ex\left\vert #1 \right\vert\kern-0.25ex\right\vert\kern-0.25ex\right\vert}}
\newcommand{\N}{\mathbb{N}}
\newcommand{\Z}{\mathbb{Z}}

\newcommand{\R}{\mathbb{R}}
\newcommand{\T}{\mathbb{T}}
\newcommand{\F}{\mathcal{F}}
\newcommand{\I}{\mathcal{I}}
\renewcommand{\P}{\mathbf{P}}
\newcommand{\C}{\mathscr{C}}
\newcommand{\dvg}{\mathord{{\rm div}}}
\newcommand{\Id}{\mathrm{Id}}
\newcommand{\ktimes}{k_i \otimes k_i}
\newcommand{\varz}{\overline{z}}
\newcommand{\varp}{\overline{p}}
\newcommand{\gammai}{\Gamma_{k_i}}
\newcommand{\lambdai}{\Lambda_{k_i}}
\newcommand{\AntiDiv}{\mathcal{R}}
\newcommand{\proj}{\mathbf{P}}
\DeclareMathOperator{\supp}{Supp}
\DeclareMathOperator{\Tr}{Tr}
\newcommand{\qq}{_{q+1}}
\newcommand{\qqq}{_{q+2}}
\newcommand{\SymOt}{\odot}
\newcommand{\SymOtNoTr}{\,\mathring{\odot}\,}
\newcommand{\odelta}{\overline{\delta}}
\newcommand{\oM}{\overline{M}}

\usepackage[obeyFinal,textwidth=2.7cm]{todonotes}

\begin{document}
\title[Dissipative solutions to randomly forced 3D Euler equations]{Dissipative solutions to randomly forced 3D Euler equations}
\author[U. Pappalettera]{Umberto Pappalettera}
  \address{Departement Mathematik und Informatik, Universit\"at Basel, Spiegelgasse 1,
4051 Basel, Switzerland}
  \email{\href{mailto:umberto.pappalettera@unibas.ch}{umberto.pappalettera@unibas.ch}}
\author[F. Triggiano]{Francesco Triggiano}
  \address{Scuola Normale Superiore, Piazza dei Cavalieri, 7, 56126 Pisa, Italia}
  \email{\href{mailto:francesco.triggiano@sns.it}{francesco.triggiano@sns.it}}

\date{\today}
\begin{abstract}
The purpose of this work is twofold. First, we construct probabilistically strong solutions to the three-dimensional Euler equations perturbed by additive noise that are $\mathbb{P}$-almost surely continuous in time, H\"older in space, and satisfy the local energy inequality up to an arbitrarily large stopping time.
Second, we prove several non-unique ergodicity results for the forced Euler equations with continuous-in-time external forcing.
The solutions we construct are genuinely random and, almost surely, strictly dissipative and not steady states.
\end{abstract}
\maketitle

\vspace{-1cm}

\setcounter{secnumdepth}{4}
\setcounter{tocdepth}{4}
\tableofcontents

\section{Introduction}
In the first part of this work, we consider the stochastic Euler equations on the three-dimensional torus $\T^3 := \R^3/2\pi\Z^3$ driven by additive noise
\begin{align} \label{eq:SE} \tag{SE}
\begin{cases}
d u 
+
\dvg (u\otimes u)\,dt
+
\nabla p \,dt
=
Q^{1/2} dW,
\\
\dvg(u) = 0.
\end{cases}
\end{align}
Euler equations govern the evolution of an ideal fluid and constitute one of the most important models in fluid dynamics.
The presence of a random external force in \eqref{eq:SE} can help represent many phenomena perturbing the dynamics that are neglected by the deterministic equation, such as endogenous microscopic thermal effects or exogenous influences of the environment on the fluid. 
First investigations on \eqref{eq:SE} date back to the works \cite{MiVa00,Ki09}, where local existence of smooth solutions is proved (see also \cite{BeFl99} for other results in dimension two).
Measure-valued solutions satisfying weak-strong uniqueness and a version of the energy inequality have been constructed in \cite{BrMo21}, building upon results by DiPerna and Majda \cite{DPMa87} for the deterministic Euler equations $Q=0$.
Whether solutions are unique or not remained an open problem until the work \cite{HoZhZh22} by Hofmanov\'a, Zhu, and Zhu, adapting for the first time the convex integration scheme developed by De Lellis and Székelyhidi Jr. in \cite{DLSz09} to the stochastic case, and obtaining non-uniqueness in law for \eqref{eq:SE}.
After \cite{HoZhZh22} made clear that convex integration techniques can be applied to stochastic equations (let us also mention \cite{ChFeFl21} in this regard), a sequence of papers has appeared \cite{HoZhZh25,LuZh24,KiKo24,LuLuZh25} extending the works \cite{DLS13,DLS14,BuDLIsSz15,BuDLIsSz16,Is18,BuDLSzVi19} on the deterministic Euler equations.
In particular, for every $\vartheta \in (0,1/3)$, L\"u, L\"u, and Zhu proved in \cite{LuLuZh25} existence and non-uniqueness of solutions of \eqref{eq:SE} with regularity $u \in C([0,\infty),C^\vartheta(\T^3,\R^3)$ almost surely, which are strictly dissipative up to an arbitrarily large stopping time $\mathfrak{t}$ in the sense that $u$ satisfies almost surely the energy inequality
\begin{align} \label{eq:EI}
    \frac12 \| u(t\wedge\mathfrak{t}) \|_{L^2_x}^2
    \leq
    \frac12 \| u(s\wedge\mathfrak{t}) \|_{L^2_x}^2
    +
    \int_{s \wedge \mathfrak{t}}^{t \wedge \mathfrak{t}}
    \langle u(r), Q^{1/2}dW_r \rangle + 
    \frac12 \mbox{Tr}(Q) (t \wedge \mathfrak{t}-s \wedge \mathfrak{t}),
\end{align}
for every $s <t$, and the inequality is strict if $s<\mathfrak{t}$.
In the present work we investigate solutions of \eqref{eq:SE} that satisfy almost surely the \emph{local energy inequality} up to time $\mathfrak{t}$:
\begin{align} \label{eq:LEI} 
\frac12 d |u|^2
+
\dvg\left( \left( \frac12 |u|^2 + p \right) u \right)dt
-
\mathfrak{q}\,dt
-
u \cdot Q^{1/2} dW
\leq 0,\tag{LEI}
\end{align}
where $\mathfrak{q} := \frac12 \sum_{k} |Q^{1/2}e_k|^2$ for a CONS $\{e_k\}_{k \in \N}$ of the space of $L^2_x$ solenoidal vector fields, $d|u|^2$ and $dW$ are It\=o stochastic differentials, and the equation is meant as an almost sure identity in the sense of distributions.
Notice that \eqref{eq:EI} can be obtained from \eqref{eq:LEI} by space integration.

The importance of a local form of energy inequality towards partial regularity results for \emph{suitable} weak solutions of the Navier-Stokes equations is known since the works of Scheffer \cite{Sc77} and Caffarelli, Kohn, and Nirenberg \cite{CaKoNi82}. Similar results are available for the stochastic Navier-Stokes equations, see \cite{Ro10,FlRo02}.
The local energy balance \eqref{eq:LEI} is a natural condition to impose also on the stochastic Euler equations \eqref{eq:SE}, since any physically relevant $L^3_{t,x}$ zero-viscosity limit of suitable solutions of the stochastic Navier-Stokes equations must inherit this property.
It describes the local change in kinetic energy due to spatial roughness of $u$: indeed, the equality in \eqref{eq:LEI} holds for smooth solutions $u$, and in general the inequality corresponds to no local creation of kinetic energy due to irregularities of $u$.
We refer to the celebrated work of Duchon and Robert \cite{DuRo00} for more details.

In the deterministic setting $Q=0$, existence and non-uniqueness of solutions to the system \eqref{eq:SE}-\eqref{eq:LEI} have been shown in the series of papers \cite{DLSz10,Is22,DLK23,GiKwNo}.
In the present work we provide similar results in the stochastic case $Q \neq 0$.
More specifically, we will prove the following: 
\begin{thm}\label{thm:main1}
    Let $(\Omega, \mathbb{F},\{\mathbb{F}_t\}_{t \geq 0}, \mathbb{P})$ be a stochastic basis with right-continuous and complete filtration supporting a sufficiently smooth-in-space $Q$-Wiener process $Q^{1/2}W$ with zero space average and null divergence. 
    Let $T>0$, $\varkappa\in (0,1)$, $\alpha \in (0,1/7)$ and $E\in C^1_{loc}([0,\infty),\R)$ be given.
    Then there exist two $\{\mathbb{F}_t\}_{t \geq 0}$-progressively measurable stochastic processes 
    \begin{align}
        u : \Omega \to C([0,\infty), C^\alpha(\T^3,\R^3)),
        \qquad
        p: \Omega \to C([0,\infty), C^{2\alpha}(\T^3,\R)),
        \qquad
        \mathbb{P}\mbox{-almost surely,}
    \end{align}
    such that $(u,p)$ is $\mathbb{P}$-almost surely a solution of \eqref{eq:SE} on the time interval $[0,\infty)$ in analytically weak sense, and a stopping time $\mathfrak{t}$ satisfying $\mathbb{P}\{\mathfrak{t} \geq T\} \geq \varkappa$ such that the local energy equality
      \begin{align} \label{eq:LEE} \tag{LEE}
      \frac12 d |u|^2
+
\dvg\left( \left( \frac12 |u|^2 + p \right) u \right)dt
-
\mathfrak{q}\,dt
-
u \cdot Q^{1/2} dW
= -E'dt
  \end{align}
holds $\mathbb{P}$-almost surely in the sense of distributions on the time interval $[0,\mathfrak{t}]$. In particular, \eqref{eq:LEI} holds strictly if $E'>0$.
Finally, non-uniqueness in law for $(u,p)$ holds within this class.
\end{thm}

In the second part of the paper, we move to addressing the problem of \emph{ergodicity} associated with Euler equations.
The so-called ergodic hypothesis postulates the existence of an equilibrium measure $\mu$ over some phase space $H$ of velocity fields $u$, such that the time average of observables $F(u_t)$ over large time intervals tends to $\int_{H} F d\mu$. 
This ergodic assumption is at the basis of several foundational studies of turbulence (see \cite[Chapter IV]{FMRT01} for a thorough discussion) but its validity lacks a rigorous mathematical justification. 
Therefore, in order to gain additional insight into the behaviour of fluids in a turbulent regime, it is of the utmost importance to consider ergodic measures $\mu$ concentrated on solutions to \eqref{eq:SE}.
We point out that the presence of an external forcing introducing energy into the system (on average) is somehow necessary if one is interested in isolating statistically stationary solutions that are strictly dissipative.

Ideally, $\mu$ should be a suitable vanishing-viscosity limit of ergodic measures $\mu^\nu$ associated with the stochastically forced Navier-Stokes equations, cf. also the discussion at page 472 in \cite{Ku04}.
However, the recent work \cite{HoZhZh25} has shown that \eqref{eq:SE} admits non-unique ergodic invariant measures; Moreover, \cite[Theorem 1.5]{HoZhZh25} states that there exists $\vartheta>0$ with the property that \emph{every} statistically stationary solution of \eqref{eq:SE} with finite moments locally in $C_t H^\vartheta_x \cap C^\vartheta L^2_x$ is a vanishing-viscosity limit in law of statistically stationary solutions of stochastically forced Navier-Stokes equations. Therefore, vanishing viscosity alone is not a good selection criterion among statistically stationary solutions of \eqref{eq:SE} (cf. also \cite[Theorem 1.3]{BuVi19} for a deterministic version of this statement).

Nonetheless, extending \autoref{thm:main1} to construct non-unique ergodic measures concentrated on solutions to \eqref{eq:SE} satisfying the local energy inequality \eqref{eq:LEI} is not straightforward.
Let us only mention that one major issue is the unboundedness of the external forcing $Q^{1/2}dW$.
On the other hand, we can construct non-unique ergodic measures $\mu$ on the path space 
\begin{align}
    \mathscr{X}:=\{ (u,p,g) \in C(\R, C^\alpha(\T^3,\R^3)) \times C(\R, C^{2\alpha}(\T^3,\R)) \times C(\R, C^{2\alpha}(\T^3,\R^3)), \, \dvg(u)=\dvg(g)=0\},
\end{align}
ergodicity being meant with respect to the shift operator on the path space\footnote{Namely, $\mu(A) \in \{0,1\}$ for every shift invariant Borel set $S_tA =A\subset \mathscr{X}$, where $S_t(u,p,g)(s) := (u,p,g)(t+s)$.
}, that are concentrated on strictly dissipative solutions of the forced Euler equations
\begin{align} \label{eq:FE} \tag{FE}
\begin{cases}
\partial_t u 
+
\dvg (u\otimes u) 
+
\nabla p  
=
g,
\\
\dvg(u) = 0,
\\
\frac12 \partial_t |u|^2
+
\dvg\left( \left( \frac12 |u|^2 + p \right) u \right)dt
-
u \cdot g
< 0,
\end{cases}
\end{align}
with some external forcing $g$ having a given law, and such that $u$ is genuinely random and almost surely not a steady-state.  
More precisely, we have the following:
\begin{thm} \label{thm:main2}
 There exist two distinct ergodic measures $\mu_1,\mu_2$ on the path space $\mathscr{X}$, concentrated on solutions of \eqref{eq:FE}, such that for $i \in \{1,2\}$ the process $u$ is not a steady state  
 \begin{align} \label{eq:no-steady}
     \sup_{s,t \in \R}   \| u(t)-u(s)\|_{C_x} > 0,
     \quad
     \mu_i\,\mbox{-\,almost surely},
 \end{align}
and is genuinely random
\begin{align} \label{eq:variance>0}
     \inf_{t \in \R} \mathbb{E}^{\mu_i} \left[ \| u(t)-\mathbb{E}^{\mu_i}[u(t)] \|^2_{C_x}\right]
        > 0.
\end{align}
 Moreover, we can construct $\mu_1,\mu_2$ such that the laws of $g$ under $\mu_1$ and $\mu_2$ coincide, but the laws of $u$ under $\mu_1$ and $\mu_2$ differ. In particular, for strictly dissipative ergodic measures associated with \eqref{eq:FE}-\eqref{eq:no-steady}-\eqref{eq:variance>0} one can not uniquely define a map $\mathscr{L}_{\mu}(g) \mapsto \mathscr{L}_{\mu}(u)$.
\end{thm}

\begin{oss}
The validity of \eqref{eq:no-steady}-\eqref{eq:variance>0} is a well-known open problem in limiting procedures à la Kuksin \cite{Ku04}, see for example the discussion in \cite[Section 2.3]{BeCZGH16}.
\end{oss}

What distinguishes the two ergodic measures $\mu_1,\mu_2$ in the previous theorem is the typical size of the process $u$, which also entails that the amount of kinetic energy dissipated under $\mu_1$ and $\mu_2$ differs. 
If we don't prescribe the law of the external forcing $g$ then we can fix the right-hand side in the third line of \eqref{eq:FE}. In addition, similarly to \cite[Theorems 6.1 - 6.3]{HoZhZh25} we can almost arbitrarily prescribe the law of $u$ up to a small error.  
\begin{thm} \label{thm:main3}
For every $\iota \in (0,1/8)$, constant $E' >0$ sufficiently small and every shift-invariant law $\nu$ on the space $\mathscr{Z} := \{Z \in C(\R,C^1(\T^3,\R^3), \, \dvg(Z)=0, \, \inf_{t \in \R}\|Z(t+1) - Z(t)\|_{C_x} > \frac12  \|Z\|_{C_{t,x}} \}$ satisfying the assumptions of \autoref{ass:Z}, there exists an ergodic measure ${\mu}$ on the path space $\mathscr{X} \times \mathscr{Z}$, concentrated on solutions of \eqref{eq:FE} with dissipation equal to $-E'$, such that \eqref{eq:no-steady} and \eqref{eq:variance>0} hold, $\mathscr{L}_\mu(Z) = \nu$, and
\begin{align}
\sup_{t \in \R} \| u(t) - Z(t) \|_{C_x} < \iota \inf_{t \in \R} \| Z(t) \|_{C_x},
\quad
\mu\,\mbox{-\,almost surely}.
\end{align}
\end{thm}

The paper is organized as follows.
In \autoref{sec:prelimi} we give some preliminaries and introduce the convex integration schemes that we will use to prove our theorems.
Iterative estimates for \eqref{eq:SE} are contained in \autoref{sec:velocity} (for the velocity field),
\autoref{sec:new_reynolds} (for the Reynolds stress and pressure), and in 
\autoref{sec:current} (for the current).
These results are also at the basis of the iteration lemma needed to construct solutions of \eqref{eq:FE}, and are complemented by the results in \autoref{sec:forced} that are specific to this case.
Some auxiliary results are collected in the Appendix.

\section*{Acknowledgements}
 The authors are sincerely grateful to Eliseo Luongo and Marco Romito for many interesting discussions.
UP has received funding from the European Research Council (ERC) under the European Union’s Horizon 2020 research and innovation programme (grant agreement No. 949981) and the Swiss National Science Foundation under the SNSF Ambizione grant No. 233216.

FT is supported by the Istituto Nazionale di
Alta Matematica (INdAM) through the project GNAMPA 2025 “Modelli stocastici in Fluidodinamica
e Turbolenza” -CUP E5324001950001.
FT acknowledges the hospitality of Universit\"at Bielefeld, where part of this work has been done.

\section{Preliminaries and main iterative propositions} \label{sec:prelimi}

\subsection{Notation}
Let $\N$ denote the set of natural numbers, including $0$ in our convention, and let $n,m \in \N$. For a multi-index $\mathbf{r} = (r_1,r_2,\dots,r_n) \in \N^n$ of order $|\mathbf{r}|:= r_1 + r_2 + \dots+ r_n$ and a smooth function $f:\mathcal{O} \to \R^m$ defined on a smooth domain $\mathcal{O} \subset \R^n$ without boundary, we denote $D^{\mathbf{r}}f := \partial_{1}^{r_1} \partial_2^{r_2} \dots \partial_n^{r_n} f$.
With a slight abuse of notation, in the following we identify functions defined on the torus $\T^3$ as periodic functions defined on $\mathcal{O}=\R^3$.

H\"older spaces $C^\gamma(\mathcal{O},\R^m)$, $\gamma \geq 0$ are defined according to the following notion of norms.
Let $N\in \N$ and $\delta \in (0,1)$, and let us introduce the following H\"older  seminorms:
\[[f]_N := \max\limits_{|\mathbf{r}|=N}\|D^{\mathbf{r}}f\|_{C^0},
\quad 
[f]_{N+\delta}:=\max\limits_{|\mathbf{r}|=N}\sup\limits_{x\neq y}\frac{|D^{\mathbf{r}}f(x)-D^{\mathbf{r}}f(y)|}{|x-y|^\delta}.\]
Then, the H\"older norms are defined as
\[\|f\|_N:=\max_{j=0,\dots,N}\{[f]_j\},
\quad 
\|f\|_{N+\delta}:= \max\{\|f\|_{N},[f]_{N+\delta}\}.\]
Furthermore, in order to simplify the statement of many propositions here contained, we adopt the notational convention $\| f \|_M \equiv 0$ whenever the index $M<0$. 
For $\gamma \geq 0$ we denote $C^\gamma_x := C^\gamma(\mathcal{O},\R^m)$, with $\mathcal{O}$ and $m$ depending on the context, and similarly for other functional spaces.
We use the symbol $c^\gamma_x$ to denote the closure of smooth functions with respect to the topology of $C^\gamma_x$.

\subsection{Preparation of the convex integration scheme}\label{subsec:prep}
The proofs of \autoref{thm:main1}, \autoref{thm:main2} and \autoref{thm:main3} rely on a convex integration scheme, mostly inspired by the works \cite{Is22,DLK23,LuLuZh25}.
Let us introduce various parameters that will allow to estimate all the objects appearing in the upcoming sections. Define for $q \in \N$: 
\[\lambda_q:= \lfloor a^{b^q} \rfloor,\quad \odelta_q:=\lambda_q^{-2\alpha},\quad \delta_q:=\begin{cases}
    \lambda_0^{\gamma}, \, &q=0,\\
    \lambda_0^{\gamma}\odelta_q\lambda_1^{2\alpha},\,&q\geq 1,
\end{cases}\]
where $a>2$ will be very large, $b>1$ will be chosen close to 1, $\gamma := (b-1)^2$, and $\alpha\in (0,1/7)$ indicates the H\"older regularity of the solution identified by the iterative procedure.

\subsection{ Stochastic Euler Equations: Induction scheme}\label{ssec:SE_InductionScheme} 
In this subsection, we intend to establish \autoref{thm:main1} by making use of an Iteration Lemma (see \autoref{IterLemma}), whose proof will be given in the following sections.

It is convenient to rewrite \eqref{eq:SE}-\eqref{eq:LEE} without the local martingale terms $Q^{1/2} dW$ and $u \cdot Q^{1/2} dW$. To attain this, we reformulate the problem using the so-called Da Prato-Debussche trick, namely we decompose solutions of \eqref{eq:SE} as $u=v+z$, where $z:=Q^{1/2}W$ and $v$ solves the random PDE
\begin{align}\label{eq:EulerV}
\begin{cases}
\partial_t v 
+
\dvg\big((v+z) \otimes (v+z)\big)
+
\nabla p 
=0,
\\
\dvg(v) = 0.
\end{cases}
\end{align}
A local energy inequality can be derived for $v$ as follows. First, rewrite the stochastic differential
\begin{align}
   \frac12  d |v|^2 = \frac12  d |u-z|^2 = \frac12  d|u|^2 + \frac12 d|z|^2 - d ( u \cdot z ).
\end{align} 
Since by assumption $z$ is spatially smooth enough, It\=o Formula gives 
\begin{align}
    \frac12 d|z|^2 = z \cdot dz + \mathfrak{q} \,dt,
    \quad
    d(u\cdot z) = du \cdot z + u \cdot dz + 2\mathfrak{q}\,dt,
\end{align}
so that the local energy equality \eqref{eq:LEE} is equivalent to
\begin{align}
    \partial_t \frac{|v|^2}{2}+v\cdot \nabla p+v \cdot \dvg\big((v+z)\otimes (v+z)\big)=-E'.
\end{align}
We refer to \cite{Ro10} for additional details.
Therefore, solutions $(u,p)$ of \eqref{eq:SE}-\eqref{eq:LEE} are equivalent to solutions of the following system:
\begin{align}\label{eq:EulerV2}
\begin{cases}
\partial_t v+\dvg\big((v+z)\otimes (v+z)\big)+\nabla p=0,
\\
\dvg(v) = 0,
\\
\partial_t\frac{|v|^2}{2}+v\cdot \nabla p+v \cdot\dvg\big((v+z)\otimes (v+z)\big)=-E'.
\end{cases}
\end{align}

 We aim at building solutions of \eqref{eq:SE}-\eqref{eq:LEE} by applying the convex integration scheme to \eqref{eq:EulerV2}. In particular, following \cite{Is22}, we obtain a solution of \eqref{eq:EulerV2} as a limit of smooth solutions of approximating systems, which include a Reynolds stress tensor $R$ in the momentum equation and a flux current $\phi$ in the local energy inequality, and boil down to the \eqref{eq:EulerV2} when $R=0$ and $\phi=0$.
In this sense, the size of the \emph{errors} $R$ and $\phi$ give an indication of ``how far'' a solution of the approximating system is from being an exact solution.

The approximating system to \eqref{eq:EulerV2} is slightly more involved than the one introduced at \cite{DLK23},  
since we need to keep track of the stochastic process $z$ affecting the dynamics.
Whenever we work on Stochastic Euler equations, we shall deal with \emph{probabilistically strong} solutions, i.e. our objects will be defined on the same stochastic basis $(\Omega, \mathbb{F},\{\mathbb{F}_t\}_{t \in \R}, \mathbb{P})$  supporting the Wiener process $z$. 

\begin{defn}\label{defn:Euler-Reynolds}
    A tuple of $\{\mathbb{F}_t\}$-progressively measurable smooth tensors $(v,p,R,\phi,z)$ is a dissipative Euler-Reynolds flow, with energy loss $E \in C^1_b$, if it satisfies the following system
    \begin{align} \label{eq:ApproxEulerV}
    \begin{cases}
    \partial_tv+\dvg\big((v+z)\otimes(v+z)\big)+\nabla p=\dvg(R),
    \\
    \dvg(v) = 0,
    \\
     \partial_t\frac{|v|^2}{2}+v\cdot \nabla p+v \cdot\dvg\big((v+z)\otimes (v+z)\big)=-E'
     +\frac{1}{2}D_t(\Tr(R))+\dvg(Rv)+R:\nabla z^T+\dvg (\phi),
    \end{cases}
    \end{align}
     where $D_t:=\partial_t+(v+z)\cdot \nabla$ denotes the advective derivative.
\end{defn}

We point out that our approximating Euler-Reynolds flow will solve \eqref{eq:ApproxEulerV} (in an analytically strong sense) also for negative times $t < 0$, as we work with one-sided time mollifications in order to preserve adaptedness to the filtration $\{\mathbb{F}_t\}_{t \in \R}$. We shall impose zero external noise for negative times, so that the process $(v,p,R,\phi,z)$ will be $\mathbb{F}_0$-adapted for negative times (and thus progressively measurable with respect to the filtration $\{\mathbb{F}_t\}_{t \in \R}$).

As usual in convex integration schemes, we set up an iterative procedure to modify a given dissipative Euler-Reynolds flow $(v_q,p_q,R_q,\phi_q,z_q)$, $q \in \N$, in such a way to identify a new dissipative Euler-Reynolds flow $(v\qq,p\qq,R\qq,\phi\qq,z\qq)$ whose errors satisfy $\|R\qq\|_0\ll\|R_q\|_0$ and $\|\phi\qq\|_0\ll\|\phi_q\|_0$.

Before reporting the Iteration Lemma, let us introduce various assumptions on $z$ and  define $z_q$.
Let $N_1:=n_0+5$, where the parameter $n_0\in \N$ will be chosen in the following sections and will depend on the value $b$. In particular, $n_0$ gets larger as $b$ is closer to $1$.
We assume that the $Q$-Wiener process $z$ takes values $\mathbb{P}$-almost surely in the Sobolev space $H^{2+N_1}(\T^3,\R^3)$ with zero mean and divergence. 

Let us also introduce $0<\delta \ll 1$ sufficiently small such that 
\[\frac{1/2+\delta}{1/2-\delta}<\frac{1-4\alpha}{3\alpha}\wedge \frac43,\]
and $L>1$ such that the following stopping time
\begin{equation}\label{stoppingtime}
    \mathfrak{t}:=\inf\left\{t\geq 0: 
    \|z\|_{C^{1/2-\delta}([0,t],C_x^{N_1})}\geq L\right\}\wedge 2T
\end{equation}
is larger than $T$ with probability greater than $\varkappa \in (0,1)$ given by the statement of \autoref{thm:main1}.
Notice that is always possible to find such $L$ since $H^{2+N_1}_x \subset C^{N_1}_x$ by Sobolev embedding. 
In this way, we have a $\mathbb{P}$-almost sure control over the size of $\| z_t \|_{N_1}$ on the time interval $t \in [0,\mathfrak{t}]$.
However, since $z$ is temporally rough, we build $z_q$ by introducing a time mollification
\begin{equation}\label{e:SmoothNoise}
    z_q(t): =\int_{\R}r_{i_q}(t-s)z(s)\,ds,
    \quad
    t \in \R
\end{equation}
where $z(-s):=z(0)$ for all $s>0$, $r$ is a smooth time mollifier with support contained in $[0,1]$, and the parameter $i_q$ is defined as 
\begin{equation}
i_q^{1/2-\delta} :=\begin{cases}
    \lambda_0^{7\alpha/2}\lambda_1^{-7\alpha/2},\, &q=0,\\
    \delta_q^{-2}\delta_{q+1}^{7/2},\, &q\geq 1.
\end{cases}
\end{equation}
The choice of the asymmetric mollifier guarantees that $z_q$ remains adapted to the filtration $\{ \mathbb{F}_t\}_{t\in \R}$, extended to negative times identically equal to $\mathbb{F}_0$. Hereafter, we will always adopt this convention when referring to the filtration $\{ \mathbb{F}_t\}_{t\in \R}$.

\begin{lemma}
\label{IterLemma}
Let $\alpha \in (0,1/7)$ and $E\in C^1_b((-\infty,2T],\R)$. Then there exist parameters $a,b$ as above, $n_0\in \N$ depending only on $b$, and $C_v, \oM \in (1,\infty)$ such that the following holds.

    Assume that $(v_q,p_q,R_q,\phi_q,z_q)$, $q \in \N$, is a $\{\mathbb{F}_t\}$-progressively measurable dissipative Euler-Reynolds flow solving \eqref{eq:ApproxEulerV} on the time interval $(-\infty,\mathfrak{t}]$ and such that $\mathbb{P}$-almost surely the following estimates hold for all $t \leq \mathfrak{t}$:
   \begin{align}
        \label{s:V} \|v_q(t)\|_{N}
        &\le C_v \oM\lambda_q^N\delta_q^{1/2}, 
        &&
        && 
        N \in \{1,2,\dots,n_0+3\},
        \\
        \label{s:P} \|p_q(t)\|_{N}&\le \overline{M} \lambda_q^N\delta_q,
        &&
        \|D_{t,q}p_q(t)\|_{N-2}\le \overline{M}\lambda_q^{N-1}\delta_q^{3/2}, 
        &&  
        N \in \{1,2,\dots,n_0+3\},
        \\
        \label{s:R} \|R_q(t)\|_{N}&\le \oM\lambda_q^{N-\gamma}\delta_{q+1},
        &&
        \|D_{t,q}R_q(t)\|_{N-2}\le \oM \lambda_q^{N-1-\gamma}\delta\qq\delta_q^{1/2}, 
        &&  
        N \in \{0,1,\dots,n_0+3\},
        \\
        \label{s:Phi} \|\phi_q(t)\|_{N}
        &\le \oM\lambda_q^{N-3\gamma/2}\delta_{q+1}^{\frac{3}{2}},
        &&
        \|D_{t,q}\phi_q(t)\|_{N-1}\le \oM \lambda_q^{N-3\gamma/2}\delta_{q+1}^{\frac{3}{2}}\delta_q^{1/2}, 
        &&   
        N \in \{0,1,2\},
   \end{align}
where $\gamma:=(b-1)^2$ and $D_{t,q}:=\partial_t+(v_q+z_q)\cdot \nabla$.

   Then, there exists a new tuple $(v\qq,p\qq,R\qq,\phi\qq,z\qq)$ of $\{\mathbb{F}_t\}$-progressively measurable tensors which solves \eqref{eq:ApproxEulerV}, such that the estimates \eqref{s:V}, \eqref{s:P}, \eqref{s:R}, \eqref{s:Phi} hold with $q$ replaced by $q+1$ and, in addition, $\mathbb{P}$-almost surely for every $t \leq \mathfrak{t}$:
   \begin{align}
   \label{s:distance_v}
       \|v_{q+1}(t)-v_q(t)\|_{0}+ \frac{1}{\lambda\qq}\|v_{q+1}(t)-v_q(t)\|_{1}
       &\le 
       C_v\oM\delta_{q+1}^{1/2}\lambda_q^{-\gamma/2},
   \\ \label{s:distance_p}
       \|p\qq(t)-p_q(t)\|_0+\frac{1}{\lambda\qq}\|p\qq(t)-p_q(t)\|_1
       &\le 
       \overline{M}\delta\qq.
   \end{align}
\end{lemma}
We point out that, by definition of $\| \cdot \|_{-1}$ and $\|\cdot\|_{-2}$, estimates \eqref{s:P}, \eqref{s:R} and \eqref{s:Phi} on the advective derivatives are void in these specific cases. Notice also that $\mathfrak{t} \leq 2T$ by definition, so that the system \eqref{eq:ApproxEulerV} is meaningful on the time interval $(-\infty,\mathfrak{t}]$. 

Since $\delta_q \to 0$ exponentially fast, the estimates above imply $\|R_q\|_0,\|\phi_q\|_0 \to 0$ as $q \to \infty$.
Moreover, \eqref{s:distance_v} and \eqref{s:distance_p} on the increments and interpolation guarantee convergence $v_q \to v$ and $p_q \to p$ in a space of continuous-in-time, H\"older-in-space functions.
In fact, given the previous lemma, we can prove \autoref{thm:main1} arguing as follows.

\begin{proof}[Proof of \autoref{thm:main1}]
\textit{Step 1: Initialization}.
Let $E \in C^1_{loc}([0,\infty),\R)$ be as in the statement of the theorem. Then we can find an extension of $E$ to negative times, still denoted $E$ with a little abuse of notation, such that $C^1_b((-\infty,2T],\R)$. 
Consider the following dissipative Euler-Reynolds flow at step $q=0$: 
\begin{equation}
    (z_0,v_0,p_0,R_0,\phi_0):=\left(z_0,0,\frac{|z_0|^2}{3},z_0\mathring{\otimes}z_0+\frac{2E}{3}\Id, -\frac{|z_0|^2}{2}z_0\right),
\end{equation}
where $\mathring{\otimes}$ denotes the traceless tensor product.
Using $\dvg (z_0)=0$ and the algebraic identity $z_0\otimes z_0:\nabla z_0^T = \dvg(z_0 |z_0|^2/2)$, it is immediate to verify that $(z_0,v_0,p_0,R_0,\phi_0)$ solves the system \eqref{eq:ApproxEulerV} for every $t \in \R$.
Moreover, on the time interval $t \in (-\infty,\mathfrak{t}]$, the following estimates on the Reynolds stress are valid for every $N \leq n_0+3$:
\begin{align}
    \|R_0(t)\|_N
    &\leq 
    |E(t)|  + \sum_{k=0}^N \binom{N}{k}\|z_0\|_k \|z_0\|_{N-k}
    \\
    &\leq
    |E(t)| + (N+1)2^N L^2
    \le \lambda_0^{N-\gamma}(L^2+|E(t)|) \lambda_0^\gamma,
    \\
    \|D_{t,0}R_0(t)\|_{N-2} 
    &\leq 
    |E'(t)|
    +
    (N-1)2^{N-2} i_0^{-1/2-\delta}L^2 
    + 
    \sum_{k + h = 0}^{N-2} \binom{N-2}{k,h} \|z_0\|_k \|z_0\|_{h} \|z_0\|_{N-1-k-h}
    \\
    &\le
    |E'(t)|
    +
    (N-1)2^{N-2} i_0^{-1/2-\delta}L^2
    +
    (N-1)^2 3^{N-2} L^3
    \leq 
    \lambda_0^{N-1-\gamma} (L^3+|E'(t)|) \lambda_0^{3\gamma/2} ,
    \end{align}
 where we have used that $i_0^{-1/2-\delta}\le \lambda_0^{1+\gamma/2}$ and we have taken $a$ sufficiently large,
 as well as the bounds on the current
    \begin{align}
    \|\phi_0(t)\|_N
    &\leq
    (N+1)^2 3^N L^3\le L^3 \lambda_0^{N-3\gamma/2}\lambda_0^{3\gamma/2},
    \\
    \|D_{t,0}\phi_0(t)\|_{N-1}
    &\leq 
    N^2 3^{N-1} i_0^{-1/2-\delta} L^3
    +
    N^3 4^{N-1} L^4\le L^4 \lambda_0^{N-3\gamma/2} \lambda_0^{2\gamma},
    \end{align}
    and on the pressure
    \begin{align}
    \|p_0(t)\|_N
    &\leq 
    (N+1)2^N L^2\le L^2\lambda_0^{N},
    \\
    \|D_{t,0}p_0(t)\|_{N-2}
    &\le 
    (N-1)2^{N-2} i_0^{-1/2-\delta}L^2
    +
    (N-1)^2 3^{N-2}L^3
    \le 
    L^3 \lambda_0^{N-1}\lambda_0^{\gamma/2}.
\end{align}
Then, recalling that by definition $\delta_0 = \delta_1 = \lambda_0^\gamma>1$, the estimates above imply the existence of a constant $\overline{M} \in (1,\infty)$, depending only on $L$ and $E$, such that \eqref{s:V}, \eqref{s:P}, \eqref{s:R}, \eqref{s:Phi} hold true $\mathbb{P}$-almost surely on the time interval $(-\infty,\mathfrak{t}]$.

\textit{Step 2: Construction of $(u_{t\wedge\mathfrak{t}},p_{t\wedge\mathfrak{t}})$}.
We iteratively apply \autoref{IterLemma}, starting from the initial tuple $(z_0,v_0,p_0,R_0,\phi_0)$ defined in Step 1. 
We obtain a sequence of dissipative Euler-Reynolds flows, $(v_q,p_q,R_q,\phi_q,z_q)$, and by \eqref{s:distance_v} we have that $v_q$ converges $\mathbb{P}$-almost surely to some $v$ in $C_t^0C^{\beta}_x$ for any $\beta<\alpha$.
Similarly, by \eqref{s:distance_p}, $p_q$ converges $\mathbb{P}$-almost surely to some $p$ in $C_t^0C^{2\beta}_x$.
Moreover, by construction $z_q$ converges $\mathbb{P}$-almost surely to $z$ uniformly on the time interval $[0,\mathfrak{t}]$, while $\|R_q\|_0$ and $\|\phi_q\|_0$ vanish by \eqref{s:R} and \eqref{s:Phi}. 
Notice that the tuple $(v,p,z)$ is $\{\mathbb{F}_t\}$-progressively measurable since each $(v_q,p_q,z_q)$ is so, by \autoref{IterLemma}. Thus, we have built a solution of the system \eqref{eq:EulerV2} on the time interval $[0,\mathfrak{t}]$. By the discussion in \autoref{ssec:SE_InductionScheme}, we have that the process $(u,p)$ with $u:=v+z$ solves \eqref{eq:SE} and \eqref{eq:LEE} on the same time interval, and has the regularity stated in \autoref{thm:main1}.

\textit{Step 3: Continuation of $(u,p)$}.
The proof is similar to \cite[Corollary 1.8]{LuLuZh25}.
Let $\tilde{W}_t := W_{t+\mathfrak{t}}-W_\mathfrak{t}$ be a cylindrical Wiener process on the filtration $\tilde{\mathbb{F}}_t := \sigma(\{\tilde{W}_s\}_{s \leq t},u_{\mathfrak{t}},p_\mathfrak{t})$. We intend to apply \cite[Theorem 1.7]{LuLuZh25} with initial condtion $u^{in} := u_{\mathfrak{t}} \in C^\beta_x$ $\mathbb{P}$-almost surely. This produces at least two distinct solutions $(\tilde{u}^1,\tilde{p}^1)$ and $(\tilde{u}^2,\tilde{p}^2)$ to \eqref{eq:SE} with noise $Q^{1/2}d\tilde{W}$, $\mathbb{P}$-almost surely of class $\tilde{u}^i \in C^0_t C^\vartheta_x$ with arbitrary $\vartheta<\beta$.
The pressure fields $\tilde{p}^i$ can be recovered from $\tilde{u}^i$ given by the aforementioned theorem by solving the equations $-\Delta \tilde{p}^i = \dvg \dvg (\tilde{u}^i \otimes \tilde{u}^i)$ subject to the boundary conditions $\int_{\T^3}\tilde{p}^i(t,x)dx = \int_{\T^3}p_{\mathfrak{t}}(x)dx$ for every $t \geq 0$, and are of class $\tilde{p}^i \in C_tC^{2\vartheta}_x$ $\mathbb{P}$-almost surely by \cite[Corollary 3.3]{CoDR20}.

Then the glued processes $(u^i_t,p^i_t) := (u_t,p_t) \mathbf{1}_{\{t \leq \mathfrak{t}\}} + (\tilde{u}^i_{t-\mathfrak{t}},\tilde{p}^i_{t-\mathfrak{t}}) \mathbf{1}_{\{t > \mathfrak{t}\}}$, $i=1,2$, are solutions to \eqref{eq:SE} that satisfy \eqref{eq:LEE} by construction on the time interval $t \in [0,\mathfrak{t}]$.
Since $\vartheta<\beta<\alpha$ is arbitrary, each $(u^i,p^i)$ has the space regularity stated in \autoref{thm:main1}, concluding the proof.

\end{proof}
\begin{oss}
    Let us stress that \eqref{s:V} and \eqref{s:distance_v} imply that, for any $q\in \N$ and $N\le n_0+3$,
    \begin{equation}\label{s:estUqN}
        \|u_q(t)\|_N+\|v_q(t)\|_N\lesssim 
        \begin{cases}
        1\,,\quad &N=0,\\
        \lambda_q^N\delta_q^{1/2},\quad &N \geq 1,
        \end{cases}
    \end{equation}
    where the prefactor does not depend on the parameter $a$. This independence of $a$ will be crucial in the following sections. Indeed, we will be able to absorb various implicit constants emerging from the convex integration scheme by picking $a$ large enough.
    Since this remark is not trivial just for $N=0$, let us focus on that case. Indeed, since $a>2$ it holds for every $q \in \N$ 
    \begin{equation}
    \begin{aligned}
        \|v_q(t)\|_0&\le \sum_{j=1}^q \|v_j(t)-v_{j-1}(t)\|_0\le C_v \oM\sum_{j=1}^q\delta_j^{1/2}\lambda_{j-1}^{-\gamma/2}\\
        &\le C_v \oM\sum_{j=1}^q\left(\frac{\lambda_1}{\lambda_j}\right)^{\alpha}\left(\frac{\lambda_0}{\lambda_{j-1}}\right)^{\gamma/2}\le C_v \oM\left(1+\sum_{j=2}^{q}a^{\alpha (b-b^j)}\right)\le C_v \oM \left(1+\sum_{j=2}^{q}2^{\alpha (b-b^j)}\right).
    \end{aligned}
    \end{equation}
\end{oss}

\subsection{Randomly forced Euler equations: Induction scheme}\label{ssec:FE_InductionScheme}
In this subsection, we focus on proving \autoref{thm:main2} and \autoref{thm:main3}, namely building solutions to \eqref{eq:FE}. Since we are not assuming any Brownian structure on $g$, we do not perform any decomposition as previously done in \autoref{ssec:SE_InductionScheme}, and we work directly on \eqref{eq:FE}.
Moreover, we do not fix a priori a probability space as we will also deal with probabilistically weak solutions.

In order to keep as coherent as possible the notation used in the different approximating Euler-Reynolds systems and in the Iteration Lemmas, in the following we will maintain the symbol $v$ for the velocity field (instead of using $u$ as in the statement of \autoref{thm:main2} and \autoref{thm:main3}). 
    \begin{defn}\label{defn:forced-Euler-Reynolds}
    A tuple of smooth tensors $(v,p,R,\phi,g,e)$, defined on some probability space $(\Omega,\mathbb{F},\mathbb{P)}$ satisfying the usual conditions, is a dissipative forced Euler-Reynolds flow, with forcing $g$ and energy loss $E' -e$, if it satisfies the following system
    \begin{align} \label{eq:ApproxEuler.forced}
    \begin{cases}
    \partial_tv+\dvg(v\otimes v)+\nabla p=\dvg(R) + g,
    \\
    \dvg(v) = 0,
    \\
     \partial_t\frac{|v|^2}{2}+\dvg\left(\left(\frac{|v|^2}{2}+p\right)v\right)=v \cdot g-E'+e
     +\frac{1}{2}D_t(\Tr(R))+\dvg(Rv)+\dvg (\phi).
    \end{cases}
    \end{align}
     where $D_t:=\partial_t+v\cdot \nabla$ denotes the advective derivative.
\end{defn}
We pass to introducing the Iteration Lemma, whose proof is postponed and will follow the same lines of the proof of \autoref{IterLemma}. 

\begin{lemma}\label{lemma:iter3}
    Let $\alpha \in (0,1/7)$. Then there exist parameters $a>2$, $b>1$, $n_0\in \N$ depending only on $b$, and $C_v, \oM \in (1,\infty)$ such that the following holds.

   Assume that $(v_q,p_q,R_q,\phi_q,g_q,e_q)$ is a solution to \eqref{eq:ApproxEuler.forced} satisfying for every $t\in \R$ the estimates \eqref{s:V}, \eqref{s:P}, \eqref{s:R}, \eqref{s:Phi}, as well as:
   \begin{align}
        \label{s:g3}
        \|g_q(t)\|_N&\le \overline{M}\lambda_q^N\delta_q,
        &&
        \|D_{t,q}g_q(t)\|_{N-2}\le\overline{M}\lambda_q^{N-1}\delta_q^{3/2},
        &&
        N \in \{1,\dots,n_0+3\},
        \\
        \label{s:e3}
        |e_q (t)| 
        &\leq \lambda_{q}^{-2\gamma} \delta\qq^{3/2},
        &&
        |e_q' (t)| 
        \leq \lambda_{q}^{1-3\gamma/2} \delta\qq^{2},
   \end{align}
where $D_{t,q}:=\partial_t+v_q\cdot \nabla$.
Then, there exists a new tuple $(v\qq,p\qq,R\qq,\phi\qq,g\qq,e\qq)$ which solves \eqref{eq:ApproxEuler.forced}, such that $p\qq=p_q$ and \eqref{s:distance_v} hold, the estimates \eqref{s:V}, \eqref{s:R}, \eqref{s:Phi},  \eqref{s:g3}, \eqref{s:e3} are satisfied with $q$ replaced by $q+1$, and in addition for all $t\in \R$ 
   \begin{align}
       \label{s:distance.g}
       \|g_{q+1}(t)-g_q(t)\|_{0}+ \frac{1}{\lambda\qq}\|g_{q+1}(t)-g_q(t)\|_{1}
       &\le 
       \oM\delta_{q+1} \lambda_q^{-\gamma/2}.
   \end{align}

Finally, if we replace \eqref{s:e3} with the condition on the increments
\begin{align}
    \label{s:distance.e3}
        |e_{q+1}(t)-e_q (t)| 
        &\leq \lambda\qq^{-2\gamma} \delta\qqq^{3/2},
\end{align}
then we can additionally enforce $g\qq = g_q$.
\end{lemma}

The first part of this iteration lemma will be used to show \autoref{thm:main3}; the final statement will be only needed to guarantee that the external forcing $g$ in \autoref{thm:main2} will have the same law under the two ergodic measures $\mu_1,\mu_2$.

For the sake of presentation, it seems better to prove \autoref{thm:main3} first, assuming \autoref{lemma:iter3}. 
We shall later discuss which modifications are needed to prove \autoref{thm:main2} instead.

As a preliminary step towards the proof of \autoref{thm:main3}, we need to properly craft the starting tuple of the induction scheme. This is the content of the next lemma.

\begin{lemma} \label{ass:Z}
Let $(\Omega,\mathbb{F},\mathbb{P})$ be a probability space satisfying the usual conditions, and let $Z$ be a smooth random process taking values in the space of divergence-free, zero-mean vector fields satisfying $\mathbb{P}$-almost surely the following properties:
\begin{itemize}
    \item $\| Z(t) \|_0 = 4\iota^{-1} \lambda_0^{\gamma/2}\lambda_1^{-\gamma/2}\lambda_1^{\alpha} \lambda_2^{-\alpha} =: \zeta^{1/2}$ for every $t \in \R$;
    \item $\| Z(t) \|_{N} \leq \lambda_1^{N-\gamma} \zeta^{1/2}$ for every $t \in \R$ and $N \in \{1,2,\dots,n_0+3\}$;
    \item $\| \partial_t Z(t) \|_{N} \leq \lambda_1^N \zeta^{1/2}$ for every $t \in \R$ and $N \in \{0,1,2,\dots,n_0+3\}$;
    \item $\| \mathcal{R}\partial_t Z(t) \|_{N} \leq \lambda_1^{N-\gamma} \zeta$ for every $t \in \R$ and $N \in \{0,1,2,\dots,n_0+3\}$;
    \item $\| \mathcal{R}\partial_t^2 Z(t) \|_{N} \leq \lambda_1^{N+1-\gamma} \zeta^{3/2}$ for every $t \in \R$ and $N \in \{0,1,2,\dots,n_0+1\}$.
\end{itemize}
Then, up to choosing $a \gg2$ sufficiently large and $E'$ small enough, the following tuple:
\begin{align}
    v_1 := Z,
    \quad
    p_1:=0,
    \quad
    g_1 := 0,
    \quad
    R_1 := Z \otimes Z + \mathcal{R}(\partial_t Z),
    \quad
    \phi _1 := -R_1 Z,
    \quad
    e_1 := E',
\end{align}
satisfies the iterative estimates \eqref{s:V}, \eqref{s:P}, \eqref{s:R}, \eqref{s:Phi}, \eqref{s:g3}, \eqref{s:e3} at step $q=1$, with constants $\overline{M} \geq (16\iota^{-2}+1)^2$ and $C_v = 1/\overline{M}$.
\end{lemma}

\begin{proof}
The iterative estimates \eqref{s:P}, \eqref{s:g3}, and \eqref{s:e3} are trivially satisfied, up to taking $E'$ small enough. 
Recalling the definition of $\delta_2 := \lambda_0^\gamma \lambda_1^{2\alpha} \overline{\delta}_2 = \frac{\iota^2 \lambda_1^\gamma}{16} \zeta > \zeta$ for $a$ sufficiently large, the assumption on $\| Z(t) \|_{N}$ readily implies \eqref{s:V}.
Let us move to verifying \eqref{s:R}. We have for $N=0$
\begin{align}
    \| R_1\|_0 \leq \| Z \|_0^2 + \| \mathcal{R}(\partial_t Z) \|_0 
    \leq
    \zeta + \lambda_1^{-\gamma}\delta_2
    \leq
    (16\iota^{-2}+1)\lambda_1^{-\gamma}\delta_2.
\end{align}
For $N \in \{1,\dots,n_0+3\}$ we have instead
\begin{align}
    \| R_1\|_N &\leq \| Z \otimes Z \|_N + \| \mathcal{R}(\partial_t Z) \|_N
    \leq 
    \lambda_1^{N-\gamma} \zeta + \sum_{k =0}^{N} \binom{N}{k} \| Z \|_k \|Z\|_{N-k}
    \\
    &\leq
    (n_0+5)2^{n_0+3}\lambda_1^{N-\gamma} \zeta
    \leq
    16\iota^{-2}(n_0+5)2^{n_0+3}\lambda_1^{N-2\gamma} \delta_2.
\end{align}
Imposing $a \gg 2$ large enough so that $(n_0+5)2^{n_0+3}\lambda_1^{-\gamma} < 1$, the first condition in \eqref{s:R} is satisfied.
For the second one we have
\begin{align}
\| D_{t,1} R_1 \|_{N-2}
&\leq
2\|\partial_t Z \otimes 
Z\|_{N-2}
+
\| \mathcal{R} \partial_t^2 Z \|_{N-2}
+
\| Z \cdot \nabla R_1 \|_{N-2}
\\
&\leq
\lambda_1^{N-1-\gamma} \zeta^{3/2} +\sum_{k=0}^{N-2} \binom{N-2}{k} \|Z\|_k (2\| \partial_t Z\|_{N-2-k}+\|R_1\|_{N-1-k})
\\
&\leq
64\iota^{-3}\lambda_1^{N-1-5\gamma/2} {\delta}_2^{3/2} 
+ 
16\iota^{-2}(n_0+2)2^{n_0+2} \lambda_1^{N-2-\gamma} {\delta}_2 
+ 
64\iota^{-3}(n_0+5)^2 2^{2n_0+4} \lambda_1^{N-1-5\gamma/2}{\delta}_2^{3/2} 
\\
&\leq
64\iota^{-3}(n_0+5)^2 2^{2n_0+5} \lambda_1^{N-1-5\gamma/2} {\delta}_2^{3/2},
\end{align}
and thus it suffices to impose $(n_0+5)^2 2^{2n_0+5} \lambda_1^{-3\gamma/2} < 1$. 
One can argue similarly for \eqref{s:Phi}; indeed
\begin{align}
  \| \phi_1 \|_0
    \leq  
\| R_1\|_0 \|Z\|_0 \leq
(16\iota^{-2}+1)^{3/2} \lambda_1^{-3\gamma/2} \delta_2^{3/2},
\end{align}
whereas for $N \in \{1,2\}$
\begin{align}
    \| \phi_1 \|_N
    &\leq
    \sum_{k=0}^{N} \binom{N}{k} \|Z\|_k \|R_1\|_{N-k} 
    \leq 
    4\iota^{-1}(16\iota^{-2}+1) \lambda_1^{N-5\gamma/2} \delta_2^{3/2} +64\iota^{-3} (n_0+5) 2^{n_0+4} \lambda_1^{N-5\gamma/2} \delta_2^{3/2}
    \\
    &\leq
    (16\iota^{-2}+1)^{3/2}(n_0+5) 2^{n_0+5}\lambda_1^{N-5\gamma/2} \delta_2^{3/2},
\end{align}
which requires the condition $(n_0+5) 2^{n_0+5} \lambda_1^{-3\gamma/2} < 1$.
Finally, for $N \in \{1,2\}$
\begin{align}
  \| D_{t,1}\phi_1 \|_{N-1} 
  &\leq \|R_1 D_{t,1}Z  \|_{N-1}+\|Z D_{t,1}R_1 \|_{N-1}
\\&
\leq \sum_{k=0}^{N-1} (\|R_1\|_k \|D_{t,1}Z  \|_{N-1-k} + \|Z\|_k \|D_{t,1}R_1  \|_{N-1-k})
\\
&\leq
4\iota^{-1}(16\iota^{-2}+1)(n_0+5)2^{n_0+3} \lambda_1^{N-1-3\gamma/2}\delta_2^{3/2}
+
(16\iota^{-2}+1)^2(n_0+5)2^{n_0+3} \lambda_1^{N-3\gamma}\delta_2^{2}
\\
&\quad+
64\iota^{-3}(n_0+5)^2 2^{2n_0+5} \lambda_1^{N-3\gamma} \delta_2^2
\leq
(16\iota^{-2}+1)^2(n_0+5)^2 2^{2n_0+6}\lambda_1^{N-3\gamma} \delta_2^2,
\end{align}
Therefore we impose $(n_0+5)^2 2^{2n_0+6} \lambda_1^{-3\gamma/2} < 1$.
\end{proof}

\begin{oss} \label{oss:Z}
    There exists a measure $\nu$ on the path space $\mathscr{Z}$ satisfying the assumptions of \autoref{thm:main3}.
    A possible construction goes as follows.
We take the probability space $(\Omega,\mathbb{F},\mathbb{P}) := ([0,1],\mathscr{B}, \mathcal{L}eb)$ and the statistically stationary process 
\begin{align}
 Z(\omega,t,x) := \zeta^{1/2}  \cos(\varpi x_3) (\sin(t+2\pi\omega) e_1 + \cos(t+2\pi \omega)e_2),   
\end{align}
with $\varpi \in 2\pi \N$. Here $\{e_i\}_{i={1,2,3}}$ are orthonormal vectors in $\R^3$ and $x_3 := x \cdot e_3$. 
Notice that $\partial_tZ$, $\mathcal{R}\partial_tZ$ and $\mathcal{R}\partial_t^2Z$ are explicitly computable:
\begin{align}
  \partial_t Z(\omega,t,x) =  \zeta^{1/2} (\cos(t+2\pi\omega) e_1 - \sin(t+2\pi \omega)e_2) \cos(\varpi x_3),
\end{align}
\begin{align}
        \mathcal{R}\partial_t Z (\omega,t,x) 
        = 
        \zeta^{1/2}  \frac{\sin(\varpi x_3)}{\varpi}  \begin{pmatrix}
0&0&\cos(t+2\pi\omega) \\
             0&0&-\sin(t+2\pi\omega) \\
             \cos(t+2\pi\omega)&-\sin(t+2\pi\omega)&0
\end{pmatrix},
\end{align}
and
\begin{align}
        \mathcal{R}\partial_t^2 Z (\omega,t,x) 
        = 
        -\zeta^{1/2}  \frac{\sin(\varpi x_3)}{\varpi}  \begin{pmatrix}
0&0&\sin(t+2\pi\omega) \\
             0&0&\cos(t+2\pi\omega) \\
             \sin(t+2\pi\omega)&\cos(t+2\pi\omega)&0
\end{pmatrix},
\end{align}
and the assumptions of \autoref{ass:Z} are satisfied if we take $\varpi \in 2\pi \N$ such that $\lambda_1^\gamma \zeta^{-1/2}\leq\varpi \leq \lambda_1^{1-\gamma}$. Such $\varpi$ exists by our choice of parameters, up to taking $a$ large enough.
By construction $\|Z(t+1) - Z(t)\|_{C_x}=\zeta^{1/2}2\sin(1/2) > \frac12 \zeta^{1/2}$ for every $t \in \R$ and $\omega \in \Omega$; in particular, $\inf_{t \in \R}\|Z(t+1) - Z(t)\|_{C_x} > \frac12  \zeta^{1/2}$ holds and therefore $\nu := \mathscr{L}_\mathbb{P}(Z)$ is concentrated on the path space $\mathscr{Z}$ defined in \autoref{thm:main3}.
    
\end{oss}

We are now ready to give the proof of \autoref{thm:main3}, assuming \autoref{lemma:iter3}. 
\begin{proof}[Proof of \autoref{thm:main3}]
\textit{Step 1: Construction of global solutions to \eqref{eq:FE}.}
Let $Z$ be the canonical process on the path space $\mathscr{Z}$ with law $\nu$.
By \autoref{lemma:iter3} and the iterative estimates we obtain a solution $(v,p,g):\mathscr{Z} \to \mathscr{X}$ of \eqref{eq:FE} with dissipation equal to $-E'$, satisfying $p=0$ and the global $\nu$-almost sure bounds:
\begin{align} \label{eq:bound.tightness}
   \|v\|_{C_tC^\alpha_x} 
   + \|v\|_{C^1_tH^{-10}_x} 
+ 
\|g\|_{C_tC^{2\alpha}_x} 
+ \|g\|_{C^\alpha_tH^{-10}_x}
\leq 
M 
\end{align}
for some finite constant $M$. 
Indeed, the bounds on $\|v\|_{C_tC^\alpha_x}$ and $\|g\|_{C_tC^{2\alpha}_x}$ descend from the iterative estimates \eqref{s:distance_v} and \eqref{s:distance.g}, 
the bound on $\|v\|_{C^1_tH^{-10}_x}$ descends from the equation satisfied by $v$, and the bound on $\|g\|_{C^\alpha_tH^{-10}_x}$ follows by \eqref{s:g3}, \eqref{s:distance.g}, and
\begin{align}
   \| g\|_{C^\alpha_tH^{-10}_x}
    &\lesssim
   \sum_{q \geq 0} \| g\qq-g_q \|_{C_tH^{-10}_x}^{1-\alpha}  \|\partial_t (g\qq-g_q )\|_{C_tH^{-10}_x}^\alpha
    \\
&\lesssim
    \sum_{q \geq 0}\| g\qq-g_q \|_{0}^{1-\alpha} \left( \|D_{t,q+1} (g\qq-g_q ) \|_0^\alpha + \|v\qq \otimes (g\qq-g_q )\|_0^\alpha \right)
    < \infty.
\end{align}

Moreover, by construction $(v,p,g)$ belongs $\nu$-almost surely to the space
\begin{align}
    \mathscr{X}_0 := \{ (u,p,g) \in C(\R, c^\alpha(\T^3,\R^3) \times  c^{2\alpha}(\T^3,\R) \times  c^{2\alpha}(\T^3,\R^3)), \, \dvg(u)=\dvg(g)=0\}
\end{align}
where $c^\vartheta$ denotes the closure of smooth functions with respect to the H\"older topology, which is a separable subspace of $C^\vartheta$.

\textit{Step 2: Estimate on $\|v-Z\|_0$.}
By \eqref{s:distance_v} and recalling that $v_1 = Z$ and $C_v\overline{M}=1$ we can bound
\begin{align}
   \|v-Z\|_{0}
   &\leq
   \sum_{q=1}^\infty
   \|v_{q+1}-v_q\|_{0}
   \leq
   \sum_{q=1}^\infty \delta\qq^{1/2} \lambda_q^{-\gamma/2}
   \leq 
   \lambda_0^{\gamma/2} \lambda_1^{-\gamma/2} \lambda_1^{\alpha} \sum_{q=1}^\infty \lambda\qq^{-\alpha}.
\end{align}
Next we want to bound the sum $\sum_{q=1}^\infty \lambda\qq^{-\alpha}$ from above. Observe that $\lambda\qq^{-\alpha} \leq 2 a^{-\alpha b^{q+1}}$ and for every $q \geq 1$
\begin{align}
    \frac{ a^{-\alpha b^{q+2}}}{a^{-\alpha b^{q+1}}}
    =
    a^{-\alpha b^{q+1}(b-1)}
    \leq
    a^{-\alpha b^2(b-1)}
    =: c
    < 1/2,
\end{align}
up to taking $a \gg 2$ large enough.
Therefore, by comparison with the geometric series we have
\begin{align} \label{eq:estimate.v-Z}
    \|v-Z\|_{0}
   &\leq
   4 
   \lambda_0^{\gamma/2} \lambda_1^{-\gamma/2} \lambda_1^{\alpha} \lambda_2^{-\alpha}
   = \iota \zeta^{1/2} = \iota \inf_{t \in \R} \|Z(t)\|_0.
\end{align}

\textit{Step 3: Construction of ergodic stationary solutions.}
Consider the Krylov-Bogoliubov time averages
\begin{align}
  \mu_T :=  \frac{1}{T} \int_0^T \mathscr{L}_\nu(S_t(v,p,g,Z))dt
\end{align}
where $S_t$ is the shift operator on the path space\footnote{In this case $S_t(v,p,g,Z)(s) := (v,p,g,Z)(t+s)$.}.
Using that $\mathscr{L}_{\nu}(S_t Z) \equiv \nu$ and Simon compactness criterion \cite[Corollary 5]{Si86} (recalling bound \eqref{eq:bound.tightness} and Ascoli-Arzelà Theorem) we have that the laws $\{\mu_T\}_{T >0}$ are tight in $\mathscr{X}^\beta_{loc} \times \mathscr{Z}_{loc}$, where for arbitrary $\beta \in(0,\alpha)$ we have denoted
\begin{align}
    \mathscr{X}^\beta_{loc} :=
    \{ (u,p,g) \in C_{loc}(\R, c^\beta(\T^3,\R^3) \times  c^{2\beta}(\T^3,\R) \times  c^{2\beta}(\T^3,\R^3)), \, \dvg(u)=\dvg(g)=0\},
\end{align}
and $\mathscr{Z}_{loc}$ means that we use the topology induced by $C_{loc}(\R,C^1(\T^3,\R^3))$.
Notice that the space $\mathscr{X}^\beta_{loc} \times \mathscr{Z}_{loc}$ is metrizable and separable. Therefore, by Prokhorov Theorem there exists a subsequence $T_n \to \infty$ such that $\mu_{T_n} \rightharpoonup \mu$, where $\mu$ is a measure on $\mathscr{X}^\beta_{loc} \times \mathscr{Z}_{loc}$. 
Skorokhod Theorem then implies the existence of an auxiliary probability space $(\tilde{\Omega},\tilde{\mathbb{F}},\tilde{\mathbb{P}})$, supporting $\mathscr{X}^\beta_{loc} \times \mathscr{Z}_{loc}$-valued random variables $\{(\tilde{v}_n,\tilde{p}_n,\tilde{g}_n,\tilde{Z}_n)\}_{n \in \N}$ with $\mathscr{L}_{\tilde{\mathbb{P}}}(\tilde{v}_n,\tilde{p}_n,\tilde{g}_n,\tilde{Z}_n) = \mu_{T_n}$ and $(\tilde{v}_n,\tilde{p}_n,\tilde{g}_n,\tilde{Z}_n) \to (\tilde{v},\tilde{p},\tilde{g},\tilde{Z})$ $\tilde{\mathbb{P}}$-almost surely, with $\mathscr{L}_{\tilde{\mathbb{P}}}(\tilde{v},\tilde{p},\tilde{g},\tilde{Z})=\mu$. 
In particular $\mathscr{L}_{\tilde{\mathbb{P}}}(\tilde{Z}) = \nu$.

By construction the law $\mu$ is shift invariant on $\mathscr{X}^\beta_{loc} \times \mathscr{Z}_{loc}$, namely the process $(\tilde{v},\tilde{p},\tilde{g},\tilde{Z})$ is statistically stationary.
Moreover, the following set
\begin{align}
A := \Big\{ &(v,p,g,Z) \in \mathscr{X}^\beta_{loc} \times \mathscr{Z}_{loc} :\quad \,(v,p,g) \mbox{ solves \eqref{eq:FE} weakly with dissipation equal to}-E',
\\
&\quad {p}=0, \quad  \eqref{eq:estimate.v-Z} \mbox{ holds},
\quad\|v\|_{C_tC^\beta_x} 
+ 
\|g\|_{C_tC^{2\beta}_x} 
\leq 
M 
 \Big\}
\end{align}
is given full probability by the measures $\mu_{T_n}$ for every $n \in \N$ and is stable under convergence with respect to the topology of $\mathscr{X}^\beta_{loc} \times \mathscr{Z}_{loc}$, thus $\mu(A) = 1$.

Next, we want to construct an ergodic measure. Let us consider the set $\mathfrak{M}$ of statistically stationary measures $\mu$ on $\mathscr{X}^\beta_{loc} \times \mathscr{Z}_{loc}$ with second marginal equal to $\nu$ and such that $\mu(A) = 1$. 
By the discussion above, $\mathfrak{M}$ is non-empty; moreover, $\mathfrak{M}$ is convex and closed and therefore by Krein-Milman Theorem $\mathfrak{M}$ admits extremal points, which correspond to ergodic measures (similar argument as in \cite[Theorem 4.2]{HoZhZh25}).
The corresponding canonical process $(v,p,g,Z)$ satisfies $\mu$-almost surely $ \|v\|_{C_tC^\beta_x} 
+ 
\|g\|_{C_tC^{2\beta}_x} 
+
\|Z\|_{C_tC^1_x} \leq 2M$,
since $\mu(A)=1$ and $\mathscr{L}_\mu(Z)=\nu$. 
Finally, using that $\beta<\alpha<1/7$ are arbitrary, $\mu$ is concentrated on the set $\mathscr{X} \times \mathscr{Z}$, namely the regularity stated in \autoref{thm:main3} is fulfilled.

\textit{Step 4: $v$ is not a steady state and is genuinely random.}
Let $\mu \in \mathfrak{M}$ be an ergodic measure as above.  
Using that $\mu(A)=1$ we have that for the canonical process $(v,p,g,Z)$ it holds $\mu$-almost surely:
\begin{align}
\sup_{s,t \in \R} \|v(t) - v(s)\|_0
&=
\sup_{s,t \in \R} \|v(t) - Z(t) + Z(t) - Z(s) + Z(s)- v(s)\|_0
\\
&\geq
\sup_{s,t \in \R} \|Z(t) - Z(s)\|_0 - 2 \|v-Z\|_0
\geq (1-2\iota) \zeta^{1/2}>0.
\end{align}

To give a lower bound on the variance of $v$, we observe that stationarity of $Z$ plus the control on the excursions $\|Z(t+1) - Z(t)\|_0 > \frac12 \zeta^{1/2}$ give a lower bound on the variance of $Z(t)$. More specifically, we have the following:
\begin{align} \label{eq:variance.Z}
    \forall \psi \in C_x, \quad \forall t \in \R:
    \quad
    \mu \left\{ Z(t) \in B_{C_x}\left(\psi,\frac14 \zeta^{1/2}\right)  \right\} \leq \frac12.
\end{align}
Indeed, since $\|Z(t+1) - Z(t)\|_0 > \frac12 \zeta^{1/2}$ we have
\begin{align}
    Z(t) \in B_{C_x}\left(\psi,\frac14 \zeta^{1/2}\right) \Longrightarrow
    Z(t+1) \notin B_{C_x}\left(\psi,\frac14 \zeta^{1/2}\right),
\end{align}
and thus by stationarity and the line above we deduce
\begin{align}
    \mu \left\{ Z(t) \notin B_{C_x}\left(\psi,\frac14 \zeta^{1/2}\right)  \right\}=\mu \left\{ Z(t+1) \notin B_{C_x}\left(\psi,\frac14 \zeta^{1/2}\right)  \right\} \geq \mu \left\{ Z(t) \in B_{C_x}\left(\psi,\frac14 \zeta^{1/2}\right)  \right\}.
\end{align}

From \eqref{eq:variance.Z} we deduce that for every $\psi \in C_x$ and $t \in \R$
\begin{align}
   \mathbb{E}^\mu \left[ \|Z(t)-\psi\|_0^2\right]
   \geq
   \mathbb{E}^\mu \left[ \|Z(t)-\psi\|_0^2 \mathbf{1}_{\{Z(t) \notin B(\psi,\zeta^{1/2}/4)\}} \right]
   \geq \frac{\zeta}{16} \mu \left\{ Z(t) \notin B_{C_x}\left(\psi,\frac14 \zeta^{1/2}\right)  \right\}
   \geq
   \frac{\zeta}{32}. 
\end{align}

From the line above and triangular inequality we deduce the following:
\begin{align}
  \mathbb{E}^\mu \left[ \|v(t)-\psi\|_0^2\right]
  \geq
  \mathbb{E}^\mu \left[ \|Z(t)-\psi\|_0^2\right]
  -
  \mathbb{E}^\mu \left[ \|v(t)-Z(t)\|_0^2\right]
  \geq
  \frac{\zeta}{32} - 
   \iota^2 \zeta
   > \frac{\zeta}{64} ,
\end{align}
since we have choosen $\iota<1/8$. 
We conclude by taking $t \in \R$ arbitrary and $\psi = \mathbb{E}^\mu[v(t)]$:
\begin{align}
    \mathbb{E}^\mu \left[ \|v(t)-\mathbb{E}^\mu[v(t)]\|_0^2\right]
  \geq
  \frac{\zeta}{64}.
\end{align}
\end{proof}

Next we sketch the proof of \autoref{thm:main2}. 

\begin{proof}[Proof of \autoref{thm:main2}]
Let $Z$ be the process defined in \autoref{oss:Z}. It holds $\fint_{\T^d}|Z|^2 = \frac12 \zeta$.
Fix $E' \equiv \lambda_1^{-2\gamma} \delta_2^{3/2}$ and, for $i \in \{1,2\}$, initialize two tuples as follows:
   \begin{align}
    v^i_1 &:= \frac1i Z,
    \quad
    p^i_1:=0,
    \quad
    g^i_1 := \frac{E'}{\zeta} Z,
    \quad
    R^i_1 := \frac{1}{i^2} Z \otimes Z + \mathcal{R}\left( \frac1i \partial_t Z- \frac{E'}{\zeta} Z \right),
    \\
    \phi^i_1 &:= -\frac1i R_1^i Z - \frac{E'}{\zeta i} \overline{\mathcal{R}}\left( Z^2-\frac12 \zeta \right),
    \quad
    e^i_1 := E' \left( 1-\frac{1}{2i}\right).
\end{align}
Similar computations as those in \autoref{ass:Z} show that the iterative estimates are satisfied for $q=1$ and $a$ sufficiently large.
Applying iteratively the second part of \autoref{lemma:iter3} we obtain two tuples $\{(v^i,p^i,g^i,e^i)\}_{i=1,2}$ satisfying $p^1=p^2=0$ and $g^1=g^2=\frac{E'}{\zeta}Z$.
We claim that $\{(v^i,p^i,g^i)\}_{i=1,2}$ solve \eqref{eq:FE}: Indeed, by \eqref{s:distance.e3} we have for $i \in \{1,2\}$ and arbitrary $t \in \R$,
\begin{align}
    E' - e^i(t) &\geq E' - e^i_1(t) - \sum_{q=1}^\infty |e^i\qq(t)-e^i_q(t)|
    \\
    &= \frac{1}{2i} \lambda_1^{-2\gamma} \delta_2^{3/2}
    -
    \sum_{q=1}^\infty
    \lambda\qq^{-2\gamma} \delta\qqq^{3/2}
    >
    \frac{1}{2i} \lambda_1^{-2\gamma} \delta_2^{3/2}
    -
    4 \lambda_2^{-2\gamma} \delta_3^{3/2}
    >
    \frac{1}{8} \lambda_1^{-2\gamma} \delta_2^{3/2},
\end{align}
up to taking $a$ sufficiently large.
Then we can repeat the Krylov-Bogoliubov argument of the proof of \autoref{thm:main3}, with the only difference that the set $A$ is now replaced by
\begin{align}
A^i := \Big\{ &(v,p,g,Z) \in \mathscr{X}^\beta_{loc} \times \mathscr{Z}_{loc}:\quad \,(v,p,g) \mbox{ solves \eqref{eq:FE} weakly with dissipation } \leq -\frac{1}{8} \lambda_1^{-2\gamma} \delta_2^{3/2},
\\
&\quad {p}=0, \quad g = \frac{E'}{\zeta}Z, \quad\|v-\frac{1}{i}Z\|_{0} \leq \iota \zeta^{1/2},
\quad\|v\|_{C_tC^\beta_x} 
\leq 
M
 \Big\}.
\end{align}
Then we obtain two ergodic measures $\tilde\mu_1$ and $\tilde\mu_2$ concentrated on $\mathscr{X} \times \mathscr{Z}$ with $\mathscr{L}_{\tilde\mu_1}(g)=\mathscr{L}_{\tilde\mu_2}(g)$ and $\tilde\mu_i(A^i)=1$;
\eqref{eq:no-steady} and \eqref{eq:variance>0} are deduced by the same computations as above, up to replacing $Z$ with $\frac1i Z$ (notice however that \eqref{eq:variance>0} requires $\iota < 1/8\sqrt{2}$ when $i=2$). Moreover, the two measures $\tilde\mu_1$ and $\tilde\mu_2$ do not coincide since
\begin{align}
    \| v \|_0 \geq \|Z\|_0-\|v - Z\|_0 \geq (1-\iota) \zeta^{1/2} > \frac34 \zeta^{1/2},
    \quad
    \tilde\mu_1 \, \mbox{-\,almost surely},
\end{align}
while 
\begin{align}
    \| v \|_0 \leq \|Z/2\|_0 +\|v - Z/2\|_0 \leq (1/2+\iota) \zeta^{1/2} < \frac34 \zeta^{1/2},
    \quad
    \tilde\mu_2 \, \mbox{-\,almost surely}.
\end{align}
To conclude, we observe that the restriction of $\tilde\mu_i$ to $\mathscr{X}$, defined as $\mu_i(\,\cdot\,) := \tilde\mu_i(\,\cdot \,\times \mathscr{Z})$, remains ergodic with respect to the shift operator $S_t$.

\end{proof}

\section{Construction of the velocity perturbation} \label{sec:velocity}
This section explains the construction of the perturbation $w:=v\qq-v_{q}$, where $(v_q,p_q,R_q, \phi_q,z_q)$ is a solution to \eqref{eq:ApproxEulerV}. In particular, $w$ is obtained by adapting the construction introduced in \cite{DLK23} to our stochastic framework, to which we refer for more details.

\subsection{Mikado flows} \label{ssec:mikado}
 The perturbation $w$ is built by combining highly oscillating vector fields modulated by scalar amplitudes, that will be chosen to cancel out $R_q$ and $\phi_q$ through the following two geometric lemmas.

\begin{lemma}%{(Geometric Lemma 1)}
\label{GeomLemma1}
Let $\mathcal{S}^3$ denote the space of $3 \times 3$ symmetric matrices, and consider a family of vectors $\mathcal{F_{R}}:=\{k_i\}_{i=1}^6\subset \Z^3$ and a constant $C>0$ such that
 \begin{equation}
     \sum\limits_{i=1}^6\ktimes=C\Id,\quad\text{and } \{k_i\otimes k_i\}\text{ is a basis of }\mathcal{S}^3.
 \end{equation}
 Then, there exists a positive constants $N_0=N_0(\mathcal{F_{R}})$ such that for any $N\le N_0$ we can find functions $\{\gammai\}_{i=1}^6 \subset \C^{\infty}(S_N,(0,\infty))$, where $S_N:=\{
 K \in \mathcal{S}^3 :  |K-Id|_{\infty}\le N\}$, such that for every $K \in S_N$ it holds
 \begin{equation}\label{RapprFormMatrix}
     K=\sum\limits_{i=1}^6\gammai^2(K)\ktimes.
 \end{equation}
\end{lemma}

\begin{lemma}%{(Geometric Lemma 2)}
\label{GeomLemma2}
   Let $\mathcal{F_{\phi}}:=\{k_i\}_{i=1}^4\subset \Z^3$ be a family of vectors such that $\{k_1,k_2,k_3\}\subset \Z^3$ is an orthogonal frame and $k_4 = -(k_1+k_2+k_3)$. 
    Then, for any $N_0>0$, there exist real-valued affine functions $\{\lambdai\}_{i=1}^4$ on $\R^3$ such that their restrictions to $\mathcal{V}_{N_0}:=\{u \in \R^3: |u|\le N_0\}$ are of class $\C^{\infty}(\mathcal{V}_{N_0},[N_0,\infty))$ and for every $u \in \mathcal{V}_{N_0}$ it holds
    \begin{equation}\label{RapprFormVect}
        u=\sum\limits_{i=1}^4\lambdai(u)k_i.
    \end{equation}
\end{lemma}

As in \cite{DLK23}, these two lemmas allow to find 27 disjoint families of directions $\{\F^j\}_{j\in\Z^3_3}$, indexed by equivalence classes $j \in \Z^3_3$ (triples of integers modulo $3$), such that any $\F^j$ can be decomposed as $\F^j=\F^j_R\cup\F^j_{\phi}$, where $\F^j_R$ and $\F^j_{\phi}$ satisfy the assumptions of \autoref{GeomLemma1}, \autoref{GeomLemma2}, respectively.
For any direction $f\in \F := \bigcup_{j \in \Z^3_3} \mathcal{F}^j$ let us define 
\begin{equation}
l_f:=\{x \in \T^3: x-\sigma f \in 2\pi\Z^3 \text{ for some }\sigma \in \R\},
\end{equation}
which is nothing but the the line in direction $f$ passing through the origin, but seen as subset of $\T^3$.
Moreover, let $\eta$ be a positive constant and $\psi_f\in C^{\infty}(\T^3,\R)$ be a function such that:
\begin{align}
    &\label{a:psi1}\supp(\psi_f)\subset B\left(l_f,\frac{\eta}{10}\right),\quad f\cdot \nabla\psi_f=0,
    \end{align}
namely with support concentrated in the pipe of radius $\eta/10$ around the line $l_f$, and invariant by translations in the direction $f$. We assume the following renormalization properties:
    \begin{align}
    &\label{e:meanPipesR}
    \int_{\T^3} \psi_fdx = \int_{\T^3} \psi_f^3dx = 0,
    \quad 
    \fint_{\T^3} \psi_f^2dx=1, 
    &\text{ for all } f\in \F_R;
    \\
    &\label{e:meanPipesPhi}
    \int_{\T^3} \psi_fdx = 0, 
    \quad 
    \fint_{\T^3} \psi^3_fdx = 1,
    &\text{ for all } f\in \F_\phi.
\end{align}

\subsection{Partition of unity and shifts}
In the following we denote $Q(x,r) := \{ y \in \T^3 : |y-x|_{\infty} \leq r \}$ the hypercube in $\T^3$ of points whose box distance from $x$ is at most $r$. 
Let $\chi_0 : \T^3 \to \R_+$ and $\theta_0 : \R \to \R_+$ be nonnegative smooth functions and suppose $\supp(\chi_0) \subset Q\left(0,9\pi/8\right)$ with $\chi_0 \equiv 1 \, \mbox{on}\,Q\left(0,7\pi/8\right)$, and similarly $\supp (\theta_0) \subset [-1/8,9/8]$ with $\theta_0 \equiv 1 \, \mbox{on}\, [1/8,7/8]$.
 
For every $n \in \Z^3$ and $m \in \Z$, define $\chi_n(x) := \chi_0(x-2\pi n)$ and $\theta_m(t)=\theta_0(t-m)$.
The functions $\chi_0,\theta_0$ can be chosen in such a way that 
\[\sum_{n\in \Z^3}\chi_n^6(x)=1,
\quad\forall x\in \T^3;
\qquad \sum_{m\in \Z}\theta_m^6(t)=1, \quad\forall t \in \R.\]

Next, let us fix $\nu \in (0,\frac{1-7\alpha}{4\alpha})$. Then, we introduce the parameter
\begin{equation}
\label{eq:definition_eps}
\epsilon:=\begin{cases}
  C_{\epsilon}  \lambda_0^{-1/2}\lambda_1^{-1/2},\,
  &q=0,
  \\
C_{\epsilon}\lambda_q^{-1/2}\lambda\qq^{-1/2} \delta_q^{-1/4-\nu}\delta\qq^{-1/4+\nu},\,
&q\geq 1,
\end{cases}
\qquad 
C_{\epsilon}:=(C_v \oM+L)^{-1}\min\left(\frac{1}{2},\frac{\inf_{j} N_0(\F_R^j)}6,\frac{\eta}{10\pi (L+1)}\right),
\end{equation}
which represents the size of the temporal partition, and
\begin{equation} \label{eq:definition_mu}
\mu^{-1}:=\begin{cases}\left\lceil\lambda_0^{1/2}\lambda_1^{1/2}\right\rceil,\,&q=0,\\
\left\lceil \lambda_q^{1/2}\lambda\qq^{1/2} \delta_q^{1/4+\nu}\delta\qq^{-1/4-\nu}\right\rceil,\,&q\geq 1,
\end{cases}
\end{equation}
which represents instead the size of the spatial partition.

Given index sets 
\begin{equation}
\begin{aligned}
  \I &:= 
  \{(m,n,f)\in \Z\times\Z^3\times\F:f\in\F^n\},
  \\
  \I_R &:= 
  \{I\in\I:f\in\F_R^n\},
  \\
  \I_{\phi} &:=\{I\in\I:f\in\F_{\phi}^n\},  
\end{aligned}
\end{equation}
then let us introduce
\begin{equation}
\chi_I(x):=\begin{cases}
        \chi_n^2(\frac{x}{\mu}),\quad I\in\I_{\phi},\\
        \chi_n^3(\frac{x}{\mu}),\quad I\in\I_{R},
    \end{cases}
    \quad
    \mbox{ and }
    \quad
    \theta_I(t):=\begin{cases}
        \theta_m^2(\frac{t}{\epsilon}),\quad I\in\I_{\phi},\\
        \theta_m^3(\frac{t}{\epsilon}),\quad I\in\I_{R}.
    \end{cases}
\end{equation}

Given a radial smooth function  $\mathfrak{m}:\R^3 \to \R$ supported in $B(0,2)$ and such that $\mathfrak{m} \equiv 1$ on $\overline{B(0,1)}$, consider the standard Littlewood-Paley operators acting on Fourier transforms as follows:
\begin{equation}
    \mathscr{F}[\P_{\le 2^j}f](\xi)
    :=
    \mathfrak{m}(\xi\,2^{-j})\F[f](\xi),
    \quad 
    \mathscr{F}[\P_{> 2^j}f](\xi)
    :=
    [1-\mathfrak{m}(\xi\,2^{-j})]\mathscr{F}[f](\xi).
\end{equation}
Let us introduce the frequency cutoff parameter 
\begin{align}
    \label{eq:definition_l}
l:=\begin{cases}
    \lambda_0^{-1/2}\lambda_1^{-1/2},\,&q=0,\\
\lambda_q^{-1/2}\lambda\qq^{-1/2} \delta_q^{-1/2}\delta\qq^{1/2},\quad&q\geq 1,\\
\end{cases}
\end{align}
and denote
\begin{align}
    \P_{\le l^{-1}}:=\P_{\le 2^S},
    \quad
    S:=\sup \{ j \in \Z \,:\, 2^j\le l^{-1}\}.
\end{align}
Incidentally, we also define here the Littlewood-Paley operator projecting on the frequency shell of size approximately equal to $2^j$, namely: 
\begin{align}
    \P_{2^j}:=\P_{>2^{j-1}}-\P_{>2^{j}},
\end{align}
that will be useful later.
Then we define the spatial regularization of our velocity fields as:
\begin{equation}\label{e:u_l}
    v_l:=\P_{\le l^{-1}}[v_q],\quad z_l:= \P_{\le l^{-1}}[z_q],\quad u_l:=v_l+z_l.
\end{equation}
Denote $D_{t,l} := \partial_t + u_l \cdot \nabla$ the advective derivative with respect to the velocity field $u_l$.
Finally, for any $I\in\I_m:=\{(m,n,f):n\in\Z^3,f\in\F^n\}$, let us consider the flow $\xi_I$ solving 
\begin{equation}\label{e:flow}
    \begin{cases}
        D_{t,l} \xi_I(t,x)=0,\\
        \xi_I(\epsilon(m-1/8),x)=x.
    \end{cases}
\end{equation}
Since $\xi_I$, $I \in \mathcal{I}_m$ only depends on $m \in \Z$, we shall indifferently denote $\xi_m := \xi_I$ hereafter.

\subsection{Perturbation of the velocity}
The principal part of the perturbation $w=v_{q+1}-v_q$ is defined as follows
\begin{equation}\label{e:w0}
    w_o(t,x):=\sum_{I\in \I}\theta_I(t)\chi_I(\xi_I(t,x))a_I(t,x)\nabla\xi_I(t,x)^{-1}f_I\psi_{I}(\lambda\qq \xi_I(t,x))=:\sum_{I\in\I}w_I^o(t,x),
\end{equation}
where the index $o$ stands for \emph{oscillatory}.
The amplitudes $a_I$ will be defined in the following \autoref{ssec:amplitude}, and $\psi_I$ is given by
\begin{equation}\label{e:psiI}
    \psi_I(x)=\psi_{f_I}(x-s_I),
\end{equation}
where $f_I$ is the direction appearing in the index $I$ (i.e. $I=(m,n,f_I)$ for some $m,n$), and the set of shifts $\{s_I\}_{I\in\I} \subset \R^3$ is chosen such that the following conditions hold:
\begin{itemize}
    \item $\supp w_I^o\cap\supp w_J^o=\emptyset$ for all $I\neq J$;
    \item For every $I=(m,n,f)$, the shift $s_I$ satisfies $s_I=s_{m,n}+p_{n,f}$ for some $s_{m,n}$ depending only on $m,n \in \Z \times \Z^3$ and $p_{n,f}$ depending only on $f \in \mathcal{F}^n$;
    \item $s_{m,n}=s_{m,n'}$ for $(n'-n)\mu \in \Z^3$.
\end{itemize}
The existence of such a family of shifts implies the following constraint on the parameters:
\begin{equation}
    C_\epsilon \epsilon\,\mu\|\nabla u_l(t)\|_0\lesssim\frac{\eta}{\lambda\qq}. 
\end{equation}
Its proof is a simple adaptation of \cite[Proposition 3.5]{DLK23}, with only minor modifications needed to take the randomness of $u_l$ into account, and is omitted for the sake of brevity. 
As in \cite{DLK23}, we can expand $\psi_I,\psi^2_I,\psi^3_I$ in Fourier series as follows:
\begin{equation}
    \psi_I(x)
    =
    \sum_{k\in\Z^3\setminus\{0\}}\mathring{b}_{I,k}e^{ik\cdot x}
    ,\quad 
    \psi^2_I(x)
    =
    \mathring{c}_{I,0}+\sum_{k\in\Z^3\setminus\{0\}}\mathring{c}_{I,k}e^{ik\cdot x},
    \quad 
    \psi^3_I(x)=\mathring{d}_{I,0}+\sum_{k\in\Z^3\setminus\{0\}}\mathring{d}_{I,k}e^{ik\cdot x},
\end{equation}
where \eqref{a:psi1} implies that $\mathring{b}_{I,k}(f_I\cdot k)=\mathring{c}_{I,k}(f_I\cdot k)=\mathring{d}_{I,k}(f_I\cdot k)=0$. Then:
\begin{align}
    \label{e:w0Fou} 
    w_o(t,x)
    &=
    \sum_{m\in\Z}\sum_{k\in \Z^3\setminus\{0\}}b_{m,k}(t,x)e^{i\lambda\qq \xi_m(t,x)\cdot k},
    \\
    \label{e:w0w0}
    (w_o\otimes w_o)(t,x)
    &=
    c_0(t,x)+\sum_{m\in\Z}\sum_{k\in \Z^3\setminus\{0\}}c_{m,k}(t,x)e^{i\lambda\qq \xi_m(t,x)\cdot k},
    \\
    \label{e:w03}
    \left( \frac{|w_o|^2}{2}w_o \right) (t,x)
    &=
    d_0(t,x)+\sum_{m\in\Z}\sum_{k\in \Z^3\setminus\{0\}}d_{m,k}(t,x)e^{i\lambda\qq \xi_m(t,x)\cdot k},
\end{align}
where the coefficients $b_{m,k}, c_{m,k}, d_{m,k}$ are given by
\begin{align}
&\label{e:b_mk}
b_{m,k}(t,x)=\sum_{I\in \mathcal{I}_m}\theta_I(t)\chi_I(\xi_I(t,x))a_I(t,x)\overline{f_I}(t,x)\,\mathring{b}_{I,k},
\\
&\label{e:c_mk}
c_{m,k}(t,x)=\sum_{I\in \mathcal{I}_m}\theta^2_I(t)\chi^2_I(\xi_I(t,x))a_I^2(t,x)\left(\overline{f_I}\otimes \overline{f_I}\right)(t,x)\,\mathring{c}_{I,k},
\\
&\label{e:d_mk}
d_{m,k}(t,x)=\sum_{I\in \mathcal{I}_m}\theta^3_I(t)\chi^3_I(\xi_I(t,x))a_I^3(t,x)\left(\frac{|\overline{f_I}|^2}2\overline{f_I}\right)(t,x)\,\mathring{d}_{I,k},
\end{align}
with
\begin{equation} \label{def:bar_f_I}
    \overline{f_I}(t,x):=(\nabla\xi_I)^{-1}(t,x)f_I,
\end{equation}
whereas the zero modes $c_0,d_0$ are given by
\begin{align}
&\label{e:c0}
c_0(t,x)=\sum_{I\in \mathcal{I}}\theta^2_I(t)\chi_I^2(\xi_I(t,x))a_I^2(t,x)\left(\overline{f_I}\otimes \overline{f_I}\right)(t,x) \fint_{\T^3}\psi_I(t,y)^2dy,
\\
&\label{e:d0}
d_0(t,x)=\sum_{I\in \mathcal{I}}\theta^3_I(t)\chi_I^3(\xi_I(t,x))a_I^3(t,x)\left(\frac{|\overline{f_I}|^2}2\overline{f_I}\right)(t,x) \fint_{\T^3}\psi_I(t,y)^3dy.
\end{align}

At last, we need to define a \emph{compressibility} corrector term $w_c$ such that $\text{div}(w_c+w_o)=0$. Let
\begin{align}
\label{e:wc}
w_c(t,x)
:=&
\frac1{\lambda\qq}\sum_{m\in\Z}\sum_{k\in \Z^3\setminus\{0\}}e_{m,k}(t,x)e^{i\lambda\qq \xi_m(t,x)\cdot k},
\\
\label{e:e_mk}
e_{m,k}(t,x)=&\sum_{I\in \mathcal{I}_m}\theta_I(t)\nabla\left(\chi_I(\xi_I(t,x))a_I(t,x)\right)\mathring{b}_{I,k}\times\left(\nabla\xi_I^T(t,x)\frac{ik\times f_I}{|k|^2}\right).
\end{align}

\subsection{Choice of the amplitudes}
\label{ssec:amplitude}
The amplitudes $a_I$ are built differently based on $I$ being in $\mathcal{I}_R$ or in $\mathcal{I}_{\phi}$.
First of all, consider a temporal mollifier $\eta$ with support in $[-1,0]$, and a time mollification parameter 
\begin{equation}
l_{\rm temp}:=\begin{cases}
\lambda_0^{-1/2}\lambda_1^{-1/2},\,&q=0,\\
\lambda_q^{-1} \delta_q^{-2-\nu}\delta\qq^{3/2+\nu},\,&q\geq 1,
\end{cases}
\end{equation}
and let $\eta_{l_{\rm temp}}$ be the rescaled mollifier $\eta_{l_{\rm temp}}(t) := l^{-1}_{\rm temp} \eta( tl^{-1}_{\rm temp})$.
The asymmetric choice for the support of $\eta$ guarantees that convolutions with $\eta_{l_{\rm temp}}$ preserve $\{\mathbb{F}_t\}_{t \geq 0}$ adaptedness.
Then, define 
\begin{align}\label{e:Mollification}
    R_{l}(t,x)&:=\int_{\R}(\P_{\le l^{-1}}R_q)\left(t+s,\mu(t;t+s,x)\right)\eta_{l_{\rm temp}}(s)ds,\\
    \phi_{l}(t,x)&:=\int_{\R}(\P_{\le l^{-1}}\phi_q)\left(t+s,\mu(t;t+s,x)\right)\eta_{l_{\rm temp}}(s)ds,
\end{align}
where $\mu(t;\cdot,\cdot)$ is the unique solution to the transport equation 
\begin{equation}
\begin{cases}
    \partial_r \mu(t;r,x)=u_l(r,\mu(t;r,x)),
    \quad
    r \geq t,\\
    \mu(t;t,x)=x.
\end{cases}
\end{equation}
Notice that this mollification along the Lagrangian flow implies that
\begin{equation}\label{e:VantaggioMollFlux}
\begin{aligned}
        D_{t,l} R_{l}(t,x)&=-\int_{\R}(\P_{\le l^{-1}}R_q)\left(t+s,\mu(t;t+s,x)\right)\eta_{l_{\rm temp}}'(s)ds,\\
        D_{t,l} \phi_{l}(t,x)&=-\int_{\R}(\P_{\le l^{-1}}\phi_q)\left(t+s,\mu(t;t+s,x)\right)\eta_{l_{\rm temp}}'(s)ds.
\end{aligned}
\end{equation}
 {\emph{First case: $I \in \mathcal{I}_{\phi}$}}.
We define $a_I$ to cancel out $\phi_l$ through $d_0$. 
Indeed, consider
\begin{equation}\label{e:amplPhi}
a_I(t,x):=|\overline{f_I}(t,x)|^{-2/3}2^{1/3}\Lambda_{f_I}^{1/3}\left(-\frac{\nabla\xi_I(t,x)\phi_l(t,x)}{\overline{M}^3\lambda_q^{-3\gamma/2}\delta_{q+1}^{3/2}}\right)\overline{M}(\lambda_q^{-3\gamma/2}\delta_{q+1}^{3/2})^{1/3}. 
\end{equation}
Notice that $\epsilon\|u_l(t)\|_1<1/2$ implies $|\overline{f}_I(t,x)|>1/2$ by \eqref{def:bar_f_I}.
Then, \eqref{e:meanPipesPhi} and \eqref{e:d0} plus the disjoint supports property gives
\begin{equation}
    \begin{aligned}
        d_0
        &=
        \sum_{I\in \mathcal{I_{\phi}}}\theta^3_I\chi_I^3(\xi_I)a_I^3\frac{|\overline{f_I}|^2}2\overline{f_I}
        \\
        &=
        \sum_{(m,n)\in\Z\times \Z^3}\theta_{m}^6(t/\epsilon)\chi_n^6(\xi_m)\nabla\xi_m^{-1}\sum_{\substack{I\in  \mathcal{I}_{\phi}
        \\ I=(m,n,f_I)}}\overline{M}^3\lambda_q^{-\gamma}\delta_{q+1}^{3/2}\Lambda_{f_I}\left(-\frac{\nabla\xi_I\phi_l}{\overline{M}^3\lambda_q^{-\gamma}\delta_{q+1}^{3/2}}\right)f_I
        \\
        &= 
        \sum_{(m,n)\in\Z\times \Z^3}
        \theta_{m}^6(t/\epsilon)
        \chi_n^6(\xi_m)(-\phi_l)=-\phi_l,
    \end{aligned}
\end{equation}
where the second-to-last equality is due to \autoref{GeomLemma2} applied with $N_0=e$, while the last one is due to the space and time partitions of unity.

{\emph{Second case: $I \in \mathcal{I}_{R}$}}.
We define $a_I$ to cancel out $R_l$ through $c_0$. 
Indeed, let $I=(m,n,f)$ and define
\begin{align}\label{e:amplR}
a_I &:= \Gamma_f\left(\Id+(\nabla\xi_I\nabla\xi_I^T-\Id)-\frac{\nabla\xi_I(R_l-M_{m,n})\nabla\xi_I^T}{\rho}\right)\rho^{1/2}, \\
\quad 
\label{eq:definitio_rho}
\rho &:=\frac{C \,\overline{M}^2\lambda_q^{-\gamma}\delta_{q+1}}
{\inf_{j} N_0(\F_R^j)},
\end{align}
where $\Gamma_f$ and $N_0(\F_R^j)$ are given by \autoref{GeomLemma1}, $C>0$ is a constant depending on the $\Gamma_f$'s functions and $\sup_{f\in\F}|f|_\infty$, and
\begin{equation}\label{eq:M_mn}
    M_{m,n}(t,x):=\sum_{\substack{J\in I_{\phi},\\
    J=(m',n',f_J),\\
    |m-m'|_{\infty}\le 1,\\
    |n-n'|_{\infty}\le 1}}a_J^2(t,x)\chi_J^2(\xi_J(t,x))\theta_J^2(t)\left(\overline{f_J}\otimes\overline{f_J}\right)(t,x).
\end{equation}
Then, \eqref{e:meanPipesR} and \eqref{e:w0w0} imply that
\begin{equation}\label{e:cancW0W0}
    \begin{aligned}
        c_0
        &=
        \sum_{(m,n)\in \Z\times\Z^3}\theta_m^6(t/\epsilon)\chi_n^6(\xi_m )
        \left[\sum_{I\in\I_{m,n} \cap \mathcal{I}_R}a_I^2 \nabla\xi_I^{-1} (f_I\otimes f_I) \nabla\xi_I^{-T} +M_{m,n} \right]
        \\
        &=
        \sum_{(m,n)\in \Z\times\Z^3}\theta_m^6(t/\epsilon)\chi_n^6(\xi_m )
        \left[\rho\nabla\xi_I^{-1}\left(\nabla\xi_I\nabla\xi_I^T-\frac{\nabla\xi_I(R_l-M_{m,n})\nabla\xi_I^T}{\rho}\right)\nabla\xi_I^{-T}+M_{m,n}\right]
        \\
        &=
        \sum_{(m,n)\in \Z\times\Z^3}\theta_m^6(t/\epsilon)\chi_n^6(\xi_m )
        \left[\rho\Id-R_l\right]=\rho\Id-R_l,
    \end{aligned}
\end{equation}
where the second to last equality comes from \autoref{GeomLemma1} under the additional assumption 
\begin{align}
    \epsilon\|u_l(t)\|_1\le \frac{\inf_{j} N_0(\F_R^j)}6.
\end{align}

\subsection{Estimates on the velocity perturbation} \label{ssec:estimates_v}
With the construction above, we are able to control the size of the perturbation $w$, as stated in the following key result. Since its proof is highly technical and involved, the reader could skip it at a first reading without losing much insight on the remainder of this work.
In brief, the main takeaway of this result is that taking the advective derivative $D_{t,l}$ effectively ``costs'' a factor $\epsilon^{-1}$, whereas a space derivative ``costs'' a factor $\lambda\qq$.
 
\begin{prop}\label{prop:estw}
    For $s\in\{0,1,2\}$, $N\le n_0+3$ and for every $t\le \mathfrak{t}$, it holds
    \begin{align}  
        &\label{est:w0}\epsilon^{s}\|D_{t,l}^sw_o(t)\|_N\lesssim \lambda\qq^N \rho^{1/2},\\
        &\label{est:wC}\epsilon^{s}\|D_{t,l}^sw_c(t)\|_N\lesssim(\mu\lambda\qq)^{-1} \lambda\qq^N \rho^{1/2},
    \end{align}
    and moreover there exists a finite constant $C_2$, independent on $C_v$, $\overline{M}$ and $L$, such that
    \begin{equation}\label{e:wIndOnM}
        \|w(t)\|_N\le C_2\, \lambda\qq^{N}\rho^{1/2}.
    \end{equation}
In addition, for all $N\le n_0+3$ and $t \le\mathfrak{t}$, we have
    \begin{align}
        &\label{e:DiffU}\|u\qq(t)-u_q(t)\|_N\lesssim \,\lambda\qq^N\rho^{1/2},\\
        &\label{e:DiffUl}\|u\qq(t)-u_l(t)\|_N\lesssim \,\lambda\qq^N\rho^{1/2}.
    \end{align}
\end{prop}

\subsubsection{Proof of \autoref{prop:estw}: Preliminary estimates}
We report some estimates that will be helpful in the proof of \autoref{prop:estw}, as well as in the remainder of the paper.

\begin{lemma}\label{lemma:stocEstimates}
    For any $N\le n_0+4$ and $t \le\mathfrak{t}$, we have
    \begin{align}
        &\label{e:Diffstoc}\|z_l(t)-z\qq(t)\|_N\vee\|z_q(t)-z\qq(t)\|_N\lesssim L\,i_q^{1/2-\delta};\\
        &\label{e:dtStoc}\|\partial_tz_q(t)\|_N\lesssim L \,i_q^{-1/2-\delta}.
    \end{align}
\end{lemma}
\begin{proof}
    The estimate for $\|z_q(t)-z\qq(t)\|_N$ is standard, while the one for $\|z_l(t)-z\qq(t)\|_N$ is due to 
    \[\|z_l(t)-z_q(t)\|_N\lesssim l\|z(t)\|_{N+1},\quad 
    l\le i_q^{1/2-\delta}.\]
    Instead, \eqref{e:dtStoc} follows from noticing that the time mollifier $\alpha$ used in \eqref{e:SmoothNoise} has compact support, implying 
    \[\int_{\R}r'(s)z(t)ds=0,\]
and giving in particular the bound
    \begin{align}
    |\partial_tz_q(t)|
    &=
        \left|\partial_tz_q(t)-i_q^{-2}\int_{\R}r'(s\,i_q^{-1})z(t)\,ds\right|
    \\
    &=i_q^{-2}\left|\int_{\R}r'(s\,i_q^{-1})[z(t-s)-z(t)]\,ds\right|
     \\
    &=i_q^{-1}\left|\int_{\R}r'(s)[z(t-i_qs)-z(t)]\,ds\right| \lesssim Li_q^{-1/2-\delta}.
    \end{align}

\autoref{lemma:stocEstimates} and \eqref{s:estUqN} imply the following estimates, that will be useful later on.
\begin{corollario}\label{cor:Dtq1Z}
    For $N\le n_0+3$ and $t\le \mathfrak{t}$, we have that
    \begin{align}
        \label{est:Dtq1ZMoll}\|D_{t,q+1}(z_q-z_l)(t)\|_N
        &\lesssim
        l\left(i_q^{-1/2-\delta}+(\lambda\qq^N\delta\qq^{1/2}\vee 1)\right),\\
        \label{est:Dtq1ZVar}\|D_{t,q+1}(z\qq-z_q)(t)\|_N
        &\lesssim 
        i\qq^{-1/2-\delta}+i_q^{1/2-\delta}(\lambda\qq^N\delta\qq^{1/2}\vee 1).
    \end{align}
\end{corollario}

\end{proof}

\begin{lemma}  
    For $N\le n_0+1$ and $t \le\mathfrak{t}$, we have
    \begin{align}
        &\label{est: DtlDiffv}\|D_{t,l}(u_q-u_l)(t)\|_N+\|D_{t,l}(v_q-v_l)(t)\|_N\lesssim l^{1-N}\lambda_q^2\delta_q.
    \end{align}
\end{lemma}
\begin{proof}
   To keep the notation light we drop the time $t$ below, keeping in mind that all estimates are uniform in $t \leq \mathfrak{t}$.
   Moving to the proof of the lemma, notice that we have 
    \[D_{t,q}u_q=\partial_tz_q+\mbox{div}(R_q)+\nabla p_q.\]
    Therefore, the bounds \eqref{s:P}, \eqref{s:R}, \cite[Lemma A.3]{DLK23}, $\lambda_q\delta_q^{1/2}\geq L$, $i_q^{-1/2-\delta}\le \lambda_q^2\delta_q$ and $\lambda_q<l^{-1}$ imply
    \begin{equation}
        \begin{aligned}
            \|D_{t,l}(u_q-u_l)\|_N&\le \|\proj_{>l^{-1}}D_{t,l}u_q\|_N+\|[\proj_{>l^{-1}},u_l\cdot \nabla]u_q\|_N\\
                &\lesssim l \|D_{t,q}u_q\|_{N+1}+l\|(u_l-u_q)\cdot \nabla u_q\|_{N+1}+l^{1-N}\|u_q\|_1^2\lesssim
            l^{1-N}\lambda_q^2\delta_q.
        \end{aligned}
    \end{equation}
    The same computations and \autoref{lemma:stocEstimates} yield the result for $\|D_{t,l}(v_q-v_l)\|_N$.
\end{proof}
\begin{lemma}
    For any $N\le n_0+3$ and $t \le\mathfrak{t}$, we have
    \begin{equation}\label{est:DtlNablaU}
        \|D_{t,l}\nabla v_l(t)\|_N+\|D_{t,l}\nabla u_l(t)\|_N\lesssim l^{-N}\lambda^2_q\delta_q.
    \end{equation} 
\end{lemma}

\begin{proof}
    Since it holds
    \[D_{t,l}(\nabla u_l)=\nabla(\partial_tu_l+u_l\cdot\nabla u_l)-(\nabla u_l)^2=\nabla\left(\dvg(R_l+R_{comm})-\nabla p_l+\partial_t z_l\right)-(\nabla u_l)^2,\] 
    where $R_{comm}:=u_l\otimes u_l-\proj_{\le l^{-1}}(u_q\otimes u_q)$ and $(\nabla u_l)^2$ denotes the standard product between matrices $[(\nabla u_l)^2]^{i,j} = \sum_k \partial_i u^k \partial_k u^j$, then by \cite[Lemma A.3]{DLK23}, \autoref{lemma:stocEstimates}, and the bound $L\,i_q^{-1/2-\delta}\le \lambda^2_q\delta_q$ we have:
    \begin{equation}
        \|D_{t,l}\nabla u_l\|_N\lesssim l^{-N}\lambda^2_q\delta_q+L\,i_q^{-1/2-\delta}\lesssim l^{-N}\lambda^2_q\delta_q.
    \end{equation}
    Then, the estimate \eqref{est:DtlNablaU} on $\|D_{t,l}\nabla v_l\|_N$ follows from the identity $D_{t,l}\nabla v_l=D_{t,l}\nabla u_l-D_{t,l}\nabla z_l$ and computations similar to those above.
\end{proof}

\subsubsection{Proof of \autoref{prop:estw}: Estimates on the oscillatory and compressible Fourier coefficients}
Recall the definition of the parameters $\epsilon$ and $\mu$, given respectively in \eqref{eq:definition_eps} and \eqref{eq:definition_mu}.

Let us start with estimates aimed to control the Fourier coefficients appearing in $w$ and in the products $w_o\otimes w_o$ and $|w_o|^2w_o$, cf. Equations \eqref{e:w0Fou}, \eqref{e:w0w0}, \eqref{e:w03}, \eqref{e:wc}. 

The key takeaway of the following lemma is that taking the advective derivative of the building blocks costs a factor $\epsilon^{-1}$, whereas taking a space derivative costs a factor $\mu^{-1}$.

\begin{lemma}\label{lemma:coeff}
    For $s\in\{0,1,2\}$ and $N\le n_0+4$ , it holds for all $t\le \mathfrak{t}$ that
    \begin{align}
        &\epsilon^{s}\|D_{t,l}^sb_{m,k}(t)\|_N\lesssim \sup_{I \in \mathcal{I}}|\mathring{b}_{I,k}|\rho^{1/2}\mu^{-N},\\
        &\epsilon^{s}\|D_{t,l}^sc_{m,k}(t)\|_N\lesssim \sup_{I\in \mathcal{I}}|\mathring{c}_{I,k}|\rho\,\mu^{-N},\\
        &\epsilon^{s}\|D_{t,l}^sd_{m,k}(t)\|_N\lesssim \sup_{I\in \mathcal{I}}|\mathring{d}_{I,k}|\rho^{3/2}\mu^{-N},\\
        &\epsilon^{s}\|D_{t,l}^se_{m,k}(t)\|_N\lesssim \sup_{I\in \mathcal{I}}|\mathring{b}_{I,k}|\rho^{1/2} \mu^{-N-1}.
    \end{align}
\end{lemma}

\begin{proof}
    Let us show the claim for $b_{m,k}$. The estimates on the other coefficients follow from similar inequalities.
     Due to the expression \eqref{e:b_mk} and the spatial partition of unity,  we can focus on a single addend corresponding to $I \in \mathcal{I}_m$ and $k \in \Z^3 \setminus \{0\}$: 
     \begin{equation} \label{eq:aux_001}
         X_I(t,x)
     :=
     \theta_I(t)\chi_I(\xi_I(t,x))a_I(t,x)\overline{f}_I(t,x)\mathring{b}_{I,k}.
     \end{equation} 
     First, we have that
     \begin{equation}\label{e:timePart}
         |D_{t,l}^s\theta_I(t)|\lesssim\epsilon^{-s}.
     \end{equation}
     Second, by \autoref{prop:transport} and \eqref{est:DtlNablaU}, together with the assumed bound $\epsilon\|u_l(t)\|_1<1$ and the algebraic identities 
     \begin{align}\label{e:MatDerFlux}
         &D_{t,l}(\nabla \xi_I)^{-1}=\nabla u_l (\nabla\xi_I)^{-1}, \quad D_{t,l}^2(\nabla\xi_I)^{-1}=D_{t,l}(\nabla u_l)(\nabla\xi_I)^{-1}+(\nabla u_l)^2(\nabla\xi_I)^{-1},
     \end{align}
     imply that for every $K \in \{0,1\dots,N\}$ we have
     \begin{align}\label{e:FluxDirection}
         &\|D_{t,l}^s\overline{f_I}\|_K\lesssim l^{-K} (\lambda_q\delta_q^{1/2})^s, \quad \text{for }s\in \{0,1,2\}.
     \end{align}
     Third, by \autoref{prop:transport}, \eqref{e:flow} and \cite[Proposition C.1]{BD15+}, we have
     \begin{align}\label{e:spacePart}
         \|D_{t,l}^s\chi_I(\xi_I)\|_K\lesssim\mu^{-K}\mathbf{1}_{\{s=0\}}.
     \end{align}  
     Now, we need to estimate $\|D_{t,l}^sa_I\|_K$ for $K \in \{0,1\dots,N\}$, either for $I\in \I_{\phi}$ or $I\in\I_{R}$.
     Before treating each case separately, let us recall some useful inequalities valid for general $C^K$ maps $g:\T^3 \to \R^k$ and $\Lambda:\R^k \to \R^n$, for general $k,n$. First, it is easy to check that the advective derivative of $\Lambda(g):\T^3 \to \R^n$ is given by $D_{t,l}\Lambda(g)= (\nabla \Lambda) (g) \cdot D_{t,l}g$. In particular, by Faà di Bruno's formula we obtain for every $K \in \{0,1\dots,N\}$:
     \begin{align}
         &\label{e:comp1} \|D_{t,l}\Lambda(g)\|_K\lesssim\sum_{K_1+K_2\le K}\|D_{t,l}g\|_{K_1}\|(\nabla\Lambda)(g)\|_{K_2};\\
         &\label{e:comp2} \|D_{t,l}^2\Lambda(g)\|_K\lesssim\sum_{K_1+K_2\le K}\|D_{t,l}^2g\|_{K_1}\|(\nabla\Lambda)(g)\|_{K_2}+\|D_{t,l}g\otimes D_{t,l}g\|_{K_1}\|(\nabla^2\Lambda)(g)\|_{K_2}.
     \end{align}

     \textit{First case: $I\in \I_{\phi}$.}
     Let $\rho_1:=\lambda_q^{-\gamma}\delta_{q+1}^{3/2}$. Recalling the definition of $\rho$ given in \eqref{eq:definitio_rho}, one immediately has $\rho_1 \lesssim \rho^{3/2}$ up to multiplicative constants independent of $\epsilon,\mu$.
      For $s=0$, \cite[Proposition C.1]{BD15+} implies that
      \begin{equation}\label{i:ampPhi0}
          \|a_I\|_K \lesssim l^{-K}\rho_1^{1/3} \lesssim l^{-K}\rho^{1/2}.
      \end{equation}
For $s=1$ instead, we have by \eqref{e:MatDerFlux}, \autoref{prop:transport} and \eqref{e:Mollification}:
      \begin{equation}\label{e:1}
          \|D_{t,l}(\nabla\xi_I\phi_l)\rho_1^{-1/3}\|_K\lesssim l_{\rm temp}^{-1}\,l^{-K}.
      \end{equation}
      On the other hand, \cite[Proposition C.1]{BD15+} implies that for $k \in \{1,2\}$
      \begin{equation}\label{e:2}
          \left\|(\nabla^k \Lambda_{f_I}^{1/3})\left((\nabla\xi_I\phi_l)\rho_1^{-1/3}\right)\right\|_K\lesssim l^{-K},
      \end{equation}
where the implicit constant depends on $\|\nabla^k \Lambda_{f_I}^{1/3}\|_K$.     
Therefore, \eqref{e:comp1}, \eqref{e:1}, \eqref{e:2} imply that
\begin{equation}\label{i:ampPhi1}
\begin{aligned}
           \|D_{t,l}a_I\|_K
           &\lesssim
           \sum_{K_1+K_2\le K}\rho_1^{1/3}\|D_{t,l}(\nabla\xi_I\phi_l)\rho_1^{-1/3}\|_{K_1}\left\|(\nabla \Lambda_{f_I}^{1/3})\left((\nabla\xi_I\phi_l)\rho_1^{-1/3}\right)\right\|_{K_2}
           \\
           &\lesssim
           \rho_1^{1/3}l_{\rm temp}^{-1}l^{-K}
           \lesssim
           \rho^{1/2}l_{\rm temp}^{-1}l^{-K}.
      \end{aligned}
      \end{equation} 
Analogously for $s=2$ we have by \eqref{e:MatDerFlux} and \eqref{e:Mollification} 
      \begin{equation}\label{e:3}
          \|D_{t,l}^2(\nabla \xi_I\phi_j)\rho_1^{-1/3}\|_K \lesssim l_{\rm temp}^{-2}\,l^{-K}
          .
      \end{equation}
Then, \eqref{e:comp2}, \eqref{e:3} and \eqref{e:2} imply that
      \begin{equation}\label{i:amplPhi2}
          \|D_{t,l}^2a_I\|_N \lesssim \rho_1^{1/3}l_{\rm temp}^{-2}\,l^{-N}
          \lesssim \rho^{1/2}l_{\rm temp}^{-2}\,l^{-N}.
      \end{equation}
To summarize, recalling the decomposition \eqref{eq:aux_001}, we can combining all the estimates given in \eqref{e:timePart}, \eqref{e:spacePart}, \eqref{e:FluxDirection}, \eqref{i:ampPhi0}, \eqref{i:ampPhi1}, and \eqref{i:amplPhi2} above and finally obtain
    \begin{equation}\label{i:phiCase}
        \|D_{t,l}^s{X}_I\|_N\lesssim \epsilon^{-s} \sup_{I}|\mathring{b}_{I,k}|\rho^{1/2}\mu^{-N}.
    \end{equation}
      \textit{Second case: $I\in \I_{R}$.}
      The desired estimates follow from the same set of inequalities used in the case $I\in \I_{\phi}$. The main difference, due to \eqref{e:amplR}, is that we now need to estimate the quantity
      \begin{align}
          \left\|D_{t,l}^s\left((\nabla\xi_I\nabla\xi_I^T-\Id)-\frac{\nabla\xi_I(R_l-M_{m,n})\nabla\xi_I^T}{\rho}\right)\right\|_N,
      \end{align}
      where $M_{m,n}$ is given by \eqref{eq:M_mn}.
      Then, \eqref{e:MatDerFlux}, \eqref{e:Mollification}, \eqref{i:phiCase} imply that
      \begin{equation}
          \left\|D_{t,l}^s\left((\nabla\xi_I\nabla\xi_I^T-\Id)-\frac{\nabla\xi_I(R_l-M_{m,n})\nabla\xi_I^T}{\rho}\right)\right\|_N\lesssim \epsilon^{-s} \mu^{-N}.
      \end{equation}
      Let us stress that we have exploited the bound \eqref{i:phiCase} since $M_{m,n}$ is given by a finite summation of $X_J$ for $J \in \I_{\phi}$.
       Therefore,
       \begin{equation}\label{i:RCase}
        \|D_{t,l}^sX_I\|_N\lesssim \epsilon^{-s} \sup_{I}|\mathring{b}_{I,k}|\rho^{1/2}\mu^{-N}.
    \end{equation}
\end{proof}
\begin{oss}\label{oss:IndependenceOnML}
    In particular, the choice of $C_\epsilon$ implies that there exists $C_1$ independent on $C_v$, $\overline{M}$ and $L$ such that
    \begin{equation}
        \|b_{m,k} \|_N+\mu^{-1}\|e_{m,k}\|_N\le C_1\rho^{1/2}\mu^{-N}.
    \end{equation}
We point out that the estimate above depends on $\epsilon$ because of the construction: whenever one has to bound $\xi_m$, terms of the form $\epsilon\|u_q\|_1l^{-j}$ appear; therefore the constant $C_\epsilon$ permits to compensate for the dependence of these terms on the parameters $C_v$ and $\overline{M}$ implicit in the quantity $\|v_q\|_1\le C_v \oM\lambda_q\delta_q^{1/2}$.
\end{oss}

\subsubsection{Proof of \autoref{prop:estw}: Conclusion}
Let us start from the estimates on $w$.
It is enough to show the result for $w_o$, as the claim for $w_c$ follows by similar arguments.
Notice that 
 \begin{align}
     D_{t,l}^s w_o (t,x)
     =
     \sum_{m\in\Z}\sum_{k\in \Z^3\setminus\{0\}}D_{t,l}^s b_{m,k} (t,x)e^{i\lambda_{q+1} \xi_m(t,x)\cdot k},
 \end{align}
where the sum over $m \in \Z$ is effectively a finite sum because the coefficients $b_{m,k}$ have disjoint supports for $|m-m'|>1$. Furthermore, for every fixed $m \in \Z$ we have  
\begin{equation}\label{i:FourierBasisFlux}
    \|e^{i \lambda\qq \xi_m \cdot k}\|_N\lesssim |k|^N \lambda^N_{ q+1},
\end{equation}
and thus the desired inequality follows from  \autoref{lemma:coeff} and the analog of \cite[Equation (3.24)]{DLK23}.

Moreover, the estimates \eqref{est:w0}, \eqref{est:wC}, \eqref{e:Diffstoc}, \eqref{s:V}, $i_q^{1/2-\delta}\le \rho^{1/2}$ and $l\lambda_q\delta_q^{1/2}\le \rho^{1/2}$ imply the bounds \eqref{e:DiffU} and \eqref{e:DiffUl} on $u$. 

\section{Construction of the new Reynolds stress and pressure} \label{sec:new_reynolds}
Based on the previous section, we have defined $v\qq=v_q+w_o+w_c$. Now, we focus on identifying the new Reynolds stress $R\qq$ and pressure $ p\qq$.
Using $v\qq=v_q+w_o+w_c$ and the equation $\partial_tv_q+\dvg\big(u_q\otimes u_q\big)+\nabla p_q=\dvg\big(R_q\big)$, 
we obtain that any putative $R\qq$ and $p\qq$ satisfying the momentum equation in \eqref{eq:ApproxEulerV} must necessarily satisfy also
\begin{equation}
    \begin{aligned}
        -\nabla p\qq+\dvg(R\qq) 
        &=
        \partial_tv\qq+\dvg\big(u\qq\otimes u\qq\big)
        \\
        &=
        \partial_t w+ \dvg(w_o\otimes w_o+R_l)-\nabla p_q+\dvg(w_c\otimes w+w_o\otimes w_c)
        \\
        &\quad+
        u_l\cdot \nabla w+w\cdot \nabla u_l+\dvg\big((v_q-v_l)\SymOt w+R_q-R_l\big)
        \\
        &\quad+
        \dvg\big(w\SymOt (z\qq-z_l)+u_q\SymOt \varz+\varz\otimes \varz\big),
    \end{aligned}
\end{equation}
where $\SymOt$ is the symmetric tensor product, $\varz:=z\qq-z_q$ and $z_l$ is given by \eqref{e:u_l}.
Now, we can exploit \eqref{e:w0w0}  and \eqref{e:cancW0W0}  to get (hereafter the symbol $\mathring{\odot}$ stands for the traceless symmetric product)
\begin{equation} \label{eq:decomposition_Rqq}
    \begin{aligned}
        -\nabla p\qq+\dvg(R\qq)
        &=
        \partial_t w+u_l\cdot \nabla w+\dvg\Big(w_o\otimes w_o-c_0\Big)+w\cdot \nabla u_l
        \\
        &\quad+
        \dvg\Big(w_c\otimes w+w_o\otimes w_c+w\SymOt (v_q-v_l)+R_q-R_l\Big)
        \\
        &\quad+ 
        \dvg\Big(w \SymOtNoTr (z\qq-z_l)+u_q\SymOtNoTr\varz+\varz\mathring{\otimes} \varz\Big)
        \\
        &\quad+
        \nabla \Big(\frac{2}{3}w\cdot (z\qq-z_l)+\frac{2}{3}u_q\cdot \varz +\frac{|\varz|^2}{3}\Big)-\nabla p_q.
    \end{aligned}
\end{equation}
Notice that the last line in the expression above is a gradient; hence, we define:
\begin{align}
  \label{e:varp} \varp:=p\qq-p_{q}:=-\Big(\frac{2}{3}w\cdot (z\qq-z_l)+\frac{2}{3}u_q\cdot \varz +\frac{|\varz|^2}{3}\Big).
\end{align}
Therefore, in order to enforce \eqref{eq:decomposition_Rqq} we can define $R_{q+1}$ as follows: 
\begin{align}\label{e:Rdec}
       R\qq
       &:=
       R_A+R_T+R_S+\frac{2}{3}f(t)\Id,
       \end{align}
where $f$ is an arbitrary space-independent scalar function and 
       \begin{align}\label{e:RdecA}
       R_A&:=\AntiDiv\Big[ \partial_t w+u_l\cdot \nabla w+w\cdot \nabla u_l +\dvg\Big(w_o\otimes w_o-c_0\Big) \Big];
       \\
       \label{e:RdecT}
       R_T&:= w_c\otimes w+w_o\otimes w_c+w\SymOt (v_q-v_l)+R_q-R_l ;
       \\
       \label{e:RdecS}
       R_S&:= w\SymOtNoTr (z\qq-z_l)+u_q\SymOtNoTr \varz+\varz\mathring{\otimes} \varz .
\end{align}
In \eqref{e:RdecA} above, the operator $\mathcal{R}$ denotes the anti-divergence operator (cf. for instance \cite[Definition 4.2]{DLS13}), which to a given smooth periodic vector field $v$ associates a field $\mathcal{R}v$ taking values in the space of symmetric and traceless matrices and satisfying $\dvg \mathcal{R}v = v - \fint_{\T^3}v$. 
In fact, the \emph{antidivergence stress} $R_A$ denotes the part of the Reynolds stress $R\qq$ containing all the terms that we will estimate through properties of the the antidivergence operator $\mathcal{R}$.
Similarly, the \emph{trace stress} $R_T$ contains all the contributions of terms having trace (except for the term $\frac{2}{3}f(t)\Id$), and the \emph{stochastic stress}  $R_S$ contains all the terms with an explicit dependence of the stochastic process $z$.
Our particular $f$ entering in the convex integration scheme will be defined by \eqref{e:f} below. 
It will depend on the flux current $\phi$ and it will serve the purpose of canceling out the space average of several terms appearing in the equation \eqref{e:NewCurrent1} for the new current $\phi\qq$.  
In passing, we record the following expression for the trace of $R_{q+1}$, that we shall use later:
\begin{align}
   &\label{e:TraceR}\Tr\big(R\qq\big)=\Tr\big(R_q-R_l\big)+|w_c|^2+2w_c\cdot w_o+2(v_q-v_l)\cdot w+2f.
\end{align}

\subsection{Estimates on the Reynolds stress}
\label{ssec:estimatesR}
We denote $D_{t,q+1} := \partial_t +u\qq \cdot \nabla$. Our main proposition on the Reynolds stress is the following: morally speaking, taking either a space derivative or an advective derivative of the new Reynolds stress $R\qq$ ``costs'' a factor $\lambda\qq$.

\begin{prop}\label{prop:R}
There exists $ b_0:=b_0(\alpha)>1$ such that, for any $b\in (1,b_0)$, there are parameters $\gamma:=(b-1)^2$,  $a_0>1$ and $n_0,h_0\in \N$ sufficiently large with the following property. 
For every $a>a_0$, $t\le \mathfrak{t}$, and $N\le n_0+3$ the following estimates hold:
    \begin{align}
        \left\|\left(R\qq(t)-\frac23f(t)\Id\right)w(t) \right\|_N&\label{est:Rw1}
        \le 
        \lambda\qq ^{N-2\gamma}\delta_{q+2}^{3/2},
        \\
        \left\|D_{t,q+1}\left(R\qq(t)-\frac23f(t)\Id \right)w(t)\right\|_{N-2}&\label{est:Rw2}
        \le 
        \lambda\qq ^{N-1-2\gamma}\delta_{q+2}^{3/2}\delta\qq^{1/2},
        \\
    \left\|R\qq(t)-\frac23f(t)\Id\right\|_N&\label{est:R1}
    \le 
    \lambda\qq ^{N-2\gamma}\delta_{q+2},\\
        \left\|D_{t,q+1}\left(R\qq(t)-\frac23f(t)\Id\,\right)\right\|_{N-2}&\label{est:R2}
        \le 
        \lambda\qq ^{N-1-2\gamma}\delta_{q+2}\delta\qq^{1/2}.
\end{align}
\end{prop}

\begin{oss}
The parameters $n_0,h_0$ of previous proposition are given by: 
\begin{align}
    n_0&:=\left\lceil\frac{24b}{b-1}\right\rceil-1,
    \quad
    h_0:=\left\lceil \frac{8b}{b-1}\right\rceil.
\end{align}
They will play a crucial role when perfoming estimates in which we need to apply the antidivergence operator $\mathcal{R}$, more specifically in the bounds \eqref{e:Antidiv1}, \eqref{e:Antidiv2}, \eqref{e:Antidiv3} from the Appendix.
The values of $n_0,h_0$ ensure that
\begin{align}\label{est:n0h0}
    \frac{\rho}{\mu}\lambda_q\delta_q^{1/2}(\mu\lambda\qq)^{-n_0-1}
    &\lesssim \lambda\qq^{-2\gamma}\delta\qqq^{3/2},
    \\
\frac{\rho}{\epsilon}\lambda_q\delta_q^{1/2}(l\lambda\qq)^{-h_0} 
&\lesssim 
\lambda\qq^{-1-2\gamma}\delta\qqq^{3/2}\delta\qq^{1/2}.
\end{align}
\end{oss}

\subsubsection{Proof of \autoref{prop:R}: bounds on $R_T,R_A,R_S$.}
As a first step towards proving \autoref{prop:R}, we separately give bounds on the stress tensors $R_T,R_A,R_S$ appearing in the decomposition of the new Reynolds stress $R\qq$.
\begin{lemma}\label{lemma:RT}
Let $R_T$ be defined by \eqref{e:RdecT}.
    For $N\le n_0+3$, it holds for any $t\le \mathfrak{t}$ that
    \begin{align}
        \|R_T(t)\|_N
        &\lesssim 
        l^{1-N}\lambda_q^{1-\gamma}\delta\qq(l^{-1}\delta_q^{1/2}l_{\rm temp}+1)
        \\
        &\quad+\lambda\qq^N\rho^{1/2}\left(\frac{\rho^{1/2}\mu^{-1}}{\lambda\qq}+l\lambda_q\delta_q^{1/2}\right) ;
        \\
        \|D_{t,q+1}R_T(t)\|_{N-2}
        &\lesssim
        \lambda_q^{N-1-\gamma}\delta\qq\delta_q^{1/2}+ \lambda\qq^{N-2}\lambda_q^{1-\gamma}\delta\qq \rho^{1/2}+l^{1-N}\lambda_q^{-\gamma}\delta\qq l_{\rm temp}^{-1}\\
        &\quad+\lambda\qq^{N-2}\rho^{1/2}\left[\lambda\qq\rho^{1/2}\left(\frac{\rho^{1/2}\mu^{-1}}{\lambda\qq}+l\lambda_q\delta_q^{1/2}\right)\right]
        .
    \end{align}
\end{lemma}

\begin{proof}
   We start by estimating $R_q-R_l 
   =(R_q-\proj_{\le l^{-1}}R_q)+(\proj_{\le l^{-1}}R_q-R_l).$
   Notice that \eqref{e:Mollification} implies that
   \begin{equation} \label{eq:aux002}
       (\proj_{\le l^{-1}}R_q-R_l)(t,x)=\int_{-l_{\rm temp}}^0ds\,\eta_{l_{\rm temp}}(s)\int_s^0(D_{t,l}\proj_{\le l^{-1}}R_q)(t+r,\mu(t;t+r,x))dr.
   \end{equation}
Moreover, due to \eqref{s:V}, \eqref{s:R}, and \cite[Lemma A.4]{DLK23}, we have the following inequality:
\begin{align}\label{est:MollSpaceR}
           \|D_{t,l}\proj_{\le l^{-1}}R_q\|_N&\le \|\proj_{\le l^{-1}}(D_{t,q}R_q)\|_N+\|[u_l\cdot \nabla, \proj_{\le l^{-1}}]R_q\|_N+\|\proj_{\le l^{-1}}((u_q-u_l)\cdot \nabla R_q)\|_N\\
           &\lesssim l^{-N}\lambda_q^{1-\gamma}\delta\qq\delta_q^{1/2}+l^{1-N}\|u_q\|_1\lambda^{1-\gamma}_q\delta\qq \lesssim l^{-N}\lambda_q^{1-\gamma}\delta\qq \delta_q^{1/2}.
       \end{align}

   Therefore, by \eqref{eq:aux002}, \eqref{est:MollSpaceR}, \cite[Proposition C.1]{BD15+}, and \eqref{e:2ndTransport} we obtain the bound
   \begin{equation}\label{est:RT1}
       \|\proj_{\le l^{-1}}R_q-R_l\|_N\lesssim l^{-N}\lambda_q^{1-\gamma}\delta\qq \delta_q^{1/2} l_{\rm temp}.
   \end{equation}

On the other hand,  Bernstein inequality and $\lambda_q<l^{-1}$ imply that
   \begin{equation}\label{est:RT2}
   \begin{aligned}
      \| R_q-\proj_{\le l^{-1}}R_q\|_N
      &\lesssim 
      \begin{cases}
          l\|R_q\|_{1}, &\mbox{if } N=0
          \\
           \lambda_q^{N-\gamma}\delta\qq, &\mbox{if } N\geq1
      \end{cases}
      \\
      &\lesssim 
      \begin{cases}
          l \lambda_q^{1-\gamma }\delta\qq, &\mbox{if }N=0
          \\
          l^{1-N}\lambda_q^{1-\gamma}\delta\qq, &\mbox{if }N\geq 1
      \end{cases}\\
      &\lesssim l^{1-N}\lambda_q^{1-\gamma}\delta\qq.
    \end{aligned}
   \end{equation}
As for the material derivative
\begin{equation}\label{est:DRT1}
\begin{aligned}
    \|D_{t,q+1}(R_q-R_l)\|_{N-2}&\le \|D_{t,q+1}R_q\|_{N-2}+\|D_{t,q+1}R_l\|_{N-2}\\
    &\le \|D_{t,q}R_q \|_{N-2}+\|(u\qq-u_q)\cdot \nabla R_q \|_{N-2}\\
    &+\|D_{t,l}R_l \|_{N-2}+\|(u\qq-u_l)\cdot \nabla R_l \|_{N-2}\\
    &\lesssim \lambda_q^{N-1-\gamma}\delta\qq\delta_q^{1/2}+ \lambda\qq^{N-2}\lambda_q^{1-\gamma}\delta\qq \rho^{1/2}+l^{2-N}\lambda_q^{-\gamma}\delta\qq l_{\rm temp}^{-1},
\end{aligned}
\end{equation}
where the last inequality is due to \eqref{s:R}, \eqref{e:DiffU}, \eqref{e:DiffUl}, the bounds $l<l_{\rm temp}$ and $l_{\rm temp}\|u_q\|_1\le 1$, as well as the inequality $\|R_l \|_N\lesssim l^{-N}\lambda_q^{-\gamma}\delta\qq$ valid for $N \geq 0$.

Now we estimate $R_T':=R_T-(R_q-R_l)=w_c\otimes w+w_o\otimes w_c+w\SymOt (v_q-v_l)$. The bounds \eqref{est:w0}, \eqref{est:wC}, and \eqref{s:V}
imply
\begin{equation}\label{est:RT3}
    \|R_T' \|_N\lesssim \lambda\qq^N\rho^{1/2}\left(\frac{\rho^{1/2}\mu^{-1}}{\lambda\qq}+l\lambda_q\delta_q^{1/2}\right).
\end{equation}
Instead, the material derivative is controlled thanks to bounds \eqref{est:w0}, \eqref{est:wC}, \eqref{est: DtlDiffv}, \eqref{e:DiffUl}, as well as  $l\lambda_q\delta_q^{1/2}\le \rho^{1/2}$ and $\epsilon^{-1} \lesssim \lambda\qq\rho^{1/2}$, giving
\begin{equation}\label{est:DRT2}
    \begin{aligned}
        \|D_{t,q+1}R_T'\|_{N-2}&\le \|D_{t,l}R_T'\|_{N-1}+\|(u\qq-u_l)\cdot \nabla R_T'\|_{N-2}\\
        &\lesssim \lambda\qq^{N-2}\rho^{1/2}\left[\lambda\qq\rho^{1/2}\left(\frac{\rho^{1/2}\mu^{-1}}{\lambda\qq}+l\lambda_q\delta_q^{1/2}\right)\right].
    \end{aligned}
\end{equation}
\end{proof}

\begin{lemma}\label{lemma:RA}
    Let $R_A$ be defined by \eqref{e:RdecA}. For $N\le n_0+3$ and $t\le \mathfrak{t}$ we have
    \begin{align}
    \|R_A(t)\|_N
    &\lesssim \lambda\qq^N\left[\lambda\qq^{-1}+(\lambda\qq\mu)^{-n_0-1}\right]\frac{\rho}{\mu},\\
    \|D_{t,q+1}R_A(t)\|_{N-2}
    &\lesssim \lambda\qq^{N-2}\left[(\epsilon\lambda\qq)^{-1}+\lambda\qq^{-n_0}\mu^{-n_0-1}+(l\lambda\qq)^{-h_0}\lambda\qq^2\lambda_q\delta_q^{1/2}\right]\frac{\rho}{\mu}
    \end{align}
\end{lemma}

\begin{proof}
Let us preliminarily observe that 
\begin{align}
   R_A &= \mathcal{R}\left( \sum_{k \in \Z^3 \setminus \{0\}} \sum_{m \in \Z} a_{m,k} e^{i\lambda\qq \xi_m \cdot k}\right),
   \\
   a_{m,k}: &= \dvg(c_{m,k})+D_{t,l}(b_{m,k}+\lambda\qq^{-1}e_{m,k})+(b_{m,k}+\lambda\qq^{-1}e_{m,k})\cdot\nabla u_l.
\end{align}   
For the estimate on $\|R_A \|_N$ we intend to apply \autoref{prop:antidiv1} with $F=R_A$. Indeed, thanks to \autoref{lemma:coeff} and the bound $\epsilon\|u_q\|\leq1$ one can check that the assumptions of the proposition are satisfied with $\lambda=\lambda\qq$, $s_j = \frac{\rho}{\mu} \mu^{-j}$, and with numbers $\mathring{a}_k$ obtained from the constants in the inequalities stated in \autoref{lemma:coeff} and $a_F$ depending only on finitely many derivatives of the  Mikado flows $\psi_I$, in particular independent of $\lambda\qq,\mu,\rho$ et cetera.
Similarly, the estimate on $\|D_{t,q+1}R_A \|_{N-2}$ follows from \autoref{prop:antidiv2}, where the assumptions therein hold true with $c=\epsilon$ because of \autoref{lemma:coeff} and 
    \begin{align}\label{est:DtlRA}
        \|D_{t,l}\dvg(c_{m,k})\|_j &\le \|D_{t,l}c_{m,k}\|_{j+1}+\|\nabla u_l:\nabla c_{m,k}^T\|_j\lesssim \mu^{-j}\frac{\rho}{\mu\,\epsilon},
        \\
        \left\|D_{t,l}\left[(b_{m,k}+\lambda\qq^{-1}e_{m,k})\cdot\nabla u_l\right]\right\|_j 
        &\lesssim
        \mu^{-j}\rho^{1/2}(\epsilon^{-1}\lambda_q\delta_q^{1/2}+\lambda_q^2\delta_q).
    \end{align}
\end{proof}

\begin{lemma}\label{lemma:RS}
    Let $R_S$ be defined by \eqref{e:RdecS}. For $N\le n_0+3$ and $t\le \mathfrak{t}$ we have
    \begin{align}\label{est:RS}
        \|R_S(t)\|_N &\lesssim i_q^{1/2-\delta}\begin{cases}
            1,\,\quad&N=0\\
            \lambda\qq^N\delta\qq^{1/2},\quad&N>0
        \end{cases},\\
        \label{est:RS2} 
            \|D_{t,q+1}R_S(t)\|_{N-2}&\lesssim i\qq^{-1/2-\delta}
            \begin{cases}
                1\vee\delta\qq^{1/2},\quad&N=2\\
                \lambda\qq^{N-2}\delta\qq^{1/2},\quad &N>2
            \end{cases}
            \\
            &\quad
            +i_q^{1/2-\delta}\left[\lambda\qq\rho(\delta\qq\vee\delta\qq^{1/2})\right]\lambda\qq^{N-2}.
    \end{align}
\end{lemma}
\begin{proof}
    We recall that, by definition:
    \begin{align}
      R_S=w\SymOtNoTr (z\qq-z_l)+u_q\SymOtNoTr \varz+\varz\mathring{\otimes} \varz,
      \quad
      \varz := z\qq-z_q.  
    \end{align}
    Therefore, the desired bounds follow from \autoref{lemma:coeff},\autoref{prop:estw}, \autoref{cor:Dtq1Z}, as well as 
    \begin{equation}\label{est:Dtq1uq}
        \|D_{t,q+1}u_q\|_N\le \|D_{t,q}u_q\|_N+\|(u\qq-u_q)\cdot \nabla u_q\|_N\lesssim i_q^{-1/2-\delta}+\lambda\qq^N\lambda_q\,\delta_q.
    \end{equation}
\end{proof}

\subsubsection{Proof of \autoref{prop:R}: Conclusion}
We are ready to show \autoref{prop:R}.
First, bounds \eqref{est:R1} and \eqref{est:R2} follow directly from previous \autoref{lemma:RT}, \autoref{lemma:RA}, \autoref{lemma:RS}, after some some standard computations.
Instead, the bounds \eqref{est:Rw1}  and \eqref{est:Rw2} are due to \autoref{lemma:RT}, \autoref{lemma:RA}, \autoref{lemma:RS} and \autoref{prop:estw} combined with
\begin{align}
   \|Xw \|_N\lesssim\sum_{j=0}^N\|X \|_j\|w \|_{N-j} 
\end{align}
applied with either $X(t,x) :=R\qq(t,x)-\frac23f(t)\Id$ or $X(t,x):=D_{t,q+1}\left(R\qq(t,x)-\frac23f(t)\Id\right).$

\subsection{Estimates on the pressure}
Our next proposition concerns the new pressure $p\qq$. We have the following:
\begin{prop} \label{lemma:pressure}
There exists $ b_0:=b_0(\alpha)>1$ such that, for any $b\in (1,b_0)$, there is $a_0>1$ such that for every $a>a_0$, $t\le \mathfrak{t}$ the following estimates hold:
    \begin{align}
        \label{est:pq1}\|p\qq(t)\|_N
        &\le \overline{M}\lambda\qq^N\delta\qq,
        \quad 
        \|D_{t,q+1}p\qq(t)\|_{N-2}
        \le \overline{M} \lambda\qq^{N-1}\delta\qq^{3/2},
        &\text{  for }N \in \{1,2,\dots,n_0+3\},\\
        \label{est:varp}\|p\qq(t)-p_q(t)\|_N&\le \overline{M}\lambda\qq^{N}\delta\qq &\text{for }N \in \{0,1\}.
    \end{align}
\end{prop}

\begin{proof}    
    Since $p\qq-p_q=-\Big(\frac{2}{3}w\cdot (z\qq-z_l)+\frac{2}{3}u_q\cdot \varz +\frac{|\varz|^2}{3}\Big)$ by construction, the estimates on $\|p\qq\|_N$ and $\|p\qq-p_q\|_N$ can be easily obtained by mimicking those on $R_S$.
    Analogously,
    \begin{align}
        \|D_{t,q+1}p\qq\|_{N-2}&\lesssim \lambda\qq^{N-1}\delta\qqq\delta\qq^{1/2}+\|D_{t,q}p_q\|_{N-2}+\|(u\qq-u_q)\cdot \nabla p_q\|_{N-2}\\
        &\lesssim \lambda\qq^{N-1}\delta\qqq\delta\qq^{1/2}+\lambda_q^{N-1}\delta_q^{3/2}+\lambda\qq^{N-2}\lambda_q\delta_q\delta\qq^{1/2}\le \overline{M}\lambda\qq^{N-1}\delta\qq^{3/2},
    \end{align}
    where the last inequality is due to picking $b_0$ and $a_0$ properly.
\end{proof}

\section{Construction of the new current} \label{sec:current}
In the previous sections we have constructed the new velocity $v\qq$, pressure $p\qq$, and Reynolds stress $R\qq$. Moreover, the noise $z\qq$ is given by \eqref{e:SmoothNoise}.
The new current $\phi\qq$ is defined by imposing that $(v\qq,p\qq,R\qq,\phi_{q+1},z_{q+1})$ satisfies the relaxed local energy equality in \eqref{eq:ApproxEulerV}, namely
\begin{align}\label{e:LEIq1}
      \partial_t&\frac{|v\qq|^2}{2}
      +
      v\qq\cdot \nabla p\qq
      +
      v\qq \cdot\dvg\big((v\qq+z\qq)\otimes (v\qq+z\qq)\big) = -E'
      \\
      &
      +
      \frac{1}{2}D_{t,q+1}\big(\Tr(R\qq)\big)
      +
      \dvg\big(R\qq v\qq\big)
      +
      R\qq:\nabla z\qq^T 
      +
      \dvg (\phi\qq).
\end{align}

Since $(v_q,p_q, R_q,\phi_q,z_q)$ solves itself the relaxed local energy inequality, we have
    \begin{align}\label{e:NewCurrent1}
        \partial_t&\frac{|v\qq|^2}{2}
        +
        v\qq\cdot \nabla p\qq
        +
        v\qq \cdot\dvg\big(u\qq\otimes u\qq\big) 
        =
        -E'
        \\
        &+\dvg(R\qq v\qq)+R\qq:\nabla z\qq^T
        \\
        &+D_{t,q}\left(\frac{\Tr(R_q)+|w|^2+2(v_q-v_l)\cdot w}{2}\right)+\dvg\left(\phi_q+\frac{|w|^2}{2}w-R\qq w+(\varz\otimes v_q)w\right)
        \\
        &+\dvg\left((R_q+w \otimes w -R\qq)(v_q-v_l)+\left[\varz \otimes \varz+u_q\SymOt \varz\right]v_l
    \right)\\
    &+\dvg\left(\left(\frac{|v_q|^2+|w|^2}{2}\right)\varz+\frac{|v_q-v_l|^2}{2}w\right)
        \\
        &+\left((w_o \otimes w_o-c_0)-R_A\right):\nabla(v_l+z_q)^T-R\qq:\nabla\varz^T
        \\
        &-\left(w\SymOtNoTr(z\qq-z_l)+u_q\SymOtNoTr \varz+\varz \otimes \varz\right):\nabla z_q^T+v_q\cdot\dvg(u_q\otimes \varz+\varz\otimes z\qq)
        \\
        &-(z\qq-z_l)\cdot \dvg(w\otimes v_l)+(v_q-v_l)\cdot\dvg(w\otimes \varz )+w\cdot\dvg\Big\{(p\qq-p_l)\Id+\P_{\leq l^{-1}}R_q
        \\
        &+R_{comm}-z\otimes v_l+(z\qq-z_l)\otimes z_q+ z_l\otimes (z_q-z_l)-v_l\otimes(z_l-z_q)+u\qq\otimes \varz\Big\},
    \end{align} 
       where $R_{comm}:=u_l\otimes u_l-\proj_{\le l^{-1}}(u_q\otimes u_q)$, and we have used multiple times the divergence-free property of all the vector fields at stake, as well as the following standard algebraic identities valid for generic matrices $M$ and vector fields $v,u$:
\begin{align}
  \dvg(Mv)&=\dvg(M)\cdot v+M:\nabla v^T, 
  \\
  u\otimes v:\nabla w^T&=v\cdot \dvg(u\otimes w).
\end{align}
Equation \eqref{e:TraceR} implies that
\begin{align}
    D_{t,q}\Big(\frac{\Tr(R_q)+|w|^2+2(v_q-v_l)\cdot w}{2}\Big)&=D_{t,q}\left(\frac{\Tr(R\qq)+\Tr(w_o \otimes w_o-c_0)}{2}\right)-\partial_t f\\
    &=D_{t,q+1}\left(\frac{\Tr(R\qq)}{2}\right)-\dvg\left((w+\varz)\frac{\Tr(R\qq)}{2}\right)\\
    &\quad+\dvg\left(\AntiDiv\left(D_{t,l}\frac{\Tr(w_o \otimes w_o-c_0)}{2}\right)\right)
    \\
    &\quad+\fint_{\T^3} D_{t,l}\frac{\Tr(w_o \otimes w_o-c_0)}{2}\,dx\\
    &\quad+\dvg\left( (v_q-v_l)\frac{\Tr(w_o \otimes w_o-c_0)}{2}\right)-\partial_t f.
\end{align}
Therefore, \eqref{e:LEIq1} holds for $\phi\qq$ and $f(t)$ given by:
\begin{align}
    \phi\qq&:=\phi_A+ \phi_{D_1}+\phi_{D_2}+\phi_{S_1}+\phi_{S_2},
    \\
    \label{e:f}
    f_{q+1}(t)
    &:=
    f_A(t)+f_{D_2}(t)+f_{S_2}(t)
    :=
    \int_0^t\fint_{\T^3} \left( \varphi_A(s,x)+\varphi_{D_2}(s,x)+\varphi_{S_2}(s,x) \right)\,dx\,ds,
    \end{align}
    where the terms at the right-hand sides above are defined as follows: the \emph{antidivergence current}
    \begin{align}
    \label{e:phiA}
\phi_A
&:=
\overline{\mathcal{R}} \varphi_A,
\\
\varphi_A  \label{e:varphiA}
&:= 
D_{t,l}\frac{\Tr(w_o \otimes w_o-c_0)}{2}
+
\dvg\left(\phi_l+\frac{|w_o|^2}{2}w_o\right)
+
\left(w_o \otimes w_o-c_0\right):\nabla(v_l+z_q)^T
\\
    &\quad+
    w\SymOtNoTr(z\qq-z_l):\nabla z_q^T-(z\qq-z_l)\cdot\dvg(w\otimes v_l)+(v_q-v_l)\cdot \dvg(w\otimes \varz) 
    \\
    &\quad+
    w\cdot\dvg\left(\proj_{\le l^{-1}}R_q+R_{comm}-\varz\otimes v_l+(z\qq-z_l)\otimes z_q\right)
    \\
    &\quad+
    w\cdot \dvg(u_l\otimes(z_q-z_l)+u\qq\otimes \varz)+w\cdot \nabla(p_q-p_l),
    \end{align}
contains all the terms that will be treated by means of the properties of the \emph{antidivergence} operator $\overline{\mathcal{R}} := -\nabla (-\Delta)^{-1}$, that to smooth scalar fields $g$ associates a vector field satisfying $\dvg \overline{\mathcal{R}} g = g - \fint_{\T^3}g$; 
the \emph{stochastic currents} 
\begin{align}
    \label{e:phiS1}\phi_{S_1}&:=\left(\frac{\Tr(R\qq-\frac32f\,\Id)}{2}+\frac{|v_q|^2+|w|^2}{2}\right)\varz+(\varz\otimes v_q)w+(\varz \otimes \varz+u_q\SymOt\varz)v_l+\varp \,w\,,
    \\
    \label{e:phiS2} 
        \phi_{S_2}
        &:=
\overline{\mathcal{R}} \varphi_{S_2},
\\
\varphi_{S_2}  \label{e:varphiS2}
&:=
R_A:\nabla (z_q-z_l)^T+\left(u_q\SymOtNoTr\varz+\varz \otimes \varz\right):\nabla z_q^T
\\
&\quad+
v_q\cdot \dvg(u_q\otimes \varz+\varz\otimes z\qq)-\left(R\qq-\frac32 f\,\Id\right):\nabla\varz^T,
    \end{align}
containing all the terms with explicit dependence of the noise $z$; and the \emph{deterministic currents}
    \begin{align}
    \label{e:phiD1} 
        \phi_{D_1}&:=w\frac{\Tr(R\qq-\frac{3}{2}f\,\Id)}{2}+(v_q-v_l)\frac{\Tr(w_o \otimes w_o-c_0)}{2}+\phi_q-\phi_l\\
        &+\frac{|w|^2}{2}w-\frac{|w_o|^2}{2}w_o-\left(R\qq-\frac{3}{2}f\,\Id\right)w+\frac{|v_q-v_l|^2}{2}w\\
        &+\left(R_q-w \otimes w-R\qq-\rho\Id+\frac32f\,\Id\right)(v_q-v_l)\,,
    \\
\label{e:phiD2}
\phi_{D_2}
&:=
\overline{\mathcal{R}} \varphi_{D_2},
\quad
\varphi_{D_2}
:=
R_A:\nabla u_l^T ,
    \end{align}
containing all the remaining terms that would be there even without external random forcing.
Obviously, if $z \neq 0$ then the deterministic currents $\phi_{D_1}$ and $\phi_{D_2}$ are random objects nonetheless, being so all the other objects involved in the construction.
Furthermore, the terms $\phi_{D_1}$ and $\phi_{S_1}$ can be written in the form $\overline{\mathcal{R}} \varphi_{D_1}$ and $\overline{\mathcal{R}} \varphi_{S_1}$ respectively, where $\varphi_{D_1} = \dvg \phi_{D_1}$ and $ \varphi_{S_1} = \dvg \phi_{S_2}$ have zero space average. On the other hand, the currents $\varphi_A, \varphi_{S_2}$, and $\varphi_{D_2}$ have non-zero space average. 

\subsection{Estimates on the current}
\label{ssec:estimatesPhi}
Our main proposition on the current is the following \autoref{prop:phi}. The same caveat as before \autoref{prop:estw} applies also here: the proof of this results could be skipped at a first reading. 
Roughly speaking, taking either a space derivative or an advective derivative of the new current $\phi\qq$ ``costs'' a factor $\lambda\qq$. On the other hand, the scalar function $f$ and its time derivative go to zero as $\lambda\qq^{-\gamma}\delta\qqq$ and $\lambda\qq^{-\gamma}\delta\qqq\delta\qq^{1/2}$, respectively.

\begin{prop} \label{prop:phi}
   There exists $ b_0:=b_0(\alpha)>1$ such that, for any $b\in (1,b_0)$, there are parameters $\gamma:=(b-1)^2$,  $a_0>1$ and $n_0,h_0\in \N$ sufficiently large with the following property.
   For every $a>a_0$, $t\le \mathfrak{t}$ and $N\le 2$ it holds 
    \begin{align} 
        \|\phi\qq(t)\|_N&\label{est:phi}\le \lambda\qq ^{N-3\gamma/2}\delta_{q+2}^{3/2},
        \\
        \|D_{t,q+1}\phi\qq(t)\|_{N-1}&\label{est:Dphi}\le \lambda\qq ^{N-3\gamma/2}\delta_{q+2}^{3/2}\delta\qq^{1/2},
        \\
        |f(t)|&\le\frac{1}{100} \lambda\qq^{-\gamma}\delta\qqq,
        \\
        |\partial_t f(t)|&\le \frac{1}{100}\lambda\qq^{-\gamma}\delta\qqq\delta\qq^{1/2}.
    \end{align}
\end{prop}

\subsubsection{Proof of \autoref{prop:phi}: bounds on $\phi_A$}

\begin{lemma}\label{lemma:phiA}
Let $\phi_A$, $f_A$ be defined by \eqref{e:phiA} and \eqref{e:f}.
    For $N\le 2$ and  $t\le \mathfrak{t}$ we have
    \begin{align}
    \|\phi_A(t)\|_N 
    &\lesssim \lambda\qq^N\left[\lambda\qq^{-1}+(\lambda\qq\mu)^{-n_0-1}\right]\frac{\delta\qq}{\epsilon},
    \\
    \|D_{t,q+1}\phi_A(t)\|_{N-1}&\lesssim \lambda\qq^{N-1}\left[(\epsilon\lambda\qq)^{-1}+\lambda\qq^{-n_0}\mu^{-n_0-1}+(l\lambda\qq)^{-h_0}\lambda\qq^2\lambda_q\delta_q^{1/2}\right]
    \frac{\delta\qq}{\epsilon},
    \\
    \label{f_A}|f_A(t)|+|f'_A(t)|&\lesssim \frac{\delta\qq}{\epsilon}(\lambda\qq \mu)^{-n_0-1}.
    \end{align}

\end{lemma}
\begin{proof}
The result is due to \autoref{prop:antidiv1}, \autoref{prop:antidiv2} applied to $F=\varphi_A$. Indeed, recall the decomposition of $\varphi_A$ given by \eqref{e:varphiA}. Following the notation of \autoref{prop:antidiv1}, let us define $a_{m,k} := a^1_{m,k}+a^2_{m,k}+a^3_{m,k}$ where
\begin{align}
    a_{m,k}^1&:=D_{t,l}\left(\frac{\Tr(c_{m,k})}2\right)+\dvg (d_{m,k})+c_{m,k}:\nabla(v_l+z_q)^T+\left(   b_{m,k}+\frac{e_{m,k}}{\lambda\qq}\right)\SymOtNoTr (z\qq-z_l):\nabla z_q^T\\
    &\quad+
    \left(   b_{m,k}+\frac{e_{m,k}}{\lambda\qq}\right)\cdot \dvg\left((z\qq-z_l)\otimes z_q\right)-(z\qq-z_l)\cdot\left[\left(b_{m,k}+\frac{e_{m,k}}{\lambda\qq}\right)\cdot \nabla v_l\right]\\
    &\quad+
    \left(b_{m,k}+\frac{e_{m,k}}{\lambda\qq}\right)\cdot\left[-(\overline{z}\cdot \nabla)v_l+u_l\cdot \nabla (z_q-z_l)+u\qq\cdot \nabla \varz
    \right]\\
    &\quad+
    (v_q-v_l)\cdot\left[\left(b_{m,k}+\frac{e_{m,k}}{\lambda\qq}\right)\cdot \nabla \varz \right],
  \\
    a_{m,k}^2&:=\left(b_{m,k}+\frac{e_{m,k}}{\lambda\qq}\right)\dvg(\proj_{\le l^{-1}}R_q+R_{comm}),\\
    a_{m,k}^3&:=\left(b_{m,k}+\frac{e_{m,k}}{\lambda\qq}\right)\cdot \nabla (p_q-p_l).
\end{align}
We intend to apply \autoref{prop:antidiv1} with $\lambda=\lambda\qq$, a sequence 
\begin{align}
    s_j = \frac{\delta\qq}{\epsilon} \mu^{-j},
\end{align}
and $\mathring{a}_k, a_F$ defined accordingly.
As for the terms of the form $a_{m,k}^1$, the assumptions of the proposition are readily checked by \autoref{lemma:coeff}, \autoref{lemma:stocEstimates}, and the inductive estimates on $v\qq$.
The assumptions of the proposition for terms of the form $a_{m,k}^2$ can be checked by \autoref{lemma:coeff}, inductive estimate \eqref{s:R}, and \cite[Lemma A.3]{DLK23}. 
Finally, for $a_{m,k}^3$ we invoke \autoref{lemma:coeff} and \autoref{lemma:pressure}. 
Observe that the constant $a_F$ from \autoref{prop:antidiv1} depends only on finitely many derivatives of the Mikado flows $\psi_I$, in particular we absorb it into the $\lesssim$ in the bound on $\|\phi_A\|_N $.

Let us move to the estimate on $\|D_{t,q+1}\phi_A\|_{N-1}$. We apply \autoref{prop:antidiv2} with $\lambda,s_j$ as above and $c=\epsilon$.
In this case, we check the assumptions on $D_{t,l}a_{m,k}^1$ by the inductive estimates, \autoref{lemma:coeff}, the bound \eqref{est:DtlNablaU}, \autoref{lemma:stocEstimates}, and \autoref{cor:Dtq1Z}.
In a similar fashion, in order to control the assumptions on $D_{t,l} a_{m,k}^2$ we use Bernstein inequality, \eqref{est: DtlDiffv} and \cite[Lemma A.4]{DLK23} to control, on one hand:
\begin{align}
    \|D_{t,l}(\dvg R_{comm})\|_N &\lesssim l^{-N}\|D_{t,l}(\dvg R_{comm})\|_0 
    \\
    &\lesssim
    \|D_{t,l}(u_q-u_l)\|_0\|u_q\|_1+l\|u_q\|_1\left[\|D_{t,l}u_l\|_1+\|u_q\|_1^2\right]+l^2\|u_q\|_1^2\|u_q\|_2
    \\
    &\lesssim l^{1-N}\left[\lambda_q^3\delta_q^{3/2}+\lambda_q\delta_q^{1/2}i_q^{-1/2-\delta}\right],
\end{align}
where the second inequality follows by the same computations shown in \cite[Lemma 5.2]{DLK23}, and on the other
\begin{align}  
        \|D_{t,l}(\dvg \proj_{\le l^{-1}}R_q)\|_N&\le\| \proj_{\le l^{-1}}D_{t,l}R_q\|_{N+1}+
        \| [u_l\cdot \nabla, \proj_{\le l^{-1}}]R_q\|_{N+1}+\|(\nabla u_l)\nabla( \proj_{\le l^{-1}}R_q)\|_N\\ &\lesssim l^{-N}\|D_{t,q}R_q\|_1+\|(u_q-u_l)\cdot \nabla R_q\|_1+l^{-N}\|u_q\|_1\|R_q\|_1\\
        &\lesssim l^{-N}\lambda^2_q\delta\qq\delta_q^{1/2}.
\end{align}

Finally, as for the terms $D_{t,l}a_{m,k}^3$, we have by \autoref{lemma:coeff}
\begin{align}
    \|D_{t,l}a_{m,k}^3\|_j
    &\lesssim \sum_{h=0}^j \left\|D_{t,l}\left(b_{m,k}+\frac{e_{m,k}}{\lambda\qq}\right)\right\|_{j-h}\|p_q-p_l\|_{h+1}\\
    &+ \sum_{h=0}^j
    \left\|\left(b_{m,k}+\frac{e_{m,k}}{\lambda\qq}\right)\right\|_{j-h}
    \left(\|D_{t,l}(p_q-p_l)\|_{h+1}+\|(\nabla u_l)\nabla(p_q-p_l)\|_h\right)
    \\
    &\lesssim \lambda_q\delta_q\rho^{1/2}\epsilon^{-1}\mu^{-j},
\end{align}
where in the last line we have used the identity
\begin{align}
    D_{t,l}(p_q-p_l)&=D_{t,q}p_q-(D_{t,q}p_q)_l+(u_l-u_q)\cdot \nabla p_q+(u_q\cdot \nabla p_q)_l-u_l\cdot \nabla u_l,
\end{align}
and the following bound, obtained by \cite[Lemma A.3]{DLK23} and Bernstein inequality:
\begin{equation}\label{est:a3first}
    \|D_{t,l}(p_q-p_l)\|_M\lesssim l^{1-M}\lambda_q^2\delta_q^{3/2}.
\end{equation}
At last, let us consider $|f_A|+|f'_A|$. Due to the spatial and temporal partitions of unity, arguing as in \autoref{prop:estw} it is sufficient to estimate $f_{m,k}(t):=\int_0^t\int_{\T^3}a_{m,k}(s,x)e^{i\lambda\qq\xi_m(s,x)\cdot k}\,dx\,ds$. Then, \cite[Lemma A.2]{DLK23} implies that 
\begin{equation}
   |f_A|+|f'_A|\lesssim \lambda\qq^{-n_0+1}(s_{n_0+1}+s_0\,l^{-n_0-1})\lesssim\frac{\delta\qq}{\epsilon}(\lambda\qq \mu)^{-n_0-1}. 
\end{equation}
\end{proof}

\subsubsection{Proof of \autoref{prop:phi}: bounds on $\phi_D$}

\begin{lemma}\label{lemma:phiD1}
Let $\phi_{D_1}$ be defined by \eqref{e:phiD1}.
For $N\le 2$ and $t\le \mathfrak{t}$ we have
    \begin{align}
    \|\phi_{D_1}(t)\|_N
    &\lesssim 
    \lambda\qq^N\left[l\,\lambda_q\delta_q^{1/2}\rho+(\lambda\qq\mu)^{-1}\rho^{3/2}\right]\\
    &\quad
    +l^{1-N}\lambda_q^{1-\gamma}\delta\qq^{3/2}(l^{-1}\delta_q^{1/2}l_{\rm temp}+1)+\lambda\qq ^{N-2\gamma}\delta_{q+2}^{3/2},
    \\
    \|D_{t,q+1}\phi_{D_1}(t)\|_{N-1}
    &\lesssim 
    \lambda\qq ^{N-2\gamma}\delta_{q+2}^{3/2}\delta\qq^{1/2}+l^{1-N}\lambda_q^{-\gamma}\delta\qq^{3/2}l_{\rm temp}^{-1}\\
   &\quad+\lambda\qq^N[l\delta\qq^{3/2}\lambda_q\delta_q^{1/2}+\lambda\qq^{-1}\lambda_q^{1-\gamma}\delta\qq^{3/2}\rho^{1/2}+(\lambda\qq\mu)^{-1}\rho^{2}]
    \end{align}
\end{lemma}
\begin{proof}
We bound each term in the decomposition $\eqref{e:phiD1}$ separately, starting from
\begin{align}
\bar{\phi}_{D_1}:=w\frac{\Tr(R\qq-\frac{3}{2}f\,\Id)}{2}-\left(R\qq-\frac{3}{2}f\,\Id\right)w.
\end{align}
In this case, the bounds \eqref{est:Rw1} and \eqref{est:Rw2} imply that
    \begin{align}\label{est:phiD1first}
        &\|\bar{\phi}_{D_1}\|_N
        \lesssim 
        \lambda\qq ^{N-3\gamma/2}\delta_{q+2}^{3/2},
        \qquad 
        \|D_{t,q+1}\bar{\phi}_{D_1}\|_{N-1}
        \lesssim 
        \lambda\qq ^{N-3\gamma/2}\delta_{q+2}^{3/2}\delta\qq^{1/2}.
    \end{align}
Moving to the other terms, we have by \eqref{s:V}, \autoref{lemma:coeff}, \eqref{i:FourierBasisFlux} and the bound $\mu^{-1} < \lambda\qq$ that 
\begin{equation}\label{est:phiD1second1}
    \left\|(v_q-v_l)\frac{\Tr(w_o \otimes w_o-c_0)}{2}\right\|_N\lesssim   l\lambda_q\delta_q^{1/2}\rho \lambda\qq^{N},
\end{equation}
while the advective derivative is controlled by 
    \begin{align}\label{est:phiD1second2} 
        \left\|D_{t,q+1}\left( (v_q-v_l)\frac{\Tr(w_o \otimes w_o-c_0)}{2}\right) \right\|_{N-1}
        &\lesssim 
        \|D_{t,l}(v_q-v_l) \Tr(w_o \otimes w_o-c_0) \|_{N-1}\\
        &\quad+\|(v_q-v_l) D_{t,l} \Tr(w_o \otimes w_o-c_0) \|_{N-1}\\
        &\quad+
        \left\|((u\qq-u_l)\cdot \nabla) \left( (v_q-v_l) \Tr(w_o \otimes w_o-c_0) \right) \right\|_{N-1}\\
        &\lesssim \lambda\qq^N l\, \rho^{3/2}\lambda_q\delta_q^{1/2},
    \end{align}
where in the last line we have used \eqref{est:phiD1second1} above, \autoref{lemma:coeff}, the bounds \eqref{est: DtlDiffv}, \eqref{e:DiffUl}, and \autoref{prop:estw}. 
By similar computations, we can control also
\begin{align}
    \left\|\frac{|v_q-v_l|^2}{2}w\right\|_N
    &\lesssim 
    \lambda\qq ^Nl^2\lambda_q^2\delta_q\,\rho^{1/2},\\
    \left\|D_{t,q+1}\left(\frac{|v_q-v_l|^2}{2}w\right)\right\|_{N-1}
    &\lesssim 
    \lambda\qq ^Nl^2\lambda_q^2\delta_q\,\rho.
\end{align}
The next term in the decomposition \eqref{e:phiD1} we deal with is $\phi_q-\phi_l$. In this case, we can repeat the computations reported in \autoref{lemma:RT}, obtaining
\begin{align}\label{est:phiD1third}
    \|\phi_q-\phi_l\|_N
    &\lesssim 
    l^{1-N}\lambda_q^{1-\gamma}\delta\qq^{3/2}(l^{-1}\delta_q^{1/2}l_{\rm temp}+1),\\
    \|D_{t,q+1}(\phi_q-\phi_l)\|_{N-1}
    &\lesssim 
    \lambda_q^{N-\gamma}\delta\qq^{3/2}\delta_q^{1/2}+ \lambda\qq^{N-1}\lambda_q^{1-\gamma}\delta\qq^{3/2} \rho^{1/2}+l^{1-N}\lambda_q^{-\gamma}\delta\qq^{3/2} l_{\rm temp}^{-1}.
\end{align}
The cubic term in the velocity perturbation is controlled by means of \autoref{prop:estw} as follows:
\begin{align}\label{est:phiD1fourth}
    \left\|\frac{|w|^2w-|w_o|^2w_o}{2}\right\|_N
    &\lesssim 
    \lambda\qq^N(\lambda\qq\mu)^{-1}\rho^{3/2},\\
    \left\|D_{t,q+1}\left(\frac{|w|^2w-|w_o|^2w_o}{2}\right)\right\|_{N-1}
    &\lesssim 
    \lambda\qq^N(\lambda\qq\mu)^{-1}\rho^2.
\end{align}
Finally, the remaining term $\tilde{\phi}_{D_1}$ can be rewritten as $\tilde{\phi}_{D_1} = \bar{R}(v_q-v_l)$, where 
\begin{align}
    \bar{R}
    &:=
     R_q-w\otimes w-R\qq-\rho\Id+\frac32f\,\Id 
    \\
    &=
    -\sum_{m\in\Z}\sum_{k \in \Z^3 \setminus \{0\}}c_{m,k}e^{i\lambda\qq\xi_m \cdot k}-R_A-R_S-w\SymOt(v_q-v_l),
\end{align}
and then we can control it by \autoref{lemma:coeff}, \autoref{lemma:RA}, \autoref{lemma:RS} as follows
\begin{align}
    \|\tilde{\phi}_{D_1}\|_N
    &\lesssim \lambda\qq^Nl\,\lambda_q\delta_q^{1/2}\rho,
\end{align}
whereas the advective derivative can be controlled by using the aforementioned lemmas, \autoref{prop:estw}, and the estimates \eqref{e:DiffUl}, \eqref{est: DtlDiffv}
\begin{align}
    \|D_{t,q+1}\tilde{\phi}_{D_1}\|_{N-1}
    &\le
    \left\|(v_q-v_l)D_{t,q+1}(R_A+R_S+w\SymOt(v_q-v_l))\right\|_{N-1}
    \\
    &\quad+
    \left\|(v_q-v_l)\sum_{m\in\Z}\sum_{k \in \Z^3 \setminus \{0\}} \left( D_{t,l}(c_{m,k})e^{i\lambda\qq\xi_m \cdot k}+(u\qq-u_l)\cdot \nabla (c_{m,k}e^{i\lambda\qq\xi_m \cdot k}) \right) \right\|_{N-1}
    \\
        &\quad+
        \left\|\left( D_{t,l}(v_q-v_l)+(u\qq-u_l)\cdot \nabla(v_q-v_l) \right) \bar{R}\right\|_{N-1}
        \\&\lesssim\lambda\qq^{N-1}\lambda_q\delta_q^{1/2}\left(\lambda\qq l\,\delta\qq^{3/2}+\rho^{3/2}\right).
\end{align}
\end{proof}

\begin{lemma}\label{lemma:phiD2}
Let $\phi_{D_2}$, $f_{D_2}$ be defined by \eqref{e:phiD2} and \eqref{e:f}. For $N\le 2$, it holds for any $t\le \mathfrak{t}$ that
    \begin{align}
    \label{est:phiD2}
    \|\phi_{D_2}(t)\|_N
    &\lesssim 
    \frac{\rho}{\mu}\lambda\qq^N\lambda_q\delta_q^{1/2}(\lambda\qq^{-2}+\mu^{-n_0-1}\lambda\qq^{-n_0-1}),
    \\
    \label{e:DtphiD2}
    \|D_{t,q+1}\phi_{D_2}(t)\|_{N-1}
    &\lesssim 
    \frac{\rho}{\mu} \lambda\qq^{N-1}  \Bigg[\rho^{1/2}\lambda_q\delta_q^{1/2}(\lambda\qq^{-1}+\lambda\qq(\lambda\qq\mu)^{-n_0-1})\\
    &\quad+
    \left(\frac{\lambda_q\delta_q^{1/2}\epsilon^{-1}}{\lambda\qq^2}+\frac{\lambda_q^2\delta_q\,l}{\lambda\qq}\right)
    + 
    (\lambda\qq\mu)^{-n_0-1} \left(\,\lambda\qq\lambda_q\delta_q^{1/2}\right)\Bigg],
    \\
    \label{f_D2}|f_{D_2}(t)|+|f'_{D_2}(t)|
    &\lesssim 
    \frac{\rho}{\mu}\lambda_q\delta_q^{1/2}(\lambda\qq\mu)^{-n_0-1}.
    \end{align}   
\end{lemma}
\begin{proof}
    Based on the decompositions \eqref{e:phiD2} and \eqref{e:RdecA}, we have that
    \begin{equation}
        \phi_{D_2}=\overline{\mathcal{R}}\left[\mathcal{R}\left(G_1+G_2+G_3\right):\nabla u_l^T\right],
    \end{equation}
where we have denoted $G_1:=\dvg(w_o \otimes w_o-c_0)$, $G_2 := D_{t,l}w$, and $G_3:= w\cdot \nabla u_l$. Notice that, for every $i=1,2,3$, $\mathcal{R}\proj_{\geq 3\lambda\qq/8}G_i:\nabla u_l^T$ has frequencies of size at least $3\lambda\qq/8-l^{-1} \gtrsim \lambda\qq$ and $\mathcal{R}\proj_{< 3\lambda\qq/8}G_i:\nabla u_l^T$ has frequencies of size at most $3\lambda\qq/8+l^{-1} \lesssim \lambda\qq$. Therefore, by the bound on the antidivergence $\|\overline{\mathcal{R}}\proj_{\gtrsim \lambda\qq} f\|_N\lesssim \lambda\qq^{-1}\|f\|_N$ we have
\begin{align} \label{est:phiD2first}
   \left\|\overline{\mathcal{R}}\left[\mathcal{R}\left(G_i\right):\nabla u_l^T\right]\right\|_N
   &\le  
   \left\|\overline{\mathcal{R}}\left[\mathcal{R}\proj_{\geq 3\lambda\qq/8}G_i:\nabla u_l^T\right]\right\|_N
   +
   \left\|\overline{\mathcal{R}}\left[\mathcal{R}\proj_{<3 \lambda\qq/8}G_i:\nabla u_l^T\right]\right\|_N
   \\
&\lesssim
\frac1{ \lambda\qq}\left\|\mathcal{R}\proj_{\geq 3\lambda\qq/8}G_i:\nabla u_l^T\right\|_N
+
{}\lambda\qq^N \left\|\mathcal{R}\proj_{< 3\lambda\qq/8}G_i:\nabla u_l^T \right\|_0
\\
&\lesssim 
\frac1{ \lambda\qq^2}\sum_{j=0}^N \|G_i\|_j\,l^{N-j}\|u_q\|_1
+
{\lambda\qq^N} \|\proj_{<3\lambda\qq/8}G_i\|_0\|u_q\|_1.
\end{align}
Now, recalling the definition on the antidivergence stress $R_A$ given by \eqref{e:RdecA}, we observe that each $G_i$ has already been treated in the proof of \autoref{lemma:RA}. Therefore, the estimates therein and \eqref{i:FourierBasisFlux} imply that
\begin{equation}\label{est:Gi}
    \|G_i\|_j\lesssim \frac{\rho}{\mu}\lambda\qq^j.
\end{equation}
Instead, \eqref{e:decF} implies that $\proj_{<3\lambda\qq/8}G_i=-\sum_{m,k}\epsilon_{n_0}^{\lambda}(k\cdot \xi_m,a_{m,k}^i)e^{i\lambda\qq k\cdot \xi_m}$, so that by \eqref{e:epsN0LSum}
\begin{equation}\label{est:projGi}
    \|\proj_{<3\lambda\qq/8}G_i\|_0\lesssim \frac{\rho}{\mu}\mu^{-n_0-1}\lambda\qq^{-n_0-1}.
\end{equation}
Therefore, \eqref{est:phiD2first},  \eqref{est:Gi}, and \eqref{est:projGi} produce the following bound
\begin{equation}
    \|\phi_{D_2}\|_N\lesssim \frac{\rho}{\mu}\lambda\qq^N\lambda_q\delta_q^{1/2}(\lambda\qq^{-2}+\mu^{-n_0-1}\lambda\qq^{-n_0-1}).
\end{equation}

Now, let us turn to \eqref{e:DtphiD2}.
Denote $H_i:= \mathcal{R}\proj_{\geq 3\lambda/8}G_i:\nabla u_l^T,$ then
    \begin{align}\label{est:DtphiD2first}
        \|D_{t,l}\overline{\mathcal{R}}\proj_{\geq \lambda\qq}H_i\|_{N-1}
        &\le 
        \|\overline{\mathcal{R}} \left(\proj_{\geq \lambda\qq}D_{t,l}H_i
        +[u_l\cdot \nabla,\proj_{\geq \lambda\qq}]H_i\right)\|_{N-1}
        \\
        &\quad+\|[u_l\cdot \nabla,\overline{\mathcal{R}}]\proj_{\geq\lambda\qq}H_i\|_{N-1}
        \\
        &\lesssim \frac{1}{\lambda\qq}\|D_{t,l}H_i+[u_l\cdot \nabla,\proj_{\geq \lambda\qq}]H_i\|_{N-1}+\|[u_l\cdot \nabla,\overline{\mathcal{R}}]\proj_{\geq\lambda\qq}H_i\|_{N-1}\\
        &\lesssim \frac{1}{\lambda\qq}\|D_{t,l}H_i\|_{N-1}+\lambda\qq^{N-3}\|u_q\|_1\|H_i\|_1+\sum_{j_1+j_2\le N-1}l\,\|\nabla u_l\|_{j_1}\|H_i\|_{j_2},
    \end{align}
where the first inequality is due to $\mathcal{R}$, the second to a trivial adaptation of \cite[Lemma A.4]{DLK23}, \autoref{oss:postAntidiv} and $(l\lambda\qq)^{-h_0}\lambda^2\qq\le l$.  
By a similar reasoning as above, we can estimate $\|D_{t,l}(\mathcal{R}\proj_{\geq \lambda\qq}G_i)\|_M$. Indeed, it holds:
\[\|D_{t,l}(\mathcal{R}\proj_{\geq \lambda\qq}G_i)\|_M\lesssim\frac1{\lambda\qq}\|D_{t,l}G_i\|_M+\lambda\qq^{M-2}\|u_l\|_1\|G_i\|_1+\sum_{M_1+M_2\le M}l\|\nabla u_l\|_{M_1}\|G_i\|_{M_2}.\]
Thanks to the previous inequality, \eqref{est:DtlNablaU}, \eqref{est:DtlRA}, we have that 
    \begin{align} \label{est:DtphiD2second}
        \|D_{t,l}H_i\|_{N-1}&\le \|D_{t,l}(\mathcal{R}\proj_{\geq \lambda\qq}G_i):\nabla u_l^T\|_{N-1}+\|(\mathcal{R}\proj_{\geq \lambda\qq}G_i):D_{t,l}\nabla u_l^T\|_{N-1}\\
        &\lesssim \lambda^{N-1}\qq\frac{\rho}{\mu}\left(\frac{\lambda_q\delta_q^{1/2}\epsilon^{-1}}{\lambda\qq^2}+\frac{\lambda_q^2\delta_q\,l}{\lambda\qq}\right)
    \end{align}
Instead by standard computations, it holds true
\begin{equation}\label{est:DtphiD2third}
    \|H_i\|_M\lesssim \lambda\qq^{M-1}\frac{\rho}{\mu}\lambda_q\delta_q^{1/2}.
\end{equation}
Therefore, \eqref{est:DtphiD2first}, \eqref{est:DtphiD2second}, \eqref{est:DtphiD2third} imply that
\begin{equation}\label{est:DtphiD2fourth}
    \|D_{t,l}\overline{\mathcal{R}}\proj_{\geq \lambda\qq}H_i\|_{N-1}\lesssim \lambda^{N-1}\qq\frac{\rho}{\mu}\left(\frac{\lambda_q\delta_q^{1/2}\epsilon^{-1}}{\lambda\qq^2}+\frac{\lambda_q^2\delta_q\,l}{\lambda\qq}\right).
\end{equation}
Instead, Bernstein inequality, \eqref{est:projGi}, \eqref{e:Antidiv3} imply that
    \begin{align} \label{est:DtphiD2fifth}
        \|D_{t,l}\overline{\mathcal{R}}(\mathcal{R}\proj_{\lesssim\lambda\qq}G_i:\nabla u_l^T)\|_{N-1}
        &\lesssim 
        \lambda\qq^{N-1}\big[\|D_{t,l}(\mathcal{R}\proj_{\lesssim\lambda\qq}G_i:\nabla u_l^T)\|_0
        \\
        &+
        \|[u_l\cdot \nabla, \overline{\mathcal{R}}](\mathcal{R}\proj_{\lesssim\lambda\qq}G_i:\nabla u_l^T)\|_0\big]\\
        &\lesssim\lambda\qq^{N-1}\Big[\big(\|D_{t,l}\proj_{\lesssim\lambda\qq}G_i\|_0+\|[u_l\cdot \nabla,\overline{\mathcal{R}}](\mathcal{R}\proj_{\lesssim\lambda\qq}G_i)\|_0\big)\|u_l\|_1\\
        &+\|\proj_{\lesssim\lambda\qq}G_i\|_0\|D_{t,l}\nabla u_l\|_0+\lambda\qq \,\|u_l\|_0\|u_l\|_1\|\proj_{\lesssim\lambda\qq}G_i\|_0\Big]\\
        &\lesssim\lambda\qq^{N-1}\frac{\rho}{\mu}(\lambda\qq\mu)^{-n_0-1}\left[\,\lambda\qq\lambda_q\delta_q^{1/2}\right]
    \end{align}

For later use, we additionally notice that the above implies in particular
\begin{align} \label{eq:auxiliary.E''}
    \|D_{t,l}(\mathcal{R}\proj_{\lesssim\lambda\qq}G_i:\nabla u_l^T)\|_0
    \lesssim
    \frac{\rho}{\mu}(\lambda\qq\mu)^{-n_0-1}\left[\,\lambda\qq\lambda_q\delta_q^{1/2}\right],
    \quad
    \forall i \in \{1,2,3\}.
\end{align}
    
At last, \eqref{e:DiffUl}, \eqref{est:phiD2}
\begin{equation}\label{est:DtphiD2sixth}
    \|(u\qq-u_l)\cdot\nabla \phi_{D_2}\|_{N-1}\lesssim\lambda\qq^N\frac{\rho^{3/2}}{\mu}\lambda_q\delta_q^{1/2}(\lambda\qq^{-2}+(\lambda\qq\mu)^{-n_0-1}).
\end{equation}
This concludes the estimates for $\phi_{D_2}$. Regarding the estimates on $f_{D_2}$, notice that for every $s \leq \mathfrak{t}$
\begin{align}
    \int_{\T^3}\left(\mathcal{R}(G_i):\nabla u_l^T\right)(s,x)\,dx=\int_{\T^3}\left(\mathcal{R}(\proj_{<3\lambda\qq/8} G_i):\nabla u_l^T\right)(s,x)\,dx.
\end{align}
Therefore, \eqref{est:projGi} implies that
\begin{equation}
   |f_{D_2}|+|f'_{D_2}|\lesssim \frac{\rho}{\mu}\lambda_q\delta_q^{1/2}(\lambda\qq\mu)^{-n_0-1}. 
\end{equation}

\end{proof}

\subsubsection{Proof of \autoref{prop:phi}: bounds on $\phi_S$}
\begin{lemma}\label{lemma:phiS1}
Let $\phi_{S_1}$ be defined by \eqref{e:phiS1}. For $N\le 2$, it holds for any $t\le \mathfrak{t}$ that
    \begin{align}
    \label{est:phiS1}
    \|\phi_{S_1}(t)\|_N
    &\lesssim 
    l^{-N}i_q^{1/2-\delta}+\lambda\qq^{N-2\gamma}\delta\qqq^{3/2}+i_q^{1/2-\delta}\begin{cases}
        1,\quad &N=0\\
        \lambda\qq^N\delta\qq^{1/2},\quad &N>0
    \end{cases},\\
    \label{e:DtphiS1}
    \|D_{t,q+1}\phi_{S_1}(t)\|_{N-1}
    &\lesssim i_q^{1/2-\delta}\lambda\qq^N\delta\qq+\lambda\qq^{N-2\gamma}\delta\qqq^{3/2}\delta\qq^{1/2}+\lambda\qq^{N-1}\delta\qqq^{3/2}\rho^{1/2}\epsilon^{-1}\\
     &+i\qq^{-1/2-\delta}\begin{cases}
        1,\quad &N=1\\
        \lambda\qq^{N-1}\delta\qq^{1/2}\vee l^{1-N},\quad &N>1
    \end{cases}
    \end{align}   
\end{lemma}
\begin{proof}
We divide the right-hand side of \eqref{e:phiS1} into several terms, each to be controlled separately.
First, the trace term can be controlled using the bounds \eqref{est:R1}, \eqref{est:R2}, and \eqref{est:Dtq1ZVar}, more specifically we have
\begin{align}\label{est:phiS1second}
\left\|\Tr\left(R\qq-\frac32f\,\Id\right)\varz \right\|_N
&\lesssim 
\lambda\qq^N \delta_{q+2}\,i_q^{1/2-\delta},
\\
\left\|D_{t,q+1}\left[\Tr\left(R\qq-\frac32f\,\Id\right)\varz \right]\right\|_{N-1}
&\lesssim \label{est:DphiS1second}
\lambda\qq^{N-1}\delta_{q+2}(i\qq^{-1/2-\delta}+i_q^{1/2-\delta}\lambda\qq).
\end{align}
Next, notice that 
    \begin{equation}
    \frac{|v_q|^2+|w|^2}{2}\varz\,+\, (\varz\otimes v_q)w=\frac{|v\qq|^2}{2}\varz,
    \end{equation}
for which we have the estimates 
    \begin{align}\label{est:phiS1first}
    \left\|\frac{|v\qq|^2}{2}\varz\right\|_N
    &\lesssim 
    i_q^{1/2-\delta}\begin{cases}
        1,\quad &N=0\\
        \lambda\qq^N\delta\qq^{1/2},\quad &N>0
    \end{cases},\\
    \left\|D_{t,q+1}\left(\frac{|v\qq|^2}{2}\varz\right)\right\|_{N-1}
    &\lesssim 
    \lambda\qq^{N}\delta\qq i_q^{1/2-\delta}+ i\qq^{-1/2-\delta}\begin{cases}
        1,\quad &N=0\\
        \lambda\qq^{N-1}\delta\qq^{1/2},\quad &N>0
    \end{cases},
    \end{align}
where the second inequality we have used \autoref{lemma:stocEstimates}, \eqref{est:Dtq1ZVar}, \eqref{est:R1}, \eqref{est:pq1} and 
\[D_{t,q+1}v\qq=\dvg\left(R\qq-\frac{3}{2}f\,\Id\right)-\nabla p\qq+u\qq\cdot \nabla z\qq
.\]
Finally, upon recalling the definition \eqref{e:varp} of the pressure increment $\varp$, the remaining term $\tilde{\phi}_{S_1}:=(\varz\otimes \varz +u_q\SymOt \varz)v_l + \varp w$ is controlled by means of \autoref{prop:estw}, \autoref{lemma:stocEstimates}, the bounds \eqref{s:V}, \eqref{est:Dtq1ZVar}, \eqref{e:DiffUl}, \eqref{est:Dtq1uq}, \eqref{est: DtlDiffv}, and the following inequality:
\begin{align} \label{eq:D_q+1_vl}
    &\|D_{t,q+1}v_l\|_{N-1}\le \|D_{t,l}v_l\|_{N-1}+\|(u\qq-u_l)\cdot\nabla v_l\|_{N-1}\lesssim \lambda\qq^{N-1}\lambda_q\,\delta_q.
\end{align}
Namely, we have 
\begin{align}\label{est:phiS1third}
    \|\tilde{\phi}_{S_1}\|_N
    &\lesssim l^{-N}i_q^{1/2-\delta}+\lambda\qq^{N-2\gamma}\delta\qqq^{3/2},\\
    \label{est:DphiS1third}
    \|D_{t,q+1} \tilde{\phi}_{S_1}\|_{N-1}
    &\lesssim i\qq^{-1/2-\delta}l^{1-N}+i_q^{1/2-\delta}\lambda\qq^{N-1}\lambda_q\delta_q\\
    &\quad+\lambda\qq^{N-2\gamma}\delta\qqq^{3/2}\delta\qq^{1/2}+\lambda\qq^{N-1}\delta\qqq^{3/2}\rho^{1/2}\epsilon^{-1}.
\end{align}

\end{proof}
\begin{lemma}\label{lemma:phiS2}
Let $\phi_{S_2}$, $f_{S_2}$ be defined by \eqref{e:phiS2} and \eqref{e:f}. For $N\le 2$, it holds for any $t\le \mathfrak{t}$ that
    \begin{align}
    \label{est:phiS2}\|\phi_{S_2}(t)\|_N
    &\lesssim \lambda\qq^N\delta\qqq i_q^{1/2-\delta}+i_q^{1/2-\delta}\begin{cases}
        1,\quad &N=0\\
        \lambda_q^N\delta_q^{1/2},\quad &N>0
    \end{cases},
    \\
    \label{e:DtphiS2}\|D_{t,q+1}\phi_{S_2}(t)\|_{N-1}
    &\lesssim 
    \lambda\qq^{N-1}\lambda_q\delta_q^{1/2}i_q^{1/2-\delta}+\lambda\qq^{N-1}\delta\qqq i\qq^{-1/2-\delta}+l\lambda\qq^{N}\delta\qqq+i\qq^{-1/2-\delta}\begin{cases}
        1,\quad &N=1\\
        \lambda_q^{N-1}\delta_q^{1/2},\quad &N>1
    \end{cases},\\
    |f_{S_2}(t)|+|f'_{S_2}(t)|
    &\lesssim i_q^{1/2-\delta}.
    \end{align}   
\end{lemma}

\begin{proof}
Since we will only use that $\|\overline{\mathcal{R}}\|_{L^{\infty}\to L^{\infty}}\lesssim 1$, the estimates on $f_{S_2}$
follow from the estimate on $\|\phi_{S_2}\|_0$.

Let us start with an estimate on $\bar{\phi}_{S_2}
:= \overline{\mathcal{R}}[R_A:\nabla(z_q-z_l)^T]$.
In this case, \eqref{est:R1} implies that
\begin{align}\label{est:phiS2first}
    &\|\bar{\phi}_{S_2}\|_N  \lesssim l\lambda\qq^{N-\gamma}\delta\qqq.
\end{align}
As for the advective derivative, we have that
\begin{equation}
\begin{aligned}
    \|D_{t,q+1}\bar{\phi}_{S_2}\|_{N-1}
    &\le 
    \left\|\overline{\mathcal{R}}[D_{t,q+1}\left(R_A:\nabla(z_q-z_l)^T\right)]\right\|_{N-1}
    +
    \left\|[u\qq\cdot\nabla,\overline{\mathcal{R}}](R_A:\nabla(z_q-z_l)^T)\right\|_{N-1}\\
    &\le \left\|D_{t,q+1}\left(R_A:\nabla(z_q-z_l)^T\right)\right\|_{N-1}+\sum_{j=0}^{N-1}\|u\qq\|_j\|R_A:\nabla(z_q-z_l)^T\|_{N-j}\\
    &\le l\lambda\qq^{N-1}\delta\qqq \left(i_q^{-1/2-\delta}+\lambda\qq\right),
\end{aligned}
\end{equation}
where the last inequality is due to \autoref{lemma:stocEstimates} and the estimates \eqref{est:R1} and \eqref{est:R2}.

The next term we study is $\tilde{\phi}_{S_2}:= \overline{\mathcal{R}}[(u_q\SymOtNoTr\varz+\varz \otimes \varz ):\nabla z_q^T ]$, which is controlled similarly to bounds \eqref{est:phiS1third} and \eqref{est:DphiS1third} previously obtained for $\tilde{\phi}_{S_1}$, with the difference that \eqref{eq:D_q+1_vl} is now replaced by the bound $\|D_{t,q+1}\nabla z_q\|_N\lesssim i_q^{-1/2-\delta}+\begin{cases}
        1,\quad &N=0\\
        \lambda\qq^N\delta\qq^{1/2},\quad &N>0
    \end{cases}$, consequence of \autoref{lemma:stocEstimates}, 
giving
\begin{align}\label{est:phiS2third}
    &\|\tilde{\phi}_{S_2}\|_N\lesssim i_q^{1/2-\delta}\begin{cases}
        1,\quad &N=0\\
        \lambda_q^N\delta_q^{1/2},\quad &N>0
    \end{cases},\\
    &\|D_{t,q+1}\tilde{\phi}_{S_2}\|_{N-1}\lesssim \lambda\qq^{N-1}\lambda_q\delta_q^{1/2}\,i_q^{1/2-\delta}+i\qq^{-1/2-\delta}\begin{cases}
        1,\quad &N=1\\
        \lambda_q^{N-1}\delta_q^{1/2},\quad &N>1
    \end{cases}.
\end{align}
Furthermore, similar arguments leading to \eqref{est:phiS1second} and \eqref{est:DphiS1second} also yield
\begin{align}\label{est:phiS2second}
    \left\|\overline{\mathcal{R}}\left[\left(R\qq-\frac32f\,\Id\right):\nabla \varz^T\right]\right\|_N
    &\lesssim 
    \lambda\qq^N \delta_{q+2}\,i_q^{1/2-\delta},\\
        \left\|D_{t,q+1}\overline{\mathcal{R}}\left[\left(R\qq-\frac32f\,\Id\right):\nabla \varz^T\right]\right\|_{N-1}
        &\lesssim 
        \lambda\qq^{N-1}\delta_{q+2}(i\qq^{-1/2-\delta}+i_q^{1/2-\delta}\lambda\qq).
\end{align}
At last, \autoref{lemma:stocEstimates}, \eqref{est:Dtq1uq}, and the following inequalities
\begin{align}
    \|D_{t,q+1}v_q\|_N
    &\lesssim \lambda\qq^N\lambda_q\delta_q,
    \\
    \|D_{t,q+1}\nabla \varz\|_N
    &\lesssim i\qq^{-1/2-\delta}+i_q^{1/2-\delta}\begin{cases}
        1,\quad &N=0\\
        \lambda\qq^N\delta\qq^{1/2},\quad &N>0
    \end{cases},
    \\
    \|D_{t,q+1}\nabla z\qq\|_N
    &\lesssim 
    i\qq^{-1/2-\delta}+\begin{cases}
        1,\quad &N=0\\
        \lambda\qq^N\delta\qq^{1/2},\quad &N>0
    \end{cases}, 
\end{align}
imply that
\begin{align}\label{est:phiS2fourth}
    \left\|\overline{\mathcal{R}}\Big[v_q\cdot \dvg(u_q\otimes \varz+\varz\otimes z\qq)\Big]\right\|_N
    &\lesssim 
    i_q^{1/2-\delta}\begin{cases}
        1,\quad &N=0\\
        \lambda_q^N\delta_q^{1/2},\quad &N>0
    \end{cases},
    \\
    \left\|D_{t,q+1}\overline{\mathcal{R}}\Big[v_q\cdot \dvg(u_q\otimes \varz+\varz\otimes z\qq)\Big]\right\|_{N-1}
    &\lesssim 
    i_q^{1/2-\delta}\lambda\qq^{N-1}\lambda_q\,\delta_q^{1/2}
    +
    i\qq^{-1/2-\delta}\begin{cases}
        1,\quad &N=1\\
        \lambda_q^{N-1}\delta_q^{1/2},\quad &N>1
    \end{cases}.
\end{align}
\end{proof}

\subsubsection{Proof of \autoref{prop:phi}: Conclusion}
 Bounds \eqref{est:phi} and \eqref{est:Dphi} follow from \autoref{lemma:phiA}, \autoref{lemma:phiD1},  \autoref{lemma:phiD2}, \autoref{lemma:phiS1}, \autoref{lemma:phiS2} above and our choice of parameters, up to taking the constant $a_0>1$ large enough so to absorb all the implicit constants coming from the aforementioned lemmas.
As for the estimates on $f$, we have $f := f_A+ f_{D_2}+f_{S_2}$ and thus by \autoref{lemma:phiA},  \autoref{lemma:phiD2}, and \autoref{lemma:phiS2}:
    \begin{align}
        &|\partial_t f|\le |\partial_t f_A+\partial_t f_{D_2}+\partial_t f_{S_2}|\lesssim\lambda\qq^{-3\gamma/2}\delta\qqq \delta\qq^{1/2}.
    \end{align}
    Finally, the claim on $f$ follows from $|f|\le T\sup_{t \in [0,\mathfrak{t}]}|\partial_t f(t)|$ and from picking $a_0>1$ large enough.

\section{Proof of \autoref{lemma:iter3}} \label{sec:forced}

In this section, we focus on proving \autoref{lemma:iter3}. Its proof closely follows that of \autoref{IterLemma}, so we will primarily emphasize the technical steps that differ.

As mentioned, \autoref{lemma:iter3} covers two separate scenarios: one in which the forcing term is updated together with the iteration, which is needed for the proof of \autoref{thm:main3}, and another in which the forcing is fixed in advance, which is used in the proof of \autoref{thm:main2}.

It is somewhat easier to start from the second scenario. In this case, recalling that $z_q \equiv0$ since we did not decompose the velocity using the Da Prato-Debussche trick, the equivalent of \eqref{eq:decomposition_Rqq} for the new pressure and Reynolds stress is given by 
    \begin{align}
        -\nabla p\qq+\dvg(R\qq)+g\qq
        &=
        \partial_t w+v_l\cdot \nabla w+\dvg\Big(w_o\otimes w_o-c_0\Big)+w\cdot \nabla v_l
        \\
        &\quad+
        \dvg\Big(w_c\otimes w+w_o\otimes w_c+w\SymOt (v_q-v_l)+R_q-R_l\Big)
        \\
        &\quad+ g_q-\nabla p_q.
    \end{align}
If we impose the constraints $p_q = p\qq$ and $g_q = g\qq$, the equation above gives the candidate new Reynolds stress    
\begin{align}
    R\qq := R_A + R_T,
\end{align}
where $R_A$, $R_T$ are defined in \eqref{e:RdecA} and \eqref{e:RdecT}, and $R_S = 0$. Notice that, differently from what we did in the proof of \autoref{IterLemma}, here we do not incorporate the spatial average corrector term $\frac{2}{3}f(t)\Id$.
This, indeed, would prevent us from obtaining estimates that are uniform with respect to $t \in \R$, and we could not apply the Krylov-Bogoliubov argument to construct statistically stationary solutions of \eqref{eq:FE}. 
However, since $f$ was chosen so that $f'(t)$ cancels out the space average in the equation \eqref{e:LEIq1} for the new current $\phi\qq$ to apply antidivergence estimates, we must now introduce an artificial average corrector $e\qq-e_q$ during the iteration. More specifically, we have the analog of \eqref{e:NewCurrent1}:
    \begin{align}\label{eq:LEI.forced.iteration}
        \partial_t&\frac{|v\qq|^2}{2}
        +
        v\qq\cdot \nabla p\qq
        +
        v\qq \cdot\dvg\big(v\qq\otimes v\qq\big) 
        =
        v_{q+1} \cdot g\qq  -E' +e\qq 
        \\
        &+\dvg(R\qq v\qq)
        +D_{t,q}\left(\frac{\Tr(R_q)+|w|^2+2(v_q-v_l)\cdot w}{2}\right)+\dvg\left(\phi_q+\frac{|w|^2}{2}w-R\qq w\right)
        \\
        &+\dvg\left((R_q+w \otimes w -R\qq)(v_q-v_l)  
    \right)
    +
    \dvg\left(\frac{|v_q-v_l|^2}{2}w\right)
        \\
        &+\left((w_o \otimes w_o-c_0)-R_A\right):\nabla v_l^T
        \\
        &+
        w\cdot\dvg\Big\{(p\qq-p_l)\Id+\P_{\leq l^{-1}}R_q
        +
        R_{comm} \Big\}
        -w \cdot g_q +e_q -e\qq,
    \end{align} 
    where $D_{t,l}=\partial_t+v_l\cdot \nabla$.
From this equation, and arguing as in \autoref{sec:current} above, we let
\begin{align}
    \phi\qq &:= \phi_A + \phi_{D_1}+ \phi_{D_2},
    \quad
    e\qq - e_q := \fint_{\T^3} \left( \varphi_A -w \cdot g_q +\varphi_{D_2}  \right),
\end{align}
with $\varphi_A$, $\varphi_{D_2}$ defined as in \eqref{e:varphiA} and \eqref{e:phiD2}, $\phi_A:= \overline{\mathcal{R}}(\varphi_A-w\cdot g_q)$, $\phi_{D_2} :=  \overline{\mathcal{R}}\varphi_{D_2}$, and $\phi_{D_1}$ defined as in \eqref{e:phiD1}.
Notice that $\phi_{D_1}$ in \autoref{sec:current} depended on $f$ only through the term $R\qq - \frac32 f \Id$, which we can simply replace by $R_A+R_T$ here.

Let us now move to describe how to set up the convex integration scheme that proves the first part of \autoref{lemma:iter3}. 
The condition $p\qq = p_q$ will still be true; however, the external forcing will not remain constant during the iteration. 
In order to absorb the increment $g_q-g\qq$ in the new Reynolds stress, here we define $R\qq := R_A + R_T$, where
\begin{align}\label{eq:RAnew}
    R_A := \mathcal{R} \left( 
    \partial_t w+u_l\cdot \nabla w+w\cdot \nabla u_l +\dvg\Big(w_o\otimes w_o-c_0 \Big)
    +g_q-g\qq
    \right) ,
\end{align}
and $R_T$ is defined as in \eqref{e:RdecT}.
We will use the increment of the forcing $\tilde{g}\qq := g\qq-g_q$ to reduce the space averages in \eqref{eq:LEI.forced.iteration} through the ``work'' $w \cdot \tilde{g}\qq$. 
In fact, in this case the analog of \eqref{eq:LEI.forced.iteration} reads as
   \begin{align}\label{eq:LEI.forced.iteration2}
        \partial_t&\frac{|v\qq|^2}{2}
        +
        v\qq\cdot \nabla p\qq
        +
        v\qq \cdot\dvg\big(v\qq\otimes v\qq\big) 
        =
        v\qq \cdot g\qq -E' +e\qq 
        \\
        &+\dvg(R\qq v\qq)
        +D_{t,q}\left(\frac{\Tr(R_q)+|w|^2+2(v_q-v_l)\cdot w}{2}\right)+\dvg\left(\phi_q+\frac{|w|^2}{2}w-R\qq w\right)
        \\
        &+\dvg\left((R_q+w \otimes w -R\qq)(v_q-v_l)  
    \right)
    +
    \dvg\left(\frac{|v_q-v_l|^2}{2}w\right)
        \\
        &+\left[(w_o \otimes w_o-c_0)-R_A\right]:\nabla v_l^T
        \\
        &+
        w\cdot\dvg\Big\{(p\qq-p_l)\Id+\P_{\leq l^{-1}}R_q
        +
        R_{comm} \Big\}
       +v_q \cdot g_q -v\qq \cdot g\qq 
       +e_q - e\qq,
    \end{align} 
and using $\tilde{g}_{q+1} = g\qq-g_q$ we can rewrite
\begin{align} 
       v_q \cdot g_q -v\qq \cdot g\qq 
       +e_q - e\qq
       &=
       -w \cdot \tilde{g}_{q+1} - v_q \cdot \tilde{g}_{q+1}- w \cdot {g}_{q}
      +e_q-e\qq.
   \end{align} 
The increment $\tilde{g}\qq$ will be decomposed into a principal part $\tilde{g}_{q+1,o}$ so that 
\begin{align}
  \varphi_{\neq 0} := -w_o \cdot \tilde{g}_{q+1,o}+e_q
  \quad
  \mbox{ is high-frequency},
\end{align}
and a compressibility corrector $\tilde{g}_{q+1,c}$;
the other terms $-w_c \cdot \tilde{g}_{q+1,o}-w_o \cdot \tilde{g}_{q+1,c}- v_q \cdot \tilde{g}_{q+1}- w \cdot {g}_{q}$ will be either small or high-frequency, and absorbed into a new current $\phi\qq := \phi_A +  \phi_{D_1}+ \phi_{D_2}+ \phi_F$, where 
\begin{align}
    \phi_F &:= \overline{\mathcal{R}} \varphi_F,
    \\
    \varphi_F &:= \varphi_{\neq 0} -w_c \cdot \tilde{g}_{q+1,o}-w_o \cdot \tilde{g}_{q+1,c}- v_q \cdot \tilde{g}_{q+1}- w \cdot {g}_{q}.
\end{align}
Then the new energy loss tensor $e\qq$ will serve the purpose to cancel out space averages:
\begin{align}
   e\qq := \fint_{\T^3} \left( \varphi_A +\varphi_{D_2} +\varphi_F  \right). 
\end{align}

\subsection{Adaptive random forcing case}
In this subsection we discuss the iterative estimates needed to prove the first part of \autoref{lemma:iter3}, where the external random forcing is modified at each step.

Let us start by defining the velocity and the random forcing increments.
Firstly, following \autoref{sec:velocity}, we can define $v\qq=v_q+w_o+w_c$.
On the other hand, we let $g\qq:=g_q+\tilde{g}\qq$ with increment defined as
 \begin{align}
     \label{eq:gtilde}
     \tilde{g}\qq(t,x)&:=\sum_{I\in \I_R}\theta_I(t)
\chi_I(\xi_I(t,x))h_I(t,x)\nabla\xi_I(t,x)^{-1}f_I
  \psi_{I}(\lambda\qq \xi_I(t,x))
  \\
  &\quad+
  \frac{1}{\lambda\qq}\sum_{I\in \I_R}\sum_{k \in \Z^3 \setminus\{0\}}
  \theta_I(t)
  \nabla(\chi_I(\xi_I(t,x))h_I(t,x))
  \times \left( \nabla\xi_I^{T}(t,x)\frac{ik \times f_I}{|k|^2}\right)
   \mathring{b}_{I,k}e^{i\lambda\qq \xi_I(t,x) \cdot k}
  \\
  &=:\tilde{g}_{q+1,o}(t,x) + \tilde{g}_{q+1,c}(t,x),
 \end{align}
 where the amplitude function is defined as
\begin{align}\label{eq:h_I}
    h_I(t,x) := - \frac{e_q(t)}{6\,a_I(t,x) |\overline{f_I}|^2},
    \quad
    I \in \mathcal{I}_R.
\end{align}
In particular, we can rewrite $\tilde{g}\qq$ as follows: 
\begin{align} \label{eq:gtildeFour}
\tilde{g}\qq (t,x)
&=
\sum_{m \in \Z}
\sum_{k \in \Z^3 \setminus\{0\}} \left(\tilde{b}_{m,k}(t,x)+ \frac{1}{\lambda\qq} \tilde{e}_{m,k}(t,x) \right) e^{i \lambda\qq \xi_m(t,x) \cdot k},
\end{align} 
where
\begin{align}
 \tilde{b}_{m,k}(t,x) &:=
\sum_{I \in \mathcal{I}_m \cap \mathcal{I}_R}
\theta_I 
\chi_I(\xi_I(t,x))h_I(t,x)\nabla\xi_I(t,x)^{-1}f_I
\mathring{b}_{I,k},
  \\
 \\
\tilde{e}_{m,k}(t,x) &:=
\sum_{I \in \mathcal{I}_m \cap \mathcal{I}_R}
\theta_I 
  \nabla(\chi_I(\xi_I(t,x))h_I(t,x))
  \times \left( \nabla\xi_I^{T}(t,x)\frac{ik \times f_I}{|k|^2}\right) \mathring{b}_{I,k}.
\end{align}
Notice that this choice of $\tilde{g}\qq$ implies that
\begin{equation}
\begin{aligned}
    \varphi_{\neq 0}
    &= 
    \frac{e_q}{6}\sum_{I\in \I_R}\theta_I^2
  \chi_I^2(\xi_I)
  \sum_{k \in \Z^3 \setminus \{0\}}
  \mathring{c}_{I,{k}} e^{i \lambda\qq \xi_I \cdot k}
  \\
  &=
  -\frac{e_q}{6} \sum_{m \in \Z}\sum_{k\in \Z^3\setminus \{0\}}\sum_{I \in \mathcal{I}_m \cap \mathcal{I}_R} \theta_I^2 \chi_I^2 (\xi_I) \mathring{c}_{I,k}\,e^{i \lambda\qq \xi_m \cdot k}.
\end{aligned}
\end{equation}
Since $v\qq$ has been built as in \autoref{sec:velocity}, \eqref{s:V} holds. Let us focus on \eqref{s:g3} and \eqref{s:distance.g}.

\begin{lemma} \label{lem:g}
    For $s \in \{0,1\}$, $N \leq n_0+3$ and $t \in \R$ we have
    \begin{align} \label{est:gtilde}
        \epsilon^s \| D^s_{t,l} \tilde{g}\qq(t) \|_N
        \lesssim
        \lambda\qq^N \rho.
    \end{align}
Moreover, 
\begin{equation}\label{est:DtqGq}
    \|D_{t,q+1}g\qq(t)\|_{N-2}\le \lambda\qq^{N-1}\delta\qq^{3/2}.
\end{equation}
\end{lemma}
\begin{proof}
Recall $    h_I(t,x) := - \frac{e_q(t)}{a_I(t,x) |\overline{f_I}|^2}$ and the bound $\epsilon^s\|D^s_{t,l} a_I\|_N \lesssim \mu^{-N} \rho^{1/2}$ from the proof of \autoref{lemma:coeff}.
Thus, by \eqref{s:e3} on the dissipation $e_q$ and arguing similarly to the proof of \autoref{lemma:coeff} we have the estimates
\begin{align} \label{est:tilde.b}
    \epsilon^s \| D^s_{t,l} \tilde{b}_{m,k} \|_N
        &\lesssim
        \sup_{I \in \mathcal{I}_R} |\mathring{b}_{I,k}| \rho \mu^{-N},
        \\ \label{est:tilde.e}
        \epsilon^s \| D^s_{t,l} \tilde{e}_{m,k} \|_N
        &\lesssim
        \sup_{I \in \mathcal{I}_R} |\mathring{b}_{I,k}| \rho \mu^{-N-1},
\end{align}
and therefore \eqref{est:gtilde} follows by similar arguments as those used in the proof of \autoref{prop:estw}. Finally, \eqref{est:DtqGq} follows from
\begin{equation}
    \begin{aligned}
        \|D_{t,q+1}g\qq(t)\|_{N-2}&= \|D_{t,q}g_q(t)+D_{t,l}\tilde{g}\qq(t)+w\cdot\nabla g\qq(t)+(v_q-v_l)\cdot\nabla \tilde{g}\qq(t)\|_{N-2}\\
        &\lesssim\lambda_q^{N-1}\delta_q^{3/2}+\epsilon^{-1}\lambda\qq^{N-2}\rho+\lambda\qq^{N-1}\rho^{3/2}.
    \end{aligned}
\end{equation}
\end{proof}

Now, we shift our focus on the Reynold stress $R\qq$, the current $\phi\qq$ and the energy loss $e\qq$.
To show \eqref{s:R}, \eqref{s:Phi} and \eqref{s:e3} we only need to take into account the new terms due to the random forcing.

\begin{lemma}\label{lemma:gtilde}
   For every $t \in \R$ and $N \leq n_0+3$ we have
    \begin{align}
        \| \mathcal{R}\tilde{g}\qq(t,x) \|_{N}
        &\lesssim 
        \rho \lambda\qq^{N} [\lambda\qq^{-1} + (\lambda\qq \mu)^{-n_0-1}],
        \\
        \| D_{t,q+1} \mathcal{R}\tilde{g}\qq(t,x) \|_{N-2}
        &\lesssim
        \rho \lambda\qq^{N-2}\left[(\epsilon\lambda\qq)^{-1}+\lambda\qq^{-n_0}\mu^{-n_0-1}+(l\lambda\qq)^{-h_0}\lambda\qq^2\lambda_q\delta_q^{1/2}\right].
    \end{align}
    As a consequence,
    \begin{align}
    \|R_A(t)\|_N
    &\lesssim \lambda\qq^N\left[\lambda\qq^{-1}+(\lambda\qq\mu)^{-n_0-1}\right]\frac{\rho}{\mu},\\
    \|D_{t,q+1}R_A(t)\|_{N-2}
    &\lesssim \lambda\qq^{N-2}\left[(\epsilon\lambda\qq)^{-1}+\lambda\qq^{-n_0}\mu^{-n_0-1}+(l\lambda\qq)^{-h_0}\lambda\qq^2\lambda_q\delta_q^{1/2}\right]\frac{\rho}{\mu}.
    \end{align}
\end{lemma}
\begin{proof}
We intend to apply \autoref{prop:antidiv1} and \autoref{prop:antidiv2} to $F=\tilde{g}\qq$, with $a_{m,k}= \tilde{b}_{m+1} + \lambda\qq^{-1} \tilde{e}_{m,k}$. By previous \autoref{lem:g}, here the assumptions of \autoref{prop:antidiv1} are satisfied with $\lambda=\lambda\qq$ and $s_j = \rho \mu^{-j}$, whereas the assumptions of \autoref{prop:antidiv2} are satisfied with $D_{t,l}a_{m,k}= D_{t,l}\tilde{b}_{m+1} + \lambda\qq^{-1}D_{t,l}\tilde{e}_{m,k}$, $c=\epsilon$ and $\lambda,s_j$ as above.
At last, the estimate on $R_A$ follows from these bounds in combination with \autoref{lemma:RA}.
\end{proof}

\begin{lemma}\label{lemma:phiF}
For every $t \in \R$ and $N\le 2$ it holds 
    \begin{align} 
        \|\phi_F(t)\|_N&\le \lambda\qq ^{N-2\gamma}\delta_{q+2}^{3/2},
        \\
        \|D_{t,q+1}\phi_F(t)\|_{N-1}&\le \lambda\qq ^{N-2\gamma}\delta_{q+2}^{2}\delta\qq^{1/2} .
    \end{align}  
\end{lemma}
\begin{proof}
To estimate $\overline{\mathcal{R}}(\varphi_{\neq 0}+v_q\cdot \tilde{g}\qq+w\cdot g_q)$, we apply \autoref{prop:antidiv1} and \autoref{prop:antidiv2} with 
\begin{equation} \label{eq:overlineAmk}
\begin{aligned}
    a_{m,k}&=-\frac{e'_q}{6}\left(\sum_{I \in \mathcal{I}_m \cap \mathcal{I}_R} \theta_I^2 \chi_I^2 (\xi_I) \mathring{c}_{I,k}\right)+v_q \cdot\left(\tilde{b}_{m,k} + \frac{1}{\lambda\qq} \tilde{e}_{m,k}  \right)
    +\left({b}_{m,k} + \frac{1}{\lambda\qq}  {e}_{m,k}  \right) \cdot g_q.
\end{aligned}
\end{equation}
Indeed, \eqref{s:e3} and \eqref{s:g3} imply that the assumptions of \autoref{prop:antidiv1} and \autoref{prop:antidiv2} hold with $\lambda=\lambda\qq$, $s_j=\rho^{1/2}\mu^{-j}$ and $c=\epsilon.$
 
For the term $\varphi_c:=-w_c \cdot \tilde{g}_{q+1,o}-w_o \cdot \tilde{g}_{q+1,c}$ it is not necessary to apply estimates for the antidivergence operator, since we already have
\begin{align} \label{e:varphiG}
  \| \varphi_c \|_N 
    &\lesssim
    (\lambda\qq \mu)^{-1} \rho^{3/2}  \lambda\qq^N,
    \\
    \|D_{t,l} \varphi_c \|_{N-1} \label{e:DvarphiG}
    &\lesssim
    \epsilon^{-1}(\lambda\qq \mu)^{-1} \rho^{3/2} \lambda\qq^{N-1}.
\end{align}
Indeed, under the condition $\nu<\frac{1-7\alpha}{4\alpha}$, it holds
\begin{align}
    \| \overline{\mathcal{R}}\varphi_c \|_N
    \lesssim
    \| \varphi_c \|_N
    \lesssim
    (\lambda\qq \mu)^{-1} \rho^{3/2}  \lambda\qq^N
    \ll
    \lambda\qq ^{N-2\gamma}\delta_{q+2}^{3/2},
\end{align}
and by Schauder estimates we have for some $0 < \delta \ll 1$
\begin{align}
    \| D_{t,q+1} \overline{\mathcal{R}} \varphi_c\|_{N-1} 
    &\lesssim
    \| \overline{\mathcal{R}} D_{t,l} \varphi_c\|_{N-1}
    +
    \| (v\qq-v_l) \cdot \nabla \overline{\mathcal{R}} \varphi_c\|_{N-1}
    +
    \| \overline{\mathcal{R}} \dvg(v_l \otimes \varphi_c)\|_{N-1}
    +
    \|v_l \cdot \nabla \overline{\mathcal{R}}\varphi_c\|_{N-1}
    \\
    &\lesssim 
    \| D_{t,l} \varphi_c\|_{N-1}
    +
    \sum_{j=0}^{N-1}
    (\| (v\qq-v_l)\|_{N-1-j} + \|v_l \|_{N-1-j} ) \|\nabla \overline{\mathcal{R}} \varphi_c\|_{j}
    +
    \| v_l \otimes \varphi_c\|_{N-1+\delta}
    \\
    &\lesssim 
    \| D_{t,l} \varphi_c\|_{N-1}
    +
  \| \varphi_c\|_{N-1+\delta}
  \\
  &\lesssim
  (\epsilon^{-1}+\lambda\qq^{\delta} )(\lambda\qq \mu)^{-1} \rho^{3/2} \lambda\qq^{N-1} 
  \ll
  \lambda_{q+1}^{N-3\gamma/2} \delta\qqq^{2}.
\end{align}
\end{proof}
\begin{lemma} \label{lem:E'}
For every $t \in \R$ we have
  \begin{align}
   |e\qq (t)| 
  &\leq \lambda_{q+1}^{-2\gamma} \delta\qqq^{3/2},
   \\
   |e'\qq (t)| 
  &\leq \lambda_{q+1}^{1-3\gamma/2} \delta\qqq^{2}.  
  \end{align}  
\end{lemma}

\begin{proof}
The control over $\fint_{\T^3} \left( \varphi_A(t,x)+\varphi_{D_2}(t,x) \right)\,dx$
follows by replicating the estimates given in \autoref{lemma:phiA}, \autoref{lemma:phiD2}, up to taking $n_0,h_0$ sufficiently large. \\
To control the average of $\varphi_{F} = \varphi_{\neq 0}
  + \varphi_c-  v_q \cdot \tilde{g}\qq - w\cdot g_q $ we argue as follows. For $|\fint_{\T^3} \varphi_c| \lesssim \| \varphi_c\|_0$ we invoke \eqref{e:varphiG}; on the other hand, for $\fint_{\T^3} (\varphi_{\neq 0}-  v_q \cdot \tilde{g}\qq - w\cdot g_q )$ we invoke the stationary phase lemma (see \cite[Lemma A.2]{DLK23}) to obtain
  \begin{align} \label{eq:auxiliary.E''bis}
      \left|\fint_{\T^3} {a}_{m,k} 
  e^{i \lambda\qq \xi_m  \cdot k} \right|
  \lesssim
  \rho^{1/2} (\lambda\qq\mu)^{-n_0-1}
 \le
 \lambda\qq^{-3\gamma} \delta\qqq^{3/2},
  \end{align}
up to choosing $n_0$ large enough.
Notice that the bound on $  |e\qq (t)| $ is slightly better than the one on $|\partial_t f|$ given by \autoref{prop:phi} since here $z=0$, hence $\varphi_{S_2} = 0$.

Furthermore, since $e\qq$ is space independent we also have
 \begin{align}
    e'\qq(t)&:= \fint_{\T^3} \left( \varphi'_A(t,x)+\varphi'_{D_2}(t,x)+\varphi_F'(t,x)\right)\,dx
    =
    \fint_{\T^3} \left( D_{t,l}\varphi_A(t,x)+ D_{t,l}\varphi_{D_2}(t,x)+D_{t,l}\varphi_F(t,x)\right)\,dx.
\end{align}

Notice that $D_{t,l}\varphi_A(t,x) =\sum D_{t,l}(a_{m,k}) e^{i\lambda\qq\xi_m \cdot k} $, where $a_{m,k}$ are defined in the proof of \autoref{lemma:phiA} and satisfy $\|D_{t,l}a_{m,k}\|_{N} \lesssim \delta\qq \epsilon^{-2} \mu^{-N}$ for every $N \leq n_0+1$, uniformly in $m \in \Z$ and with enough summability with respect to $k \in \Z^3\setminus \{0\}$. Then, the stationary phase lemma implies that
\begin{align}
 \left|  \fint_{\T^3}  D_{t,l}\varphi_A   \right|
 =
 \left|  \sum_{m,k} \fint_{\T^3}  D_{t,l} a_{m,k} e^{i\lambda\qq \xi_m \cdot k} \right|
\lesssim
\frac{\delta\qq \epsilon^{-2} \mu^{-N} + \delta\qq \epsilon^{-2} l^{1-N}}{\lambda\qq^N}
\le
\lambda_{q+1}^{1-2\gamma} \delta\qqq^{2},
\end{align}
up to taking $n_0,h_0$ sufficiently large.

Next, by \eqref{e:phiD2} we have $\varphi_{D_2} = R_A : \nabla v_l^T$, where we recall that 
\begin{gather} 
  R_A = \mathcal{R}(G_1+G_2+G_3+G_4) , 
\end{gather}
with tensors $\{G_i\}_{i \in \{1,2,3\}}$  defined in the proof of \autoref{lemma:phiD2} and $G_4 := -\tilde{g}\qq$.
Moreover, notice that since $v_l$ has Fourier frequencies of size $\lesssim l^{-1}$, we have
\begin{align}
  \fint_{\T^3} 
  D_{t,l} \varphi_{D_2}
  =
  \sum_{i=1}^4
  \fint_{\T^3} D_{t,l}(\mathcal{R} G_i:\nabla v_l^T)
  =
  \sum_{i=1}^4
  \fint_{\T^3} D_{t,l}(\mathcal{R}\proj_{\lesssim\lambda\qq}G_i:\nabla v_l^T).
\end{align}

By \eqref{eq:auxiliary.E''} we have for every $i \in \{1,2,3\}$
\begin{align}\label{est4}
    \|D_{t,l}(\mathcal{R}\proj_{\lesssim\lambda\qq}G_i:\nabla v_l^T)\|_0
    \lesssim
    \frac{\rho}{\mu}(\lambda\qq\mu)^{-n_0-1}\left[\,\lambda\qq\lambda_q\delta_q^{1/2}\right]
    \le
    \lambda_{q+1}^{1-2\gamma} \delta\qqq^{2},
\end{align}
up to taking $n_0$ large enough. 
A similar argument and \autoref{lemma:gtilde} imply that \eqref{est4} holds for $i=4$, too.

Let us move to the term with $\varphi_{F} = \varphi_{\neq 0}
  + \varphi_c-  v_q \cdot \tilde{g}\qq - w\cdot g_q $. For $\varphi_c$ we have by \eqref{e:DvarphiG}
\begin{align}
    \left|\fint_{\T^3}  D_{t,l}\varphi_c  \right| \leq \| D_{t,l}\varphi_c\|_0
    \lesssim
    \epsilon^{-1}(\lambda\qq \mu)^{-1} \rho^{3/2} 
    \le
    \lambda_{q+1}^{1-2\gamma} \delta\qqq^{2}.
\end{align}
The other terms are treated by applying again the stationary phase lemma. Indeed, $\varphi_{\neq 0}
  -  v_q \cdot \tilde{g}\qq - w\cdot g_q=\sum {a}_{m,k} e^{i\lambda\qq\xi_m \cdot k}$, where ${a}_{m,k}$ are given by \eqref{eq:overlineAmk}.
  Since $\|D_{t,l}{a}_{m,k}\|_N\lesssim \rho^{1/2}\epsilon^{-1}\mu^{-N}$ for all $N\le n_0+1$, then the stationary phase lemma implies that 
  \begin{equation}
      \left|\fint_{\T^3}D_{t,l}(\varphi_{\neq 0}
  -  v_q \cdot \tilde{g}\qq - w\cdot g_q)\right|\lesssim\frac{\rho^{1/2}\epsilon^{-1}\mu^{-N}+\rho^{1/2}\epsilon^{-1}l^{1-N}}{\lambda\qq^N}.
  \end{equation}
  Therefore, the claim follows after picking $n_0$ sufficiently large.
\end{proof}
Putting all these lemmas together, plus the previously obtained estimates of sections \autoref{ssec:estimates_v}, \autoref{ssec:estimatesR}, and \autoref{ssec:estimatesPhi}, we obtain the first part of \autoref{lemma:iter3}.

\subsection{Fixed random forcing case}
As mentioned at the beginning of this section, the fixed random forcing case (i.e. $g_q=g\qq$) requires control over
\begin{equation}
    \phi_A=\overline{\mathcal{R}}(\varphi_A-w\cdot g_q),\quad e\qq-e_q=\fint_{\T^3}\varphi_A+\varphi_{D_2}-w\cdot g_q,
\end{equation}
where $\varphi_A$ and $\varphi_{D_2}$ are given by \eqref{e:varphiA} and \eqref{e:phiD2}.
To conclude the proof of \autoref{lemma:iter3} we only need the following:
\begin{lemma}
For every $t \in \R$ and $N\le 2$ it holds 
    \begin{align}\label{est5}
        \|\phi_A(t)\|_N&\le \lambda\qq ^{N-2\gamma}\delta_{q+2}^{3/2},
        \\
        \|D_{t,q+1}\phi_A(t)\|_{N-1}&\le \lambda\qq ^{N-2\gamma}\delta_{q+2}^{2}\delta\qq^{1/2} .
    \end{align}  
Moreover,  
    \begin{equation}\label{est6}
        |e\qq(t)-e_q(t)|\le \lambda\qq^{-2\gamma}\delta_{q+2}^{3/2}.
    \end{equation}
\end{lemma}
\begin{proof}
    The claim follows from previous iterative estimates. In particular, \eqref{est5} follows directly from \autoref{lemma:phiA} and \autoref{lemma:phiF}. Instead, \eqref{est6} is due to \autoref{lem:E'}.
\end{proof}

\appendix
\section{}

\subsection{Estimates for transport equations}
We report some well-known estimates regarding smooth solutions $f:\R \times \R^3 \to \R$ of the transport equation
\begin{equation}\label{e:transport}
    \begin{cases}
        &\partial_t f+v\cdot \nabla f=g,\\
        &f(t_0,\cdot)=f_0,
    \end{cases}
\end{equation}
where $v:\R \times\R^3 \to \R^3$ is a smooth vector field and $g :\R \times\R^3 \to \R$ is a smooth scalar field. 
We adopt the notation $f_t := f(t,\cdot)$, and similarly for $v,g$, et cetera.
The following proposition is a slight modification of \cite[Proposition B.1]{BuDLSzVi19}.
\begin{prop}\label{prop:transport}
In the same setting as above, for any $\alpha \in[0,1]$ it holds
    \begin{equation}\label{e:1stTransport}
        \|f_t\|_\alpha\lesssim e^{\alpha(t-t_0)\|v_{[t_0,t]}\|_{1}}\left(\|f_0\|_\alpha+\int_{t_0}^t\|g_s\|_\alpha ds\right),
        \qquad
        \|v_{[t_0,t]}\|_1:=\sup_{s\in [t_0,t]}\|v_s\|_{C^1_x}.
    \end{equation}
   
In addition, if $|t-t_0|\|v_{[t_0,t]}\|_1\le 1$, then for any $N\geq 1$ and $\alpha\in[0,1)$ 
\begin{equation}\label{e:2ndTransport}
    [f_t]_{N+\alpha}\lesssim [f_0]_{N+\alpha}+|t-t_0|[v_{[t_0,t]}]_{N+\alpha}[f_0]_1+\int_{t_0}^t\left([g_s]_{N+\alpha}+(t-s)[v_{[t_0,t]}]_{N+\alpha}[g_s]_1\right)ds,
\end{equation}
where $[v_{[t_0,t]}]_{N+\alpha}:=\sup_{s\in[t_0,t]}[v_s]_{N+\alpha}$.

Moreover, the spatial inverse $\phi$ of the Lagrangian flux:
\begin{align}
    \phi_t(x) = (X_t)^{-1}(x),
    \quad
    \dot{X}_t=v(t,X_t),
    \quad
    X(x,t_0)=x,
    \quad
    \forall x \in \R^3,
\end{align}
satisfies the bounds
\begin{align} \label{e:ClosetoIdFlux}
\|\nabla \phi_t-\Id\|_0 &\lesssim |t-t_0|\|v_{[t_0,t]}\|_1,
    \\
    [\phi_t]_N &\lesssim|t-t_0|[v_{[t_0,t]}]_N, \quad \forall N\geq 2.
\end{align}
\end{prop}

\subsection{Antidivergence estimates}

In this subsection, we report a technical generalization of \cite[Corollary 8.2]{DLK23}. Let us start by recalling the following lemma from \cite{DLK23}.

\begin{lemma}(\cite[Lemma 8.1
]{DLK23}).  \label{lemma:antidiv}
    Let $\mathfrak{m} \in \mathcal{S}(\R^3)$ and define the Fourier multiplier operator $T_\mathfrak{m}$, acting an smooth functions $f \in C^\infty(\T^3,\mathbb{C})$ (smoothness meant in real sense) as follows:
    \begin{equation}
        \mathscr{F}
    [T_\mathfrak{m}(f)](k)=\mathscr{F}[f](k)\mathfrak{m}(k), \quad k \in \Z^3 .
    \end{equation}
Then, for any $n_0\in\N$, $\lambda>0$ and any scalar functions $a,\, \xi\in C^{\infty}(\T^3)$, the function $T_\mathfrak{m}(ae^{i\lambda \xi})$ can be decomposed as
\[T_\mathfrak{m}(ae^{i\lambda \xi})=\left[a\,\mathfrak{m}(\lambda\nabla\xi)+\sum_{k=1}^{2n_0}C_k^{\lambda}(\xi,a):(\nabla^k\mathfrak{m})(\lambda\nabla\xi)+\mathcal{\epsilon}_{n_0}(\xi,a)\right]e^{i\lambda\xi},\]
for some tensor-valued coefficents $C_k^{\lambda}(\xi,a)$ and smooth function $\mathcal{\epsilon}_{n_0}$ given by
    \begin{align}\label{e:epsn0}
        \mathcal{\epsilon}_{n_0}(\xi,a)(x)
        &:=
        \sum_{n_1+n_2=n_0+1}\frac{(-1)^{n_1}c_{n_1,n_2}}{n_0!}\\
        &\quad
        \cdot\int_0^1\int_{\R^3}\overset{\vee}{\mathfrak{m}}(y)e^{-i\lambda \nabla\xi(x)\cdot y}((y\cdot \nabla)^{n_1}a)(x-ry)e^{i\lambda Z[\xi]_{x,y}(r)}\beta_{n_2}[\xi](r)(1-r)^{n_0}dy\,dr,
    \end{align}
where $c_{n_1,n_2}$ are coefficients depending only on $n_1,n_2$, and $\beta_n[\xi]$ is  defined by the following expressions:
\begin{align}
    &\label{e:betan} \beta_n[\xi](r):=B_n(i\lambda Z'(r),\cdots,i\lambda Z^{(n)}(r)),\\
    &\label{e:Z}Z(r):=Z[\xi]_{x,y}(r):=r\int_0^1(1-s)(y\cdot\nabla)^2\xi(x-rsy)\,ds, \quad x ,y  \in \R^3 ,\\
    &\label{e:Bn}B_n(x_1,\dots,x_n):=\sum_{k=1}^nB_{n,k}(x_1,x_2,\dots,x_{n-k+1}),\\
    &\label{e:Bnk}B_{n,k}(x_1,x_2,\dots,x_{n-k+1}):=\sum \frac{n!}{j_1!j_2!\cdots j_{n-k+1}!}\left(\frac{x_1}{1!}\right)^{j_1}\cdots\left(\frac{x_{n-k+1}}{n-k+1!}\right)^{j_{n-k+1}},
\end{align}
and the last summation is over $\{j_k\}\subset \N$ such that
\[j_1+\cdots+j_{n-k+1}=k,\quad j_1+2j_2+\cdots+(n-k+1)j_{n-k+1}=n.\]
\end{lemma}

\begin{prop}\label{prop:antidiv1}
    Let $v\in C^{\infty}(\T^3,\R^3)$ and $l,\epsilon,\lambda$ be positive constants, and $\overline{n},n_0\in \N^*$ s.t. $\overline{n}\geq n_0+1$.
   Let $v_l:=\proj_{\le l^{-1}}v$ and consider, for any $m\in \Z$, the solution $\xi_m$ to the vector-valued transport equation 
    \begin{equation}
    \begin{cases}
        &\partial_t \xi_m+v_l\cdot \nabla \xi_m=0,\\
        &\xi_m(x,\epsilon(m-1/8))=x, \quad \forall x \in \R^3.
    \end{cases}
\end{equation}

Let $F:=\sum_{k\in\Z^3\setminus\{0\}}\sum_{m\in\Z}a_{m,k}e^{i\lambda\xi_m\cdot k}$, where the functions $a_{m,k}$ satisfy the following conditions:
\begin{enumerate}
    \item\label{ass:SupportTime} $\supp_t(a_{m,k})\subset (\epsilon(m-1/8),\epsilon(m+9/8))$, and $\epsilon>0$ is such that $2\epsilon\|v\|_{C_tC^1_x}\le1/4$;\\
    \item\label{ass:estimateAmk} there exist a positive real-valued sequence $(\mathring{a}_k)_{k \in \N}$, and constants $a_F>0$ with the following property: for every {$j \le  \overline{n}
    $} there exists a finite constant $s_j$ such that
    \begin{equation}
       \sum_k|k|^{{\overline{n}}}\mathring{a}_k\le a_F, \quad \|a_{m,k}(t)\|_j\leq  s_j\mathring{a}_k, \quad
       \forall (m,k) \in \Z \times (\Z^3 \setminus \{0\});
    \end{equation}

    \item\label{ass:s_n} the sequence $(s_n)_{n \in \N}$ is non-decreasing and the following inequalities hold for every $n \leq m$: 
    \[l^{-m}s_n\le s_{m+n},\quad s_n\lambda^m\le  s_0\lambda^{n+m}, \quad s_m\lambda^{-m}\le s_n\lambda^{-n}.\]
\end{enumerate}
Then, for every {$N \le \overline{n}$} and every $t \in \R$ we have
\begin{align}
    &\label{e:Antidiv1}\|\mathcal{R}F(t)\|_N\lesssim \left(s_0\lambda^{N-1}+\lambda^{-n_0-1}s_{n_0+1}\lambda^N\right)a_F.
\end{align}
\end{prop}

\begin{proof}
     First of all, let us quickly show (more details can be found at \cite{DLK23}) that
 \begin{equation}\label{e:decF}
     F=\proj_{\gtrsim3\lambda/8 }\left(\sum_{m,k}a_{m,k}e^{i\lambda k \cdot \xi_m}\right)-\sum_{m,k}\epsilon_{n_0}^{\lambda}(k\cdot \xi_m,a_{m,k})e^{i\lambda k\cdot \xi_m},
 \end{equation}
 where
 \begin{align}
     \proj_{\gtrsim3\lambda/8}&=\sum_{2^j\geq 3\lambda/8}\proj_{2^j},
     \\
     \epsilon_{n_0}^{\lambda}(k\cdot \xi_m,a_{m,k})&\label{e:epsN0L}=\sum_{2^j\geq 3\lambda/8}\epsilon_{n_0,j}(k\cdot \xi_m,a_{m,k}),
 \end{align}
with $\epsilon_{n_0,j}$ being obtained by \autoref{lemma:antidiv} applied to the Fourier multiplier operator $\proj_{2^j}$.
Notice that, by construction, $\xi_m$ is smooth and space-periodic, since $v_l$ is so.

In order to show \eqref{e:decF}, let us consider $(m,k)\in \Z\times(\Z^3 \setminus \{0\})$ and $j_k \in \N$ such that $2^{j_k}\geq \frac54\lambda|k|$. Then 
\begin{equation}
 \proj_{\gtrsim3\lambda/8}
 =
 \sum_{3\lambda/8\le2^s\le 2^{j_k}}\proj_{2^s}+\proj_{>2^{j_k}}=:T_{\mathfrak{m}_1}+\sum_{j>j_k}T_{\mathfrak{m}_j},   
\end{equation}
where the functions $\mathfrak{m}_1\in \mathcal{S}(\R^3)$ and $T_{\mathfrak{m}_j}=\proj_{2^j}$ satisfy
\begin{equation}
  \supp \,\mathfrak{m}_1\subset B(0,2^{j_k+1}),
  \, \quad 
  \mathfrak{m}_1\equiv 1 \,\text{ on  }\,\overline{B(0,2^{j_k})}\setminus B\left(0,\frac{3}{4}\lambda |k|\right),
  \quad 
  \mathfrak{m}_j\equiv 0 \text{ on }\overline{B}\left(0,\frac{5}{4}\lambda |k|\right).  
\end{equation}
Moreover for every $t\in \supp_t a_{m,k}$, Assumption \ref{ass:SupportTime} of this proposition and \eqref{e:ClosetoIdFlux} imply that
\begin{equation}
    \|| (\nabla \xi_m(t))^Tk|-|k|\|_0\le 2\epsilon|k|\|v\|_{C_tC_x^1}\le \frac{|k|}{4}\implies \|\nabla(k\cdot  \xi_m(t))\|_0 \in \left[\frac34|k|,\frac54|k|\right]
\end{equation}
Therefore,  \autoref{lemma:antidiv} applied to $T_{m_1}$, and noticing that $\mathfrak{m}_1(\lambda\nabla(k\cdot\xi_m)) \equiv 1$, $\nabla^h \mathfrak{m}_1(\lambda\nabla(k\cdot\xi_m)) \equiv 0$ for every $h \geq 1$, giving:
\begin{align}
    T_{\mathfrak{m}_1}(a_{m,k}e^{i\lambda k\cdot \xi_m})
    &-\sum_{3\lambda/8\le 2^s\le 2^{j_k}}
    \mathcal{\epsilon}_{n_0,s}(k \cdot\xi_m,a_{m,k}) 
    e^{i\lambda k \cdot\xi_m}
    \\
    &=
    a_{m,k}\,\underbrace{\mathfrak{m}_1(\lambda\nabla(k\cdot\xi_m))}_{\equiv1}
    e^{i\lambda k \cdot\xi_m}
    +
    \sum_{h=1}^{2n_0}C_h^{\lambda}(k\cdot\xi_m,a_{m,k}):(\underbrace{\nabla^h \mathfrak{m}_1}_{\equiv0})(\lambda\nabla(k\cdot\xi_m))
    e^{i\lambda k \cdot\xi_m}
    \\
    &=
    a_{m,k}     e^{i\lambda k \cdot\xi_m}.
\end{align}
At last, \eqref{e:decF} follows from applying \autoref{lemma:antidiv} to $T_{\mathfrak{m}_j}$ for all $j>j_k$. Indeed, 
\[T_{\mathfrak{m}_j}(a_{m,k}e^{i\lambda k\cdot \xi_m})=\epsilon_{n_0,j}(k\cdot \xi_m,a_{m.k})e^{i\lambda k\cdot \xi_m}\]
thanks to $(\nabla^h\mathfrak{m}_j)(\lambda \nabla(k\cdot\xi_m))=0$ for all $h\in \N$.

 Now, for fixed $(m,k)$, we want to estimate $\|\epsilon_{n_0}^{\lambda}(k\cdot \xi_m,a_{m,k})\|_0$. Based on \eqref{e:epsN0L}, we have to consider $\epsilon_{n_0,j}(k\cdot \xi_m,a_{m.k})$, which is given by \eqref{e:epsn0}.

Let $Z$ be defined by \eqref{e:Z} and $Z_0(r):=r^{-1}Z(r)$, then
\begin{equation}\label{e_Zn}
    Z^{(n)}=nZ_0^{(n-1)}+Z_0^{(n)}, \quad Z^{(n)}(r):=\int_0^1(1-s)(-s)^n(y\cdot \nabla)^{n+2}(k\cdot \xi_m)(x-rsy)\,ds.
\end{equation}
Therefore,
Assumption \ref{ass:SupportTime} and \autoref{prop:transport} imply that
\begin{align}\label{est:Z}
    \|Z^{(n)}\|_0\lesssim |y|^{n+1}\|k\cdot \xi_m\|_{n+1}+|y|^{n+2}\|k\cdot \xi_m\|_{n+2}\lesssim |y|^{n+1}l^{-n}|k|[1+|y|l^{-1}].
\end{align}
Then, by means of \eqref{e:Bn}, \eqref{e:Bnk}, we have that
    \begin{align} \label{est:B}
        \|\beta_{n_2}[k\cdot \xi_m](r)\|_0&\lesssim\sum_{m=1}^{n_2}\lambda^m|y|^{m+n_2}|k|^m[1+|y|l^{-1}]^ml^{-n_2}\\
        &\lesssim(|y|l^{-1})^{n_2}|k|^{n_0+1}\sum_{m=1}^{n_2}(\lambda |y|)^m\sum_{i=0}^m(|y|l^{-1})^{i}.
    \end{align}

Instead, Assumption \ref{ass:estimateAmk} implies that
\begin{align}
    &\label{est:aPart}\|(y\cdot \nabla)^{n_1}a_{m,k}(x-ry)\|_0\lesssim |y|^{n_1}|\mathring{a}_k|s_{n_1}.
\end{align}
Then, \eqref{est:B}, \eqref{est:aPart} imply that
    \begin{align}
        \|\epsilon_{n_0,j}\|_0&\lesssim \sum_{n_1,n_2}s_{n_1}|k|^{n_0+1}|\mathring{a}_k|\sum_{m=1}^{n_2}\lambda^m\sum_{i=0}^ml^{-n_2-i}\int |\overset{\vee}{\mathfrak{m}}_{j}(y)||y|^{n_0+1+m+i}dy\\
        &\lesssim \sum_{n_1,n_2}s_{n_1}|k|^{n_0+1}|\mathring{a}_k|\sum_{m=1}^{n_2}\lambda^m\sum_{i=0}^ml^{-n_2-i}2^{-j(n_0+1+m+i)},
    \end{align}
where the last inequality is due to $\|\overset{\vee}{\mathfrak{m}}_{j}(y)|y|^n\|_{L^1}\lesssim2^{-jn}$.
Therefore, we have that 
\begin{align}\label{est:epsN0L}
    \|\epsilon_{n_0}^{\lambda}\|_0
    &\lesssim
    \sum_{n_1,n_2}s_{n_1}|k|^{n_0+1}|\mathring{a}_k|\sum_{m=1}^{n_2}\lambda^m\sum_{i=0}^ml^{-n_2-i}\sum_{2^j\gtrsim\lambda}2^{-j(n_0+1+m+i)}
    \\ \nonumber
    &\lesssim |k|^{n_0+1}|\mathring{a}_k|s_{n_0+1}\lambda^{-n_0-1},
\end{align}
where the last inequality makes use of Assumption \ref{ass:s_n}.
The assumption on the time supports, i.e. condition \ref{ass:SupportTime}, implies that 
\begin{equation}\label{e:epsN0LSum}
    \left\|\sum_{m,k}\epsilon_{n_0}^{\lambda}(k\cdot \xi_m,a_{m,k})\right\|_0
    \lesssim
    \sum_{k}|k|^{n_0+1}|\mathring{a}_k|s_{n_0+1}\lambda^{-n_0-1}\le a_Fs_{n_0+1}\lambda^{-n_0-1}.
\end{equation}
So, \eqref{e:decF} allows to conclude as follows:
    \begin{align}
        \|\mathcal{R}F\|_N&\le \|\mathcal{R}\proj_{\gtrsim3\lambda/8}F\|_N+\|\mathcal{R}\sum_{m,k}\epsilon^{\lambda}_{n_0}(k\cdot\xi_m,a_{m,k})\|_N\\
        &\lesssim \frac{\|F\|_N}{\lambda}+\lambda^N\|\sum_{m,k}\epsilon_{n_0}^{\lambda}(k\cdot \xi_m,a_{m,k})\|_0\\
        &\lesssim a_F[s_0\,\lambda^{N-1}+s_{n_0+1}\lambda^{-n_0-1} \lambda^N],
    \end{align}
where the second inequality is due to the Bernstein inequality, while the third one is due to
\begin{equation}\label{est:Fnorm}
\|F\|_N\lesssim\sum_k\sum_{N_1+N_2\le N}|\mathring{a}_k| s_{N_1}|k|^{N_2}\lambda^{N_2}\le a_F\, s_0\,\lambda^N,
\end{equation}
and the last inequality holds because of Assumption \ref{ass:s_n}.
\end{proof}

\begin{prop}\label{prop:antidiv2}Under the same assumptions of \autoref{prop:antidiv1}, let $l>\frac{100}{\lambda}$ and $w\in C^{\infty}(\T^3,\R^3)$ such that there exists a non-decreasing sequence $(b_n)_{n \in \N}$ satisfying for every $n,m \in \N$
\begin{align}
  b_n\lambda^m\le b_0\lambda^{m+n}  ,
  \quad
  \|w\|_n\le b_n.
\end{align}

Let us moreover suppose there exists $c>0$ such that
\begin{equation}\label{ass:estamk2}
    c\|D_{t,l}a_{m,k}(t)\|_j \le s_j\mathring{a}_k,
\end{equation}
where $D_{t,l}:=\partial_t+v_l\cdot\nabla$ and $j\le {\overline{n}}$, and that for any $n \in \N$
\begin{equation}\label{ass:controlV}
    \|v\|_{n+1}\le l^{-n}\|v\|_1.
\end{equation}
Then for every $N\le {\overline{n}-2}$ and $h_0\in \N\setminus\{0,1\}$ it holds
     \begin{align}\label{e:Antidiv2}
         \|D_t\mathcal{R}F(t)\|_N
         &\lesssim \Big[s_0(c\lambda)^{-1}+c^{-1}s_{n_0+1}\lambda^{-n_0-1}+s_0\lambda^{-1}\|v\|_1+\|v\|_0\lambda^{-n_0} s_{n_0+1}\\
         &\quad+(1+\lambda^{-n_0} s_{n_0+1})(b_0+l\|v\|_1)+s_0(l\lambda)^{-h_0}\lambda^2\|v\|_1\Big]\lambda^Na_F,
\end{align}
where $D_t:=\partial_t+(w+v)\cdot\nabla$.
\end{prop}

\begin{proof}
By standard computations, we have that
\begin{equation}
    D_t\mathcal{R}F=\mathcal{R}(D_{t,l}F)+[v_l\cdot \nabla, \mathcal{R}]F+[(w+(v-v_l))\cdot \nabla]\mathcal{R}F.
\end{equation}
Since $D_{t,l}F=\sum_{k,m}D_{t,l}a_{m,k}e^{i\lambda\xi_m \cdot k}$, then the previous \autoref{prop:antidiv1} implies that
\begin{equation}\label{est:DRF1}
    \|\mathcal{R}(D_{t,l}F)\|_N\lesssim a_F\lambda ^N[s_0(c\lambda)^{-1}+s_{n_0+1}\lambda^{-n_0-1}c^{-1}].
\end{equation}
Instead,
\begin{align}\label{est:DRF2}
    \|[(w+(v-v_l))\cdot \nabla]\mathcal{R}F\|_N&\lesssim \sum_{N_1+N_2\le N}\|w+(v-v_l)\|_{N_1}\|RF\|_{N_2+1}\\
    &\lesssim \sum_{N_1+N_2\le N} (b_{N_1}+\|v\|_1 l^{1-N_1})\lambda^{N_2}(1+s_{n_0+1}\lambda^{-n_0})\\
    &\lesssim (b_0+l\|v\|_1)(1+s_{n_0+1}\lambda^{-n_0})\lambda^N,
\end{align}
where the second inequality is due to the previous result, Bernstein inequality, \eqref{ass:controlV} and the assumptions on $w$, while the third is due to $b_n\lambda^m\le b_0\lambda^{m+n}$.

At last, we need to estimate
\begin{equation}
    [v_l\cdot\nabla,\mathcal{R}]F
    =
    [v_l\cdot\nabla,\mathcal{R}]\proj_{\gtrsim \lambda}F
    -
    [v_l\cdot \nabla,\mathcal{R}]\sum_{m,k}\epsilon^{\lambda}_{n_0}(k\cdot \xi_m,a_{m,k})e^{i\lambda k\cdot \xi_m}=:I_1+I_2.
\end{equation}
Since $I_2$ has localized active frequencies, the Bernstein inequality implies that
\begin{equation}\label{est:DRF3}
\begin{aligned}
    \|I_2\|_N\lesssim \lambda^N\|I_2\|_0&\lesssim\lambda^N\|v_l\|_0\|\nabla \proj_{\lesssim\lambda}F\|_0
    \\
    &\lesssim
    \lambda^{N+1} \|v\|_0 \left\|\sum_{m,k}\epsilon^{\lambda}_{n_0}(k\cdot \xi_m,a_{m,k})\right\|_0\\
    &\lesssim \lambda^N\|v\|_0s_{n_0+1}\lambda^{-n_0}a_F,
\end{aligned}
\end{equation}
where the last inequality is due to \eqref{e:epsN0LSum}.
To study $I_1$, we preliminarily observe that by \cite[Proposition 4.12]{RT25} the operator $\mathcal{R}$ can be seen as a Fourier multiplier with coefficients $\mathscr{R}(k)$, $k \in \mathbb{Z}^3 \setminus \{0\}$, and since $I_1$ has zero space average we can rewrite
\begin{align}
    -I_1
        &=
        \sum_{\substack{k,\eta \in \Z^3 \setminus \{0\}}}
        i\eta e^{ik\cdot x}\mathscr{F}[v_l](k-\eta)\mathscr{F}[\proj_{\gtrsim \lambda}F](\eta)
        \left( \mathscr{R}(k)-\mathscr{R}(\eta) \right).
\end{align}
We can extend the function $\mathscr{R}$ to a smooth function defined on $\R^3$, still denoted by $\mathscr{R}$ with a slight abuse of notation, that coincides with the expression given in \cite[Equation 4.17]{RT25} for every $k \in \R^3 \setminus B(0,1/2)$. 
Since $l^{-1} < \frac{\lambda}{100}$ by assumption and by construction $\mathscr{F}[\proj_{\gtrsim \lambda}F](\eta) = 0$ for $\eta \lesssim \lambda$, $\mathscr{F}[v_l](k-\eta) = 0$ for $|k-\eta| > 2l^{-1}$, for every $k,\eta \in \Z^3 \setminus \{0\}$ contributing to the sum above it holds $|\eta-\sigma(\eta-k)|\geq \frac{|\eta|}2 \gtrsim \lambda$ for every $\sigma \in [0,1]$. As a consequence, we can apply integration by parts and the mean value theorem to obtain:
\begin{equation}
    \begin{aligned}
        -I_1
        &=
        \sum_{\substack{k,\eta \in \Z^3\setminus \{0\}}}
        i\eta e^{ik\cdot x}\mathscr{F}[v_l](k-\eta)\mathscr{F}[\proj_{\gtrsim \lambda}F](\eta)
        \left( \mathscr{R}(k)-\mathscr{R}(\eta) \right)
        \\
        &=
        \sum_{\substack{k,\eta \in \Z^3\setminus \{0\},\\|\eta|\gtrsim \lambda}}
        i\eta e^{ik\cdot x} \mathscr{F}[v_l](k-\eta)\mathscr{F}[F](\eta)
        \Bigg( \sum_{h=1}^{h_0}\frac{1}{h!}[(k-\eta)\cdot \nabla]^h\mathscr{R}(\eta)\\
        &+\frac{1}{h_0!}\int_0^1[(k-\eta)\cdot  \nabla  ]^{h_0+1}\mathscr{R}(\eta+\sigma(k-\eta))(1-\sigma)^{h_0}d\sigma \Bigg)=:J_1+J_2.
    \end{aligned}
\end{equation}
Let us observe that by \cite[Equation 4.17]{RT25} we have $|\nabla^h \mathscr{R}(\eta)| \sim |\eta|^{-1-h}$.   
Therefore
\begin{equation}\label{est:DRF4}
\begin{aligned}
    \|J_1\|_N
    &\lesssim
    \sum_{h=1}^{h_0}\sum_{N_1+N_2\le N}l^{1-h-N_1}\|v\|_1\|F\|_{N_2+1}\lambda^{-h-1}\\
    &\lesssim s_0\|v\|_1\lambda^{-2}\sum_{N_1+N_2\le N}l^{-N_1}\lambda^{N_2+1}a_F\lesssim a_F s_0\|v\|_1\lambda^{N-1},
\end{aligned}
\end{equation}
where the second-to-last inequality is due to \eqref{est:Fnorm}.
At last, 
\begin{equation}\label{est:DRF5}
    \begin{aligned}
        \|J_2\|_N
        &\lesssim
        \sum_{\substack{k,\eta \in \Z^3,\\|\eta|\gtrsim \lambda}}
        |k|^N|\eta||\mathscr{F}[v_l](k-\eta)||\mathscr{F}[F](\eta)||k-\eta|^{h_0+1}|\eta-\sigma(\eta-k)|^{-h_0-2}
        \\
        &\lesssim
        \sum_{\substack{k,\eta \in \Z^3,\\|\eta|\gtrsim \lambda}} 
        |\eta|^{N-h_0-1}|\mathscr{F}[F](\eta)||k-\eta|^{h_0}|\mathscr{F}[\nabla v_l](k-\eta)|
        \\
         &\lesssim
         \frac{l^{-h_0}}{\lambda^{h_0+1}} \|\nabla^N \proj_{\gtrsim\lambda}F \|_{L^2}\|\nabla v\|_{L^2} 
         \lesssim 
         (l\lambda)^{-h_0}\lambda^2\|v\|_1s_0\,a_F\lambda^N.
    \end{aligned}
\end{equation}
where the second inequality is due to $|\eta-\sigma(\eta-k)|\geq \frac{|\eta|}2$ and $|k|\lesssim |\eta|$.
All in all, \eqref{est:DRF1}, \eqref{est:DRF2}, \eqref{est:DRF3}, \eqref{est:DRF4}, \eqref{est:DRF5} prove the claim.
\end{proof}
\begin{oss}\label{oss:postAntidiv}
    Notice that 
    \begin{equation}
     \sum_{h=1}^{h_0}\|\nabla^hv_l\nabla^h\mathcal{R}\proj_{\gtrsim\lambda}F\|_N\lesssim \sum_{h=1}^{h_0}\sum_{N_1+N_2\le N}l^{1-h}\|\nabla v_l\|_{N_1}\lambda^{-h}\|F\|_{N_2}.   
    \end{equation}
    This inequality and \eqref{est:DRF5} imply that
    \begin{equation}
        \|[v_l\cdot\nabla,\mathcal{R}]\proj_{\gtrsim \lambda}F\|_N\lesssim (l+(l\lambda)^{-h_0}\lambda^2)\sum_{N_1+N_2\le N}\|\nabla v_l\|_{N_1}\|F\|_{N_2}.
    \end{equation}
\end{oss}
\begin{prop}
    Under the same assumption of  \autoref{prop:antidiv2}, it holds that
    \begin{align}
        \label{e:Antidiv3}
        \left\|D_{t,l}\sum_{m,k}\epsilon^{\lambda}_{n_0}(k\cdot \xi_m,a_{m,k})e^{i\lambda k\cdot \xi_m}\right\|_0
        \lesssim 
        a_F(c^{-1}\vee \|v_l\|_1)s_{n_0+1}\lambda^{-n_0-1},
    \end{align}
\end{prop}
\begin{proof}
     For fixed $(m,k)$, we want to estimate $\|D_{t,l}\epsilon_{n_0}^{\lambda}\|_0$, given by \eqref{e:epsN0L}. we have to consider $\epsilon_{n_0,j}(k\cdot \xi_m,a_{m.k})$, which is given by \eqref{e:epsn0}.
Based on \eqref{e_Zn} and $Z_0(r):=r^{-1}Z(r)$, we have that
\begin{equation}
\begin{aligned}
    \|D_{t,l}Z^{(n)}_0\|_0
    &\lesssim 
    \left\|\int_0^1(1-s)(-s)^n[v_l(x)-v_l(x-yrs)]\cdot \nabla_x[(y\cdot \nabla)^{n+2}(k\cdot\xi_m)(x-rsy)]ds\right\|_0 
    \\
    &+
    \left\|\int_0^1(1-s)(-s)^n[D_{t,l}(y\cdot \nabla)^{n+2}]k\cdot\xi_m(x-rsy)ds\right\|_0
    \\
    &\lesssim
    \|v_l\|_1|y|^{n+2}l^{-n-1}|k|[1+l^{-1}|y|],
\end{aligned}
\end{equation}
where the last inequality is due to the following one, which is proved in \cite{DLK23},
\begin{equation}\label{e:techical1}
    \|D_{t,l}(y\cdot\nabla)^{n}g\|_0
    \lesssim 
    |y|^{n}
    \left(\|D_{t,l}g\|_n+\sum_{\substack{m_1+m_2=n,\\m_1\geq 1}} \|v_l\|_{m_1}\|g\|_{m_2+1}\right).
\end{equation}
Then, similarly to \eqref{est:Z}, \eqref{est:B}, we have that
\begin{align}
    \label{est:DZ} \|D_{t,l}Z^{(n)}\|_0
    &\lesssim 
    |y|^{n+1}l^{-n}|k|\|v_l\|_1[1+|y|l^{-1}+|y|^2l^{-2}],\\
    \label{est:DB} \|D_{t,l}\beta_{n_2}[k\cdot \xi_m]\|_0
    &\lesssim
    \|v_l\|_1[1+|y|^2l^{-2}](|y|l^{-1})^{n_2}|k|^{n_0+1}\sum_{m=1}^{n_2}(\lambda |y|)^m\sum_{i=0}^m(|y|l^{-1})^{i}.
\end{align}
Now, \eqref{e:MatDerFlux},\eqref{prop:transport} imply that
\begin{align}
    \label{est:DexpGradXi} \|D_{t,l}(e^{-i\lambda y\cdot \nabla (k\cdot \xi_m)})\|_0&\lesssim \lambda |y|\|v_l\|_1|k|,\\
    \label{est:DexpZ}
        \|D_{t,l}e^{iZ(k\cdot \xi_m)}\|_0
        &\lesssim
        \lambda \Bigg\|\int_0^1(1-s)(y\otimes y):[(D_{t,l}\nabla^2(k\cdot \xi_m)(x-ry)\\
        &\quad+\left((v_l(x)-v_l(x-ry))\cdot\nabla\right)\nabla^2(k\cdot \xi_m)(x-ry)ds]e^{i\lambda Z(r)}\Bigg\|_0\\
        &\lesssim\lambda |y|^2|k|l^{-1}\|v_l\|_1[1+|y|l^{-1}].
\end{align}
Instead, Assumption \ref{ass:estimateAmk} and \eqref{e:techical1} imply that
\begin{align}
    &\label{est:DaPart}\begin{aligned}
        \|D_{t,l}[(y\cdot \nabla)^{n_1}a_{m,k}(x-ry)]\|_0&\lesssim\|(v_l(x)-v_l(x-yr))\cdot \nabla [(y\cdot \nabla)^{n_1}a_{m,k}(x-ry)]\|_0\\
        &\quad+\|(D_{t,l}(y\cdot \nabla)^{n_1}a_{m,k}(x-ry)\|_0\\
        &\lesssim |y|^{n_1+1}\|v_l\|_1s_{n_1+1}|\mathring{a}_k|+|y|^{n_1}|\mathring{a}_k|s_{n_1}(c^{-1}\vee \|v_l\|_1).
    \end{aligned}
\end{align}
By following the computations used to get \eqref{est:epsN0L}, \eqref{est:DB}, \eqref{est:DexpZ}, \eqref{est:DexpGradXi}, \eqref{est:DaPart} imply that 
\begin{equation}\label{est:DepsN0L}
    \|D_{t,l}\epsilon^{\lambda}_{n_0}\|_0\lesssim |k|^{n_0+2}\mathring{a}_k(c^{-1}\vee \|v_l\|_1)s_{n_0+1}\lambda^{-n_0-1}.
\end{equation}
Then, the claim follows by just summing up in $m$ and $k$.
\end{proof}

\bibliography{biblio}{}
\bibliographystyle{alpha}
\end{document}